\documentclass[12pt,a4paper,notitlepage]{article}
\usepackage{amsthm}
\usepackage{amsfonts}
\usepackage{amsmath}
\usepackage{mathtools}

\usepackage{amssymb}

\usepackage{graphicx}

\usepackage{pdfpages}

\usepackage{cite}
\usepackage{indentfirst}
\usepackage{color}
\usepackage{hyperref}
\hypersetup{
   colorlinks,
   citecolor=black,
   filecolor=black,
   linkcolor=black,
   urlcolor=black
}

\usepackage{blindtext}
\usepackage[utf8]{inputenc}

\newtheorem{thm}{Theorem}[section]
\newtheorem*{thm*}{Theorem}
\newtheorem{lem}[thm]{Lemma}
\newtheorem{prop}[thm]{Proposition}
\newtheorem{cor}[thm]{Corollary}

\theoremstyle{definition}
\newtheorem{defn}[thm]{Definition}

\newtheorem{exmp}[thm]{Example}

\newtheorem{rmk}[thm]{Remark}

\newcommand{\RN}[1]{%
 \textup{\uppercase\expandafter{\romannumeral#1}}%
}

\newcommand{\Addresses}{{
 \bigskip
 \footnotesize
 \textsc{Instituto de Matem\'{a}ticas, Universidad Nacional Aut\'{o}noma de M\'{e}xico,
   Circuito exterior, Ciudad Universitaria, 04510, M\'{e}xico City, M\'{e}xico}\par\nopagebreak
 \textit{E-mail address}: \texttt{jon.wilson@im.unam.mx} }}

 \title{Positivity for quasi-cluster algebras}
\author{Jon Wilson}
\date{}

\usepackage{fancyhdr}

\fancypagestyle{plain}{
 \fancyhf{}
 \fancyfoot[C]{\thepage}%
}
\fancypagestyle{firstpage}{
 \fancyhf{}
 \fancyfoot[L]{\footnotesize{This work was supported by the author's postdoctoral fellowship at IMUNAM, Prof. Christof Geiss' CONACyT-239255 grant, and Prof. Daniel Labardini-Fragoso's grants: CONACyT-238754 and C\'{a}tedra Marcos Moshinsky.}}%
}

\usepackage[titles]{tocloft}

\usepackage{cite}

\usepackage{xcolor}
\newcommand*\circled[1]{\kern-2.5em%
  \put(0,4){\color{white}\circle*{18}}\put(-0.5,4){\circle{12}}%
  \put(-3,0){\color{black}\small#1}~~}
\usepackage{enumitem}

\usepackage{float}

\usepackage{stackengine}

\begin{document}

\maketitle

\thispagestyle{firstpage}

\begin{abstract}

We generalise the expansion formulae of Musiker, Schiffler and Williams \cite{musiker2011positivity}, obtained for cluster algebras from orientable surfaces, to a larger class of coefficients which we call \textit{principal laminations}. In doing so, for any quasi-cluster algebra from a non-orientable surface, we are able to obtain expansion formulae for each cluster variable with respect to any initial quasi-triangulation $T$, and any choice of principal lamination. Moreover, generalising the `separation of additions' formula of Fomin and Zelevinsky \cite{fomin2007cluster}, we settle a conjecture of Lam and Pylyavskyy \cite{lam2012laurent} in the setting of quasi-cluster algebras. Namely, we prove the \textit{positivity conjecture} for quasi-cluster algebras with respect to any choice of coefficients.

\end{abstract}

\pagenumbering{arabic}

\tableofcontents

\section{Introduction}

The underlying framework of the paper revolves around the notion of a \textit{cluster structure}. This is a ring whose generators, \textit{cluster variables}, are grouped into overlapping subsets of the same cardinality, called \textit{clusters}. Often the complete set of generators is not known from the outset; rather an initial cluster is provided, which is also equipped with an iterative rule, \textit{mutation}, describing how more cluster variables/clusters are formed.

$$ \mathbf{x} = \{x_1,\ldots, x_n\} \xrightarrow{\text{mutation}} \mathbf{x'} = \{x_1',\ldots, x_n'\}$$

This act of mutation on a cluster only changes one cluster variable $x_k$ for some, $k \in \{1,\ldots, n\}$. Specifically, $x_i = x_i'$ for $i \neq k$ and $$x_k' = \frac{L_k(x_1,\ldots, x_n)}{x_k}$$ where $L_k(x_1,\ldots, x_n)$ is a Laurent polynomial in the cluster variables $x_1,\ldots, x_n$. One should note that $L_k$ is not fixed; mutations of different clusters will in general involve different Laurent polynomials.

This idea was first introduced by Fomin and Zelevinsky in the context of \textit{cluster algebras} \cite{fomin2002cluster}. There the data describing how clusters mutate comes from skew-symmetrizable matrices $B= (b_{ij})$. Namely, $L_k$ is the binomial associated to the $k^{th}$ column of $B$. $$x_k' = \frac{\displaystyle \prod_{b_{ik}>0} x_i^{b_{ik}} + \prod_{b_{ik}<0} x_i^{-b_{ik}}}{x_k}$$ \indent The original motivation was to help probe the study of total positivity and the dual canonical bases of reductive groups, initiated by Lusztig \cite{lusztig1994total}. Although this remains a very important part of the theory, the true growth of the subject is really owed to its unexpected ubiquity. Indeed, cluster structures are now known to intertwine and propagate through mirror symmetry, integrable systems, Poisson geometry, quiver representations, dilogarithm identities, Teichm\"{u}ller theory, and the list goes on. The reward of finding a cluster structure is that it allows one to study a global `picture' by local means; the iterative nature of mutation allows each cluster variable to be expressed as a rational function in the initial cluster variables. The foundational miracle of cluster algebras is the \textit{Laurent phenomenon} which states that each cluster variable is in fact a Laurent polynomial in the initial cluster variables \cite{fomin2002cluster}. Inspired by cluster algebras' roots in 'total positivity' the so called \textit{positivity conjecture} was postulated by Fomin and Zelevinsky, speculating that this Laurent polynomial has non-negative coefficients. Numerous papers have been devoted to special cases of the conjecture, which has now been proven by Lee and Schiffler for skew-symmetric type \cite{lee2015positivity}, and by Gross, Hacking, Keel and Kontsevich in the full generality of skew-symmetrizable type \cite{gross2018canonical}. \newline 
\indent  In recent years a whole wealth of cluster structures have emerged which lie outside the realm of cluster algebras. Surprisingly, almost all of these fall (or are at least believed to fall) under the umbrella of Lam and Pylyavskyy's \textit{Laurent phenomenon (LP) algebras}; this is a very broad cluster structure framework that was specifically designed to produce the Laurent phenomenon \cite{lam2012laurent}. Examples include the cluster structures of Gross, Hacking and Keel via mutations of toric models for certain log CY varieties \cite{gross2013birational}; the generalised cluster algebras of Chekhov and Shapiro \cite{chekhov2013teichmuller}, the electrical networks arising from the so called \textit{electrical Lie groups} of Lam and Pylyavskyy \cite{lam2015electrical}; and the positivity tests of \textit{response matrices} for circular planar electrical networks \cite{alman2016laurent}. \newline \indent A large open conjecture of Lam and Pylyavskyy essentially states that the positivity conjecture (of cluster algebras) also holds for LP algebras. There is some subtly to the statement since LP algebras can have exchange polynomials with negative coefficients, which would clearly prevent the positivity property in general. However, outside of this case the positivity conjecture is believed to hold. For a special choice of initial seed this was proven for \textit{graph LP algebras} -- a class of LP algebras which possess only finitely many cluster variables.\newline
\indent In the setting of cluster algebras, a particularly profitable testing ground for conjectures comes from the study of triangulated orientable surfaces. Given an orientable marked surface we may \textit{triangulate} it. For each (tagged) triangulation $T$ of the surface we may assign a \textit{seed} consisting of a cluster and a skew-symmetric matrix $B$; here the cluster variables correspond to (tagged) arcs in $T$, and $B$ is determined by inscribing oriented cycles in each triangle (with respect to the surface's orientation). These seeds form a cluster algebra structure where mutations correspond to flipping arcs in triangulations \cite{fomin2008cluster}. One can explain this behaviour by endowing the surface with any decorated hyperbolic structure. The correspondence is then uncovered by recognising that cluster variables can alternatively be viewed as the \textit{lambda length} of their corresponding arc, and the matrices encode how these lengths are related. Moreover, Fomin and Thurston showed this framework could also encode arbitrary coefficients systems via collections of curves known as \textit{laminations} \cite{fomin2012cluster}.\newline
\indent Dupont and Palesi studied the analogous cluster structures arising from unpunctured non-orientable surfaces \cite{dupont2015quasi}, which we subsequently extended to punctured surfaces \cite{wilson2018laurent}. Here the notion of `triangulation' is called a \textit{(tagged) quasi-triangulation} which is comprised of \textit{(tagged) quasi-arcs}; a class of curves consisting of one-sided closed curves as well as the usual (tagged) arcs of Fomin, Shapiro and Thurston. For each quasi-arc $\alpha$ one obtains a cluster variable $x_{\alpha}$ by considering its associated lambda length $\lambda(\alpha)$. There is a notion of flip for each quasi-arc in a quasi-triangulation, so a cluster structure is obtained by calculating the relationship between the lambda lengths of flipped quasi-arcs. The resulting structure is known as the \textit{quasi-cluster algebra}. Unlike cluster algebras, not all exchange relations in a quasi-cluster algebra are binomial.  By adding laminations to the surface, we also extended this construction to form the concept of a \textit{quasi-cluster algebra with coefficients}. We showed in \cite{wilson2017laurent},\cite{wilson2018laurent} that the resulting cluster structures all fall into the framework of LP algebras. An immediate consequence being that the Laurent phenomenon also holds in the context of quasi-cluster algebras. In this paper we significantly strengthen that result by proving Lam and Pylyavskyy's positivity conjecture for all quasi-cluster algebras (with coefficients): \newline

\noindent \textbf{Main Theorem}. \textit{Let $(S,M)$ be a non-orientable marked surface and let $\mathbf{L}$ be a multi-lamination. Fixing a quasi-triangulation $T$ we may consider the initial seed $\Sigma_T := \{\mathbf{x}, T\}$ and the associated quasi-cluster algebra $\mathcal{A}_{\mathbf{L}}(S,M)$. Then each cluster variable of $\mathcal{A}_{\mathbf{L}}(S,M)$ is a Laurent polynomial, in $\mathbf{x}$, with non-negative coefficients.} \newline

Recall that each cluster parameterizes the decorated Teichm\"{u}ller space of the associated bordered surface $(S,M)$. The affirmed conjecture therefore has particularly deep consequences -- it tells us the change of coordinate functions between any two of these parameterizations are actually Laurent polynomials with non-negative coefficients. \newline \indent 

The paper is organised as follows. In Section 2 we first recall the construction of a quasi-cluster algebra, $\mathcal{A}(S,M)$, from a bordered surface $(S,M)$, and proceed in defining the \textit{laminated quasi-cluster algebra}, $\mathcal{A}_{\mathbf{L}}(S,M)$, with respect to a multi-lamination $\mathbf{L}$ on $(S,M)$. In view of Remark \ref{coefficientbijectionrmk}, one should view these multi-laminations as all possible coefficient systems of the underlying quasi-cluster algebra $\mathcal{A}(S,M)$. We finish the section by discussing the combinatorics behind cluster algebras from orientable surfaces, and the inherited combinatorics on the orientable double cover of non-orientable surfaces. Sections 3 is devoted to introducing abstract snake and band graphs. \newline \indent The overall goal of the paper is to obtain explicit expansion formulae for each cluster variable $x_{\alpha}$ in $\mathcal{A}_{\mathbf{L}}(S,M)$, with respect to an initial tagged quasi-triangulation $T$. In section 4, in the interest of comprehensibility, we restrict ourselves to tagged quasi-triangulations of $(S,M)$ which do not contain notched arcs nor one-sided closed curves -- we refer to them as \textit{triangulations}. Our approach imitates the work of Musiker, Schiffler and Williams \cite{musiker2011positivity},\cite{musiker2013bases} -- fixing such a triangulation, $T$, we associate a snake or band graph $S_{\alpha,T}$ to each quasi-arc $\alpha$ of $(S,M)$ (with respect to $T$). Section 5 is dedicated to finding expansion formulae, with respect to a triangulation $T$, for cluster variables corresponding to quasi-arcs. We begin by obtaining such formulae for the coefficient-free quasi-cluster algebra, $\mathcal{A}(S,M)$. Specifically, for each quasi-arc $\alpha$, we show that the corresponding cluster variable, $x_{\alpha}$, has the following expansion:

$$x_{\alpha} = \frac{1}{cross(T,\alpha)}\sum_{P \in \mathcal{P}_{T,\alpha}} x(P).$$

\noindent When $\alpha$ is an arc the above expansion formula follows by lifting to the orientable double cover and invoking the work of Musiker, Schiffler and Williams \cite{musiker2011positivity}. The real difficulty is to obtain the expansion formula when $\alpha$ is a one-sided closed curve. To achieve this we utilise the fact there exists arcs $\beta$ and $\gamma$ such that $x_{\gamma} = x_{\alpha}x_{\beta}$. We are thus able to derive an expansion formulae for $x_{\alpha}$ by embedding $S_{\beta,T}$ in $S_{\gamma,T}$, and studying the compliment of these various embeddings. \newline \indent
In the later half of Section 5 we embark on the mission of obtaining expansion formulae for quasi-cluster algebras with a particular class of coefficients. This class comes from a type of multi-lamination called \textit{principal laminations}, which may be thought of as a generalisation of the principal coefficients of cluster algebras. Indeed, when $(S,M)$ is orientable then $\mathcal{A}(S,M)$ is a cluster algebra, and principal coefficients of $\mathcal{A}(S,M)$ correspond to a particular principal lamination. In the coefficient free case we were able to obtain expansion formulae `for free' by lifting the picture to the orientable double cover $\overline{(S,M)}$. However, the situation is no longer this simple. The explicit expansion formulae of Musiker, Schiffler and Williams were only obtained for cluster algebras with principal coefficients. To employ their result we would need a multi-lamination on $(S,M)$ which lifts to principal coefficients on the orientable double cover -- no such multi-lamination exists. One of the motivating reasons behind principal laminations on $(S,M)$ comes from the fact their lifts are (usually) principal laminations on $\overline{(S,M)}$. Using the \textit{separation of additions} formulae of Fomin and Zelevinsky, in Subsection 5.2 we concentrate on orientable surfaces and extend the results of Musiker, Schiffler and Williams to all principal laminations -- which should be an interesting result for cluster algebraists in its own right.  
$$x_{\mathbf{L_T}}(\alpha) = \frac{1}{cross(\alpha, T)} {\displaystyle \sum_{P \in \mathcal{P}_{T,\alpha}} x(P)y_{\mathbf{L_T}}(P)}.$$ Here the $y$-monomial, $y_{\mathbf{L_T}}(P)$, comes from the notion of $\mathbf{L_T}$-oriented diagonals of $S_{\alpha,T}$ -- a generalisation of the $\alpha$-oriented diagonals of Musiker and Schiffler \cite{musiker2010cluster}.

As a corollary, if $\mathbf{L_T}$ is a principal lamination which contains no one-sided closed curves, we immediately obtain expansion formulae for every arc in the corresponding quasi-cluster algebra $\mathcal{A}_{\mathbf{L_T}}(S,M)$. If $\mathbf{L_T}$ does contain one-sided closed curves then its lift is no longer a principal lamination on $\overline{(S,M)}$. In which case, we pass to a closely related \textit{quasi-principal lamination} $\mathbf{L_T}^*$ on $(S,M)$. We obtain an expansion formulae for $x_{\mathbf{L_T}}(\alpha) \in \mathcal{A}_{\mathbf{L_T}}(S,M)$ by comparing it to $x_{\mathbf{L_T}^*}(\alpha) \in \mathcal{A}_{\mathbf{L_T}^*}(S,M)$. Namely, we show: 

$$x_{\mathbf{L_T}}(\alpha) = \frac{x_{\mathbf{L_T}^*}(\alpha)}{bad(\mathbf{L_T},\alpha)} = \frac{1}{bad(\mathbf{L_T},\alpha)cross(\alpha, T)} {\displaystyle \sum_{P} x(P)y_{\mathbf{L_T^*}}(P)}.$$ \newline \indent
In Section 6 we obtain expansion formulae for any tagged quasi-arc $\alpha$ of $(S,M)$, with respect to any initial tagged triangulation $T$. We begin the section by explaining how the snake and band graphs, $S_{\alpha,T}$, are constructed in this setting. We also define the $y$-monomial, $y_{\mathbf{L_T}}(P)$ associated to each good matching $P$ of $S_{\alpha,T}$. From there, using analogous techniques to those employed in Section 5, we obtain expansion formulae for each combinatorial type of tagged quasi arc. \newline \indent In Section 7 we extend the results to the case where our initial $T$ is an arbitrary tagged quasi-triangulation. In Section 8 we prove a generalised version of Fomin and Zelevinsky's `separation of additions' formula with respect to any principal lamination. Namely, for any quasi-arc $\alpha$ and any multi-lamination $\mathbf{L}$, we show that $x_{\mathbf{L}}(\alpha)$ can be obtained via a certain specialisation of $x_{\mathbf{L_T}}(\alpha)$, where $\mathbf{L_T}$ is a principal lamination. As a direct corollary, we prove Lam and Pylyavskyy's positivity conjecture for all quasi-cluster algebras, with respect to any choice of coefficients.

\section*{\large \centering Acknowledgements}

The results of this paper were obtained during my postdoctoral fellowship at IMUNAM. I wish to thank Christof Geiss and Daniel Labardini-Fragoso for the warm and stimulating environment they provided during my time here. I am also grateful to Ilke Canakci, Emily Gunawan and Daniel Labardini-Fragoso for helpful discussions.

\section{Quasi-cluster algebras}

\subsection{Quasi-triangulations and flips}

In this section we recall the work of \cite{wilson2018laurent}, which arose from the combined works of Fomin, Thurston \cite{fomin2012cluster}, and Dupont, Palesi \cite{dupont2015quasi}. \newline \indent

Let $S$ be a compact $2$-dimensional manifold and let $M$ be a finite set of \textit{marked} points of $S$ such that each boundary component contains at least one marked point. We refer to marked points in the interior of $S$ as \textit{punctures}. \newline \indent The general idea is that `\textit{triangulations}' of a tuple $(S,M)$ will provide us with a cluster structure. However, some $(S,M)$ will be too `small' to admit such a structure, and we therefore wish to exclude them from our consideration -- the following definition indicates all such tuples.

\begin{defn}
A tuple $(S,M)$ is called a \textit{\textbf{bordered surface}} if $(S,M)$ is not an unpunctured or once-punctured monogon, digon or triangle; a once or twice punctured sphere; a M\"obius strip with one marked point on the boundary; the once-punctured projective space; the thrice-punctured sphere; the twice-punctured projective space; the once-punctured Klein bottle.

\end{defn}

\begin{defn}

An \textit{\textbf{ordinary arc}} of $(S,M)$ is a simple curve in $S$ whose two endpoints are in $M$, and which is not homotopic to a boundary arc or a marked point.
\end{defn}

\begin{defn}

A \textit{\textbf{tagged arc}} $ \gamma $ is an ordinary arc which has been `tagged' at its endpoints in one of two possible ways; \textbf{\textit{plain}} or \textbf{\textit{notched}}. Moreover, the tagging is required to satisfy the following conditions:

\begin{itemize}

\item Any endpoint of $ \gamma $ lying on the boundary $ \partial S $ is tagged plain.

\item If the endpoints of $ \gamma $ coincide they must be tagged the same way.

\end{itemize}

\end{defn}

\begin{defn}

A simple closed curve in $ S $ is called \textbf{\textit{two-sided}} if it admits a regular neighbourhood which is orientable. If no such neighbourhood exists it is called \textit{\textbf{one-sided}}.

\end{defn}

\begin{defn}

A \textbf{\textit{tagged quasi-arc}} is either a tagged arc or a one-sided closed curve.\end{defn}

\begin{rmk}

Often, when the context is clear, we shall simply refer to a tagged quasi-arc as a quasi-arc. Throughout the paper we shall always consider (tagged) quasi-arcs up to isotopy.

\end{rmk}

\begin{defn}

A \textbf{\textit{cross-cap}} is a cylinder where antipodal points on one of the boundary components are identified. Namely, it is real projective plane with one boundary component. See Figure \ref{crosscap}.
\end{defn}

\begin{rmk}

Note that any compact non-orientable surface (with boundary) is homeomorphic to a sphere where (more than) $ k $ open disks are removed, and $ k $ of them have been exchanged for cross-caps, for some $ k\geq 1 $.

\end{rmk}

\begin{figure}[H]
\begin{center}
\includegraphics[width=3cm]{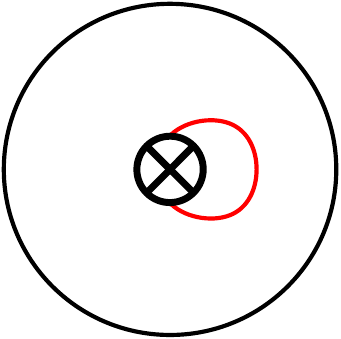}
\caption{Here we depict a crosscap. The red curve is an example of a one-sided closed curve.}
\label{crosscap}
\end{center}
\end{figure}

\begin{defn}[Compatibility of arcs] 

Let $ \alpha $ and $ \beta $ be two arcs of $ (S,M) $. We say $ \alpha $ and $ \beta $ are \textbf{\textit{compatible}} \textit{if and only if} the following conditions are all satisfied:

\begin{itemize}

\item There exist isotopic representatives of $ \alpha $ and $ \beta $ that do not intersect in the interior of $S$.

\item Suppose the untagged versions of $ \alpha $ and $ \beta $ do not coincide. Then if $\alpha$ and $\beta$ share an endpoint $p$, the ends of $\alpha$ and $\beta$ must be tagged the same way at $p$.

\item Suppose the untagged versions of $ \alpha $ and $ \beta $ do coincide. Then exactly one end of $ \alpha $ must be tagged the same way as the corresponding end of $\beta$.

\end{itemize}

\end{defn}

\underline{\textbf{Notation}}: Let $\gamma$ be an arc bounding a M\"{o}bius strip with one marked point, $M_1^{\gamma}$. We denote the unique one-sided closed curve of $M_1^{\gamma}$ by $\alpha_{\gamma}$, and we denote the unique arc of $M_1^{\gamma}$ by $\beta_{\gamma}$ -- see Figure \ref{intersectioncompatible}.

\begin{figure}[H]
\begin{center}
\includegraphics[width=3cm]{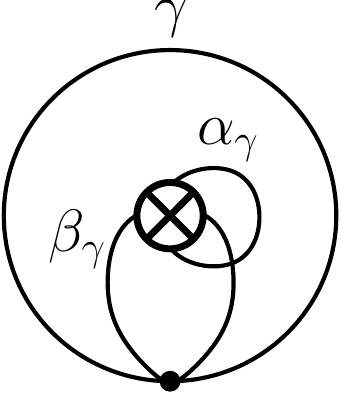}
\caption{The two quasi-arcs $\alpha_{\gamma}$ and $\beta_{\gamma}$ enclosed in the M\"obius strip $M_1^{\gamma}$.}
\label{intersectioncompatible}
\end{center}
\end{figure}

\begin{defn}[Compatibility of quasi-arcs]
\label{newcompatibilitydef}
Two quasi-arcs $ \alpha $ and $\beta$ are said to be \textbf{\textit{compatible}} if either:
\begin{itemize}
\item $\alpha$ and $\beta$ are compatible arcs;
\item $\alpha$ and $\beta$ are \underline{not} both arcs, and either $\alpha$ and $\beta$ do not intersect or $\{ \alpha,\beta \} = \{ \alpha_{\gamma} , \beta_{\gamma} \}$ for some arc $\gamma$ bounding a M\"obius strip with one marked point on the boundary - see Figure \ref{intersectioncompatible}.
\end{itemize}

\end{defn}

\begin{defn}

\label{quasitriangulationdef}
A \textbf{\textit{quasi-triangulation}} of $ (S,M) $ is a maximal collection of pairwise compatible quasi-arcs of $ (S,M) $ which contains no arcs that cut out a once-punctured monogon, or a M\"obius strip with one marked point on the boundary.

\end{defn}

\begin{prop}[Proposition 3.11, \cite{wilson2018laurent}]
\label{flip}
Let $ T $ be a quasi-triangulation of $ (S,M) $. Then for any quasi-arc $ \gamma \in T $ there exists a unique quasi-arc $\gamma'$ in $(S,M)$ such that $\gamma' \neq \gamma$ and $\mu_{ \gamma }(T) := T \setminus \{ \gamma \} \cup \gamma' $ is a quasi-triangulation. We say $\gamma'$ is the \textbf{\textit{flip}} of $\gamma$ with respect to $T$.

\end{prop}

In Figure \ref{combinatorialflips} we list the possible flips (up to a change in tagging at any potential punctures) between quasi-arcs and the relationships between their corresponding lambda lengths. \begin{enumerate}[label=\arabic*)]

\item $\gamma$ is the diagonal of quadrilateral in $(S,M)$ in which no two consecutive edges are glued together.

\begin{figure}[H]
\begin{center}
\includegraphics[width=10cm]{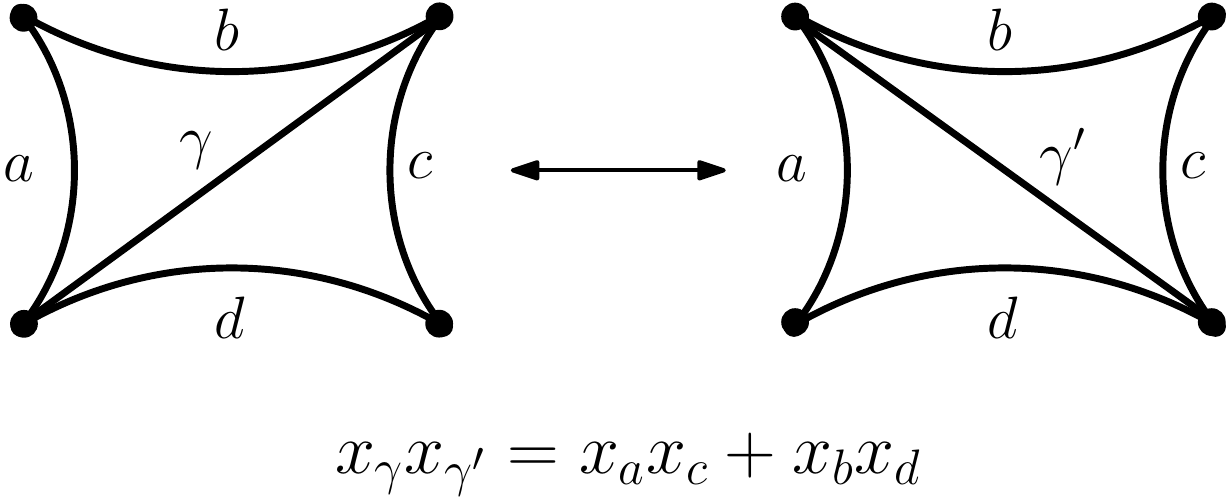}
\end{center}
\end{figure}

\item $\gamma$ is an interior arc of a once-punctured digon.

\begin{figure}[H]
\begin{center}
\includegraphics[width=11cm]{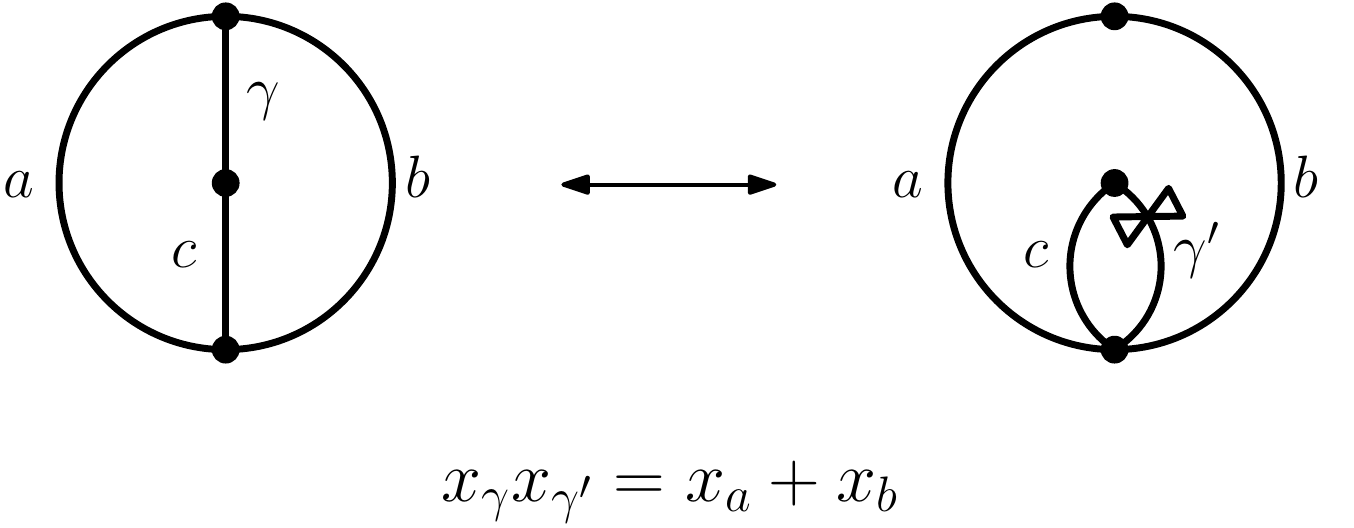}
\end{center}
\end{figure}

\item $\gamma$ is an arc that flips to a one-sided closed curve, or vice verca.

\begin{figure}[H]
\begin{center}
\includegraphics[width=11cm]{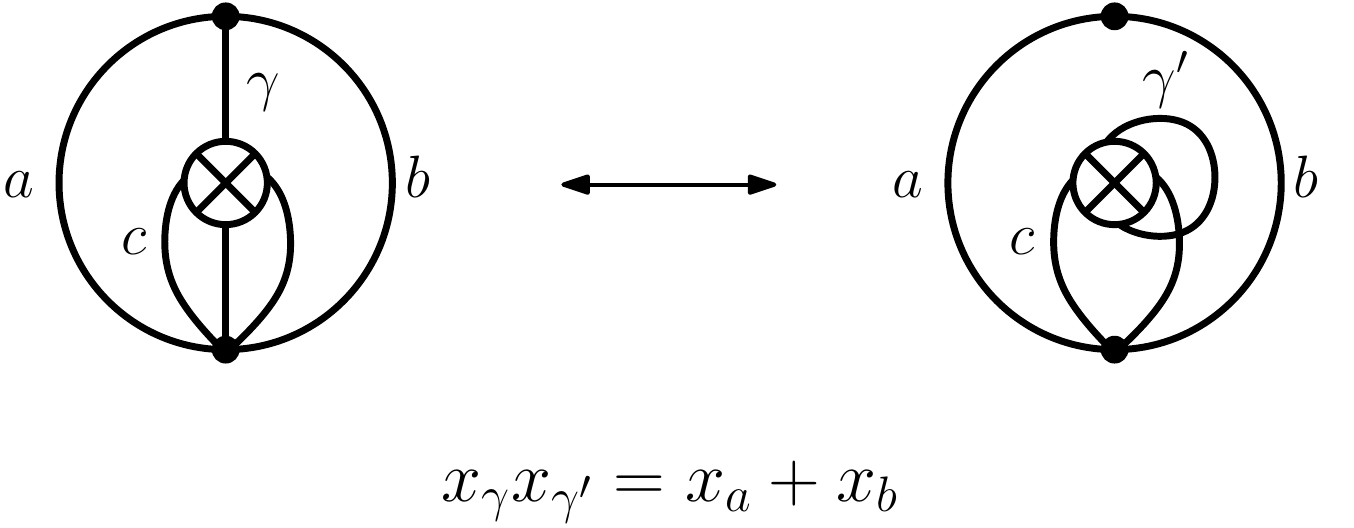}
\end{center}
\end{figure}

\item $\gamma$ is an arc intersecting a one-sided close curve $c$.

\begin{figure}[H]
\begin{center}
\includegraphics[width=11cm]{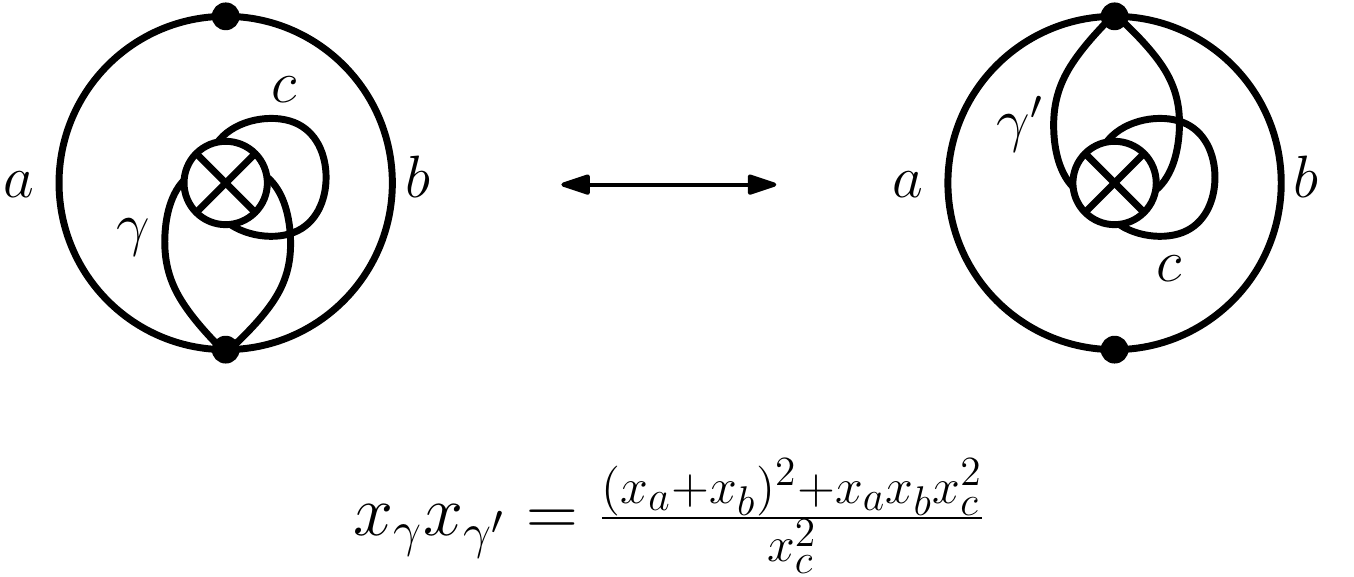}
\caption{A list of all combinatorial flip types as well as their corresponding exchange relations. Note that we may change the tagging at each marked point on the boundary of these configurations, however, the exchange relation will remain unchanged.}
\label{combinatorialflips}
\end{center}
\end{figure}

\end{enumerate}

\begin{defn}

A quasi triangulation is referred to as a \textit{\textbf{triangulation}} if it contains no one-sided closed curves.

\end{defn}

\subsection{The quasi-cluster algebra}

\begin{defn}

Let $\mathbb{T}_n$ be the (labelled) $n$-regular tree where edges are labelled by the numbers $1,\ldots, n$ such that the $n$ edges incident to a vertex receive different labels.

\end{defn}

Let $(\mathbf{x},T)$ be a seed of $(S,M)$. If we label the cluster variables of $\mathbf{x}$ $1,\ldots, n$ then we can consider the labelled n-regular tree $\mathbb{T}_n$ generated by this seed through mutations. Each vertex in $\mathbb{T}_n$ has $n$ incident vertices labelled $1,\ldots,n$. Vertices represent seeds and the edges correspond to mutation. In particular, the label of the edge indicates which direction the seed is being mutated in. \newline

Let $\mathcal{X}$ be the set of all cluster variables appearing in the seeds of $\mathbb{T}_n$. $\mathcal{A}_{(\mathbf{x},T)}(S,M) := \mathbb{ZP}[\mathcal{X}]$ is the \textit{\textbf{quasi-cluster algebra}} of the seed $(\mathbf{x},T)$.

The definition of a quasi-cluster algebra depends on the choice of the initial seed. However, if we choose a different initial seed the resulting quasi-cluster algebra will be isomorphic to $\mathcal{A}_{(\mathbf{x},T)}(S,M)$. As such, it makes sense to talk about the quasi-cluster algebra of $(S,M)$.

\subsection{Quasi-cluster algebras with arbitrary coefficients}

\begin{defn}
\label{laminationdef}
Let $(S,M)$ be a bordered surface and let $L$ be a finite collection of curves in $(S,M)$ which are non-self-intersecting and pairwise non-intersecting curves. Such a collection $L$ is called a \textit{\textbf{lamination}} of $(S,M)$ if each connected component of $L$ is any one of the following curves:

\begin{itemize}

\item A curve which connects two unmarked points in $\partial S$, but is not isotopic to a piece of boundary containing one or zero marked points;

\item A curve with one endpoint belonging to $\partial S \setminus M$, and whose other end spirals into a puncture;

\item A curve with both ends spiralling into puncture(s). However, we forbid the case when both ends of the curve spiral into the same puncture if the curve does not enclose anything else;

\item A one-sided closed curve;

\item A two-sided closed curve which does not bound a M\"{o}bius strip or a disk with zero or one punctures.

\end{itemize}

\end{defn}

\begin{defn}

A \textbf{\textit{multi-lamination}}, $\mathbf{L}$, of a bordered surface $(S,M)$ consists of a finite collection of laminations of $(S,M)$.

\end{defn}

\begin{figure}[H]
\begin{center}
\includegraphics[width=11cm]{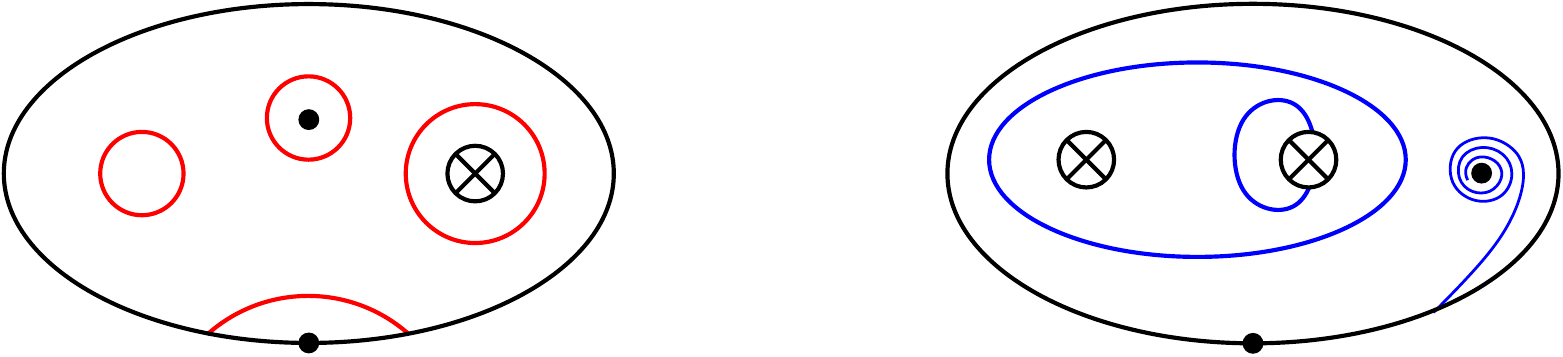}
\caption{In the left surface none of the curves are laminations. All curves in the right surface are legitimate laminations.}
\label{laminations}
\end{center}
\end{figure}

\begin{defn}[Tropical lambda length]

Let $\gamma$ be a quasi-arc on an unpunctured surface $(S,M)$, and let $\mathbf{L}$ be a multi-lamination. For each $L$ in $\mathbf{L}$ we define $l_{L}(\gamma))$ as follows:

\begin{itemize}

\item If $\gamma$ is an arc then $l_{L}(\gamma)$ is the minimal number of intersection points between $L$ and any curve isotopic to $\gamma$.

\item If $\gamma$ is a one-sided closed curve then the definition splits into two cases: 

\begin{itemize}

\item  if $ L $ contains no curves homotopic to $ \gamma $ then $ l_{L}(\gamma) $ is the minimal number of intersection points between $ L $ and any one-sided closed curve which is homotopic to $ \gamma $.

\item if $L$ contains a curve homotopic to $\gamma$, then $l_{L}(\gamma)$ is defined as \textbf{minus} the number of curves in $L$ which are homotopic to $\gamma$.
 
\end{itemize}
\end{itemize}

The \textit{\textbf{tropical lambda length}}, $c_{\mathbf{L}}(\gamma)$, of $\gamma$ is defined as:

\begin{equation}
c_{\mathbf{L}}(\gamma) = \prod_{L_i \in \mathbf{L}} x_{n+i}^{-\frac{l_{L_i}(\gamma)}{2}}
\end{equation}

\end{defn}

\begin{rmk}

Note that for punctured surfaces the tropical lambda length is not well defined. Namely, if a lamination $L$ spirals around a given puncture $p$, then any arc incident to $p$ will intersect $L$ infinitely many times. However, one can `open' up the punctures to obtain an unpunctured surface, and can then proceed to define the tropical lambda for all tagged arcs $\gamma$ and laminations on punctured surfaces too. This tropical lambda length will depend on the choice of `lift', $\overline{\gamma}$, of $\gamma$ to the opened surface, however, the `laminated lambda length' defined below only depends on the original arc $\gamma$ (as then the quantity $x_{\mathbf{\overline{L}}}(\gamma) := \frac{\lambda(\overline{\gamma})}{c_{\mathbf{\overline{L}}}(\overline{\gamma})}$ turns out to be independent of the choice of $\overline{\gamma}$). A vigourous exposition of this may be found in \cite{wilson2018laurent}.

\end{rmk}

\begin{defn}

Let $(S,M)$ be a bordered surface, and let $\mathbf{L}$ be a multi-lamination. For each quasi-arc $\gamma$ we define the \textbf{\textit{laminated lambda length}}, $x_{\mathbf{L}}(\gamma)$, of $\gamma$ to be:

\begin{equation}
\label{laminated length}
x_{\mathbf{L}}(\gamma) := \frac{\lambda(\gamma)}{c_{\mathbf{L}}(\gamma)}
\end{equation}

\end{defn}

\begin{prop}[Theorem 5.44, \cite{wilson2018laurent}]

Let $T = \{x_1,\ldots, x_n\}$ be a quasi-triangulation and let $\mathbf{L} = \{L_1,\ldots, L_m\}$ be a multi-lamination. Then the collection of variables $x_{\mathbf{L}}(\gamma_1), \ldots, x_{\mathbf{L}}(\gamma_n)$ are algebraically independent over $\mathbb{Z}[x_{n+1},\ldots, x_{n+m}]$.

\end{prop}

\begin{defn}

Fix a bordered surface $(S,M)$ of rank $n$ and a multi-lamination $\mathbf{L} = \{L_1,\ldots, L_m\}$. Let $\mathcal{F}$ be the field of rational functions in $n$ formal variables over the coefficient ring $\mathbb{Z}[x_{n+1},\ldots,x_{n+m}]$. \newline \indent A \textbf{\textit{(labelled) seed}} of $(S,M,\mathbf{L})$ is a pair, $(\mathbf{x},T)$, where:

\begin{itemize}

\item $T$ is a quasi-triangulation with quasi-arcs labelled $1 \ldots, n$,

\item $\mathbf{x} := (x_{1} \ldots, x_{n})$ is an (ordered) $n$-tuple of elements in $\mathcal{F}$ which are algebraically independent over $\mathbb{Z}[x_{n+1},\ldots ,x_{n+m}]$.
 
 \end{itemize}
 
 We say $\mathbf{x}$ is a \textbf{\textit{cluster}} and each $x_{i}$ is a \textbf{\textit{cluster variable}}.

\end{defn}

\begin{defn}[Mutation of seeds]
\label{mutation of seeds}

Let $$(\mathbf{x} := (x_{1},\ldots, x_{n}),T)$$ be a labelled seed of $(S,M,\mathbf{L})$ and let $k \in \{1,\ldots, n\}$. We define \textit{\textbf{mutation}} of $(\mathbf{x},T)$ in \textit{\textbf{direction $k$}} to be the new seed $\mu_{k}(\mathbf{x},T) := (\mathbf{x}',T')$, where $\mathbf{x}' := (x_{1}',\ldots, x_{n}')$ and  $T' $ are defined as follows:

\begin{itemize}

\item $T'$ is obtained from flipping the quasi-arc $\gamma_k$ in $T$ labelled by $k$. The new quasi-arc $\gamma_k'$ is then labelled by $k$, and the label on all other quasi-arcs remains the same.

\item If $i \neq k$ then $x_i' := x_i$.

\item $x_k' := \frac{P(x_1, \ldots, x_n)}{x_k}$, where $P$ is the Laurent polynomial with coefficients in $\mathbb{Z}[x_{n+1},\ldots, x_{n+m}]$ such that: $$x_{\mathbf{L}}(\gamma_k)x_{\mathbf{L}}(\gamma_k') = P(x_{\mathbf{L}}(\gamma_k), \ldots, x_{\mathbf{L}}(\gamma_k))$$

\end{itemize}

\end{defn}

\begin{rmk}

Note that the Laurent polynomial $P$ in Definition \ref{mutation of seeds} can be obtained from substituting equation (\ref{laminated length}) into the exchange relation of the corresponding flip type found in Definition \ref{flip}.

\end{rmk}

\begin{defn}

As before, a \textit{\textbf{(labelled) seed pattern}} is an assignment of a labelled seed to each vertex of a labelled tree $\mathbb{T}_n$, such that any two seeds connected by an edge labelled by a $k$, are related by a mutation in direction $k$.

Note that a single labelled seed $(\mathbf{x},T)$ completely determines the whole labelled seed pattern. Let $\mathcal{X}$ be the set of all cluster variables appearing in all of the seeds in the seed pattern generated by $(\mathbf{x},T)$. We say $$\mathcal{A}_{(\mathbf{x},T)}(S,M,\mathbf{L}) := \mathbb{Z}[x_{n+1},\ldots, x_{n+m}][\mathcal{X}]$$ is the \textit{\textbf{laminated quasi-cluster algebra}} of the initial labelled seed $(\mathbf{x},T)$. \newline \indent

\end{defn}

\begin{rmk}

For any two labelled seeds $(\mathbf{x},T)$ and $(\mathbf{x}',T')$ we see that their corresponding laminated quasi-cluster algebras are isomorphic. In that sense, we may talk about, $\mathcal{A}(S,M,\mathbf{L})$, the laminated quasi-cluster algebra of $(S,M,\mathbf{L})$.

\end{rmk}

\subsection{The combinatorics behind cluster algebras from surfaces}

Let $T$ be a tagged triangulation of an orientable bordered surface $(S,M)$ and let $\mathbf{L}$ be a multi-lamination. We now wish to describe a convenient way of encoding the exchange relations of the tagged arcs of $T$ via the so called \textit{extended exchange matrix} $\tilde{B}_T$. We shall first explain the procedure in the absence of $\mathbf{L}$.

\begin{defn}

An \textit{\textbf{ideal triangulation}} of $(S,M)$ is a maximal collection of pairwise non-intersecting ordinary arcs. To each tagged triangulation $T$ we may associate a unique ideal triangulation $T^{\circ}$ by applying the rules below in the order they are listed.

\begin{itemize}

\item If a puncture $p$ has more than one incident notched endpoint, then replace all these notched endpoints at $p$ with plain ones.

\item If a puncture $p$ has precisely one incident notched endpoint, then replace the tagged arc $\gamma^{(p)}$ (to which this notch belongs) with the ordinary arc $\ell_p$ which encloses $\gamma^{(p)}$ and $p$ in a once-punctured monogon. We call $\ell_p$ a \textbf{\textit{loop}}.

\end{itemize}

\end{defn}

\begin{defn}

Let $T =\{\gamma_1,\ldots, \gamma_n\}$ be a tagged triangulation with associated ideal triangulation $T^{\circ} =\{\gamma_1^{\circ},\ldots, \gamma_n^{\circ}\}$. We define a quiver $Q_T$ of $T$. The vertices $i$ of $Q_T$ correspond to the tagged arcs $\gamma_i$ of $T$, which in turn correspond to $\gamma_i^{\circ}$ in $T^{\circ}$. The arrows of $Q_T$ are defined as follows: for each non-self-folded triangle $\Delta$ in $T^{\circ}$ there is an arrow $i \rightarrow j$ \textit{if and only if} one of the following is satisfied:

\begin{itemize}

\item $\gamma_i^{\circ}$ and $\gamma_j^{\circ}$ are sides of $\Delta$, and $\gamma_j^{\circ}$ follows $\gamma_i^{\circ}$ in a clockwise ordering;

\item $\ell$ is a loop in $T^{\circ}$ enclosing $\gamma_i^{\circ}$. Moreover, $\ell$ and $\gamma_j^{\circ}$ are sides of $\Delta$, and $\gamma_j^{\circ}$ follows $\ell$ in a clockwise ordering;

\item $\ell$ is a loop in $T^{\circ}$ enclosing $\gamma_j^{\circ}$. Moreover, $\ell$ and $\gamma_i^{\circ}$ are sides of $\Delta$, and $\ell$ follows $\gamma_i^{\circ}$ in a clockwise ordering.

\end{itemize}

\end{defn}

Each quiver $Q_T$ may naturally be identified with a skew-symmetric matrix $B(T) = (b_{ij})$. Specifically, the columns and rows of $B(T)$ are labelled by the vertices of $Q_T$, and $$b_{ij} := \# \{\text{arrows \hspace{1mm}} i \rightarrow j \text{\hspace{1mm} in \hspace{1mm}} Q_{T}\} - \# \{\text{arrows \hspace{1mm}} j \rightarrow i \text{\hspace{1mm} in \hspace{1mm}} Q_{T}\}.$$ This form is particularly convenient since the columns of $B(T)$ encode the exchange relations of the tagged arcs in $T$. We shall now describe how one can define the anbalagous \textit{extended exchange matrix}: $$\tilde{B}(T) = (b_{ij})_{\substack{i \in \{1,\ldots,n\}\cup\{L \in \mathbf{L}\} \\ \hspace{-11.5mm}j \in \{1,\ldots,n\}}}$$ which encodes the exchange relations in the presence of a multi-lamination $\mathbf{L}$. To explain this procedure it suffices to consider a single lamination $L$.

\begin{defn}[Shear coordinates of ideal triangulations] 
\label{shear}
Let $ T $ be an ideal triangulation of an orientable bordered surface $ (S,M) $, and let $ L $ be a lamination. Suppose that $ \gamma $ is an arc of $ T $ which is not the folded side of a self-folded triangle, and let $ Q_{\gamma} $ be the quadrilateral of $ T $ whose diagonal is $ \gamma $. The \textit{\textbf{shear coordinate}}, $b_{T}(L, \gamma)$, of $ L $ and $ \gamma $, with respect to $T$, is defined by:
$$ b_{T}(L, \gamma) := \# \Big\{ \stackanchor{\text{$S$-shape intersections}}{\text{of $L$ with $Q_{\gamma}$}}\Big\} - \# \Big\{ \stackanchor{\text{$Z$-shape intersections}}{\text{of $L$ with $Q_{\gamma}$}}\Big\}$$

In the context of the extended exchange matrix $\tilde{B}(T)$, for each $L \in \mathbf{L}$ we define $b_{L j} := b_{T}(L, \gamma_j)$.

\end{defn}

\begin{figure}[H]
\begin{center}
\includegraphics[width=11cm]{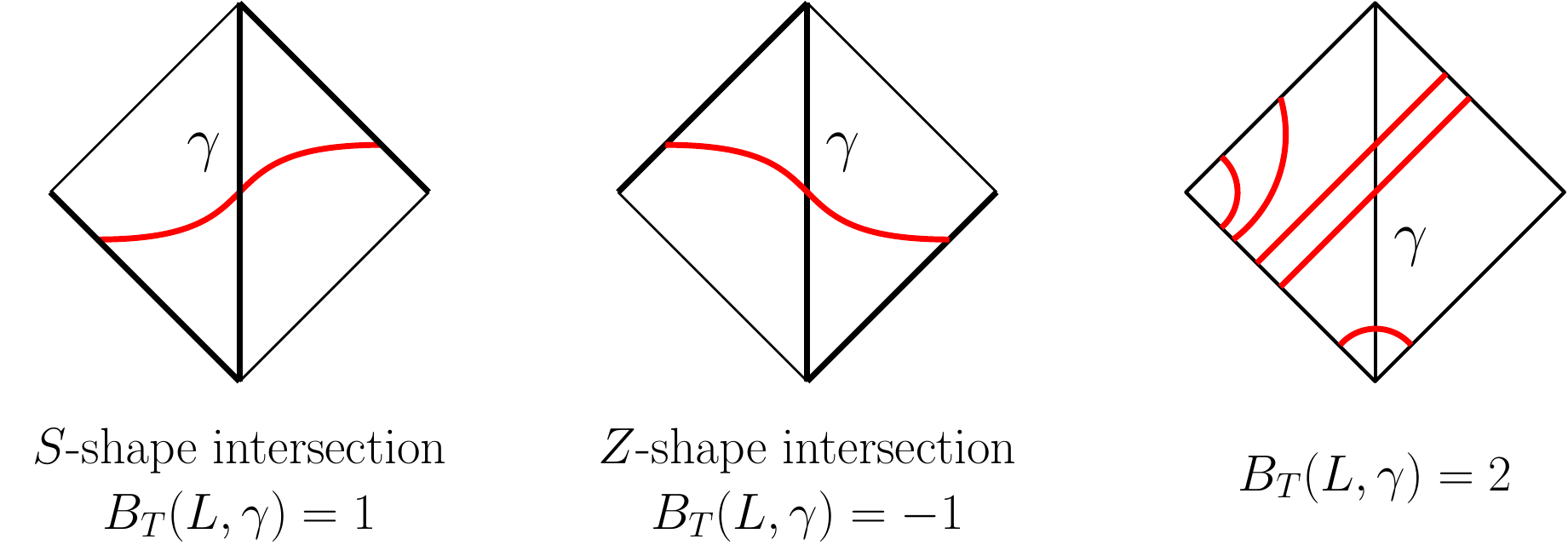}
\caption{$S$-shape and $Z$-shape intersections.}
\label{s-z-figure}
\end{center}
\end{figure}

We explain below how Fomin and Thurston \cite{fomin2012cluster} extended the notion of shear coordinates to (tagged) triangulations of orientable bordered surfaces.

\begin{defn}[Shear coordinates of triangulations]

Let $T$ be a triangulation and $L$ a lamination. To each puncture $p$ we apply the following procedure: if $L$ spirals into a puncture $p$, and all arcs incident to $p$ are notched at $p$, then reverse the direction of spiralling of $L$ at $p$, and replace all these notched taggings with plain ones. \newline
\indent Applying the procedure above we can convert the lamination $L$ of $T$ into a lamination $L_1$ of a triangulation $T_1$, where any notched arc in $T_1$ appears with its plain counterpart. Let $T^{\circ}$ denote the ideal triangulation associated to $T_1$. For each arc $\gamma$ of $T$, let $\gamma^{\circ}$ denote the corresponding arc in $T^{\circ}$. We define the \textit{\textbf{shear coordinate}}, $ b_{T}(L,\gamma) $, of $ L $ and $ \gamma $ (with respect to $T$) as follows:

\begin{itemize}

\item if $ \gamma^{\circ} $ is \underline{not} the self-folded side of a triangle in $T^{\circ}$ then define $$ b_{T}(L,\gamma) := b_{ T^{\circ}}(L_1,\gamma^{\circ}); $$

\item otherwise, $\gamma^{\circ}$ is the self-folded side of a triangle in $T^{\circ}$ with associated puncture $p$. In this case, reverse the direction of spiralling of $ L_1 $ at $ p $ and denote the new lamination by $ L_2 $. Consider the triangle $\Delta$ in $T^{\circ}$ that is folded along $\gamma^{\circ}$, and denote the remaining side of $\Delta$ by $\beta$ . We define $$ b_{T}(L,\gamma):= b_{T^{\circ}}(L_2,\beta)$$ .

\end{itemize}

Analogous to Definition \ref{shear}, for each $L \in \mathbf{L}$ we define $b_{L j} := b_{T}(L, \gamma_j)$.

\end{defn}

Each multi-lamination $\mathbf{L}$ will give rise to a different cluster algebra $\mathcal{A}_{\mathbf{L}}(S,M)$. However, the following remarkable theorem of Fomin an Zelevinsky tells us that each cluster variable $x_{\alpha}^T$ in $\mathcal{A}_{\mathbf{L}}(S,M)$ may be obtained by specialising the corresponding cluster variable $X_{\alpha}^{T}$ in $\mathcal{A}_{\bullet}(S,M)$ -- the cluster algebra with \textit{principal coefficients}. See Remark \ref{principalcoeffcients} for more details on principal coefficients.

\begin{thm}[Theorem 3.7, \cite{fomin2007cluster}]
\label{orientablesep}

Consider a triangulation $T$, an arc $\alpha$, and a multi-lamination $\mathbf{L} = \{L_{n+1},\ldots, L_{m}\}$. Then we have

$$ x_{\alpha}^T(x_1,\ldots, x_n, x_{n+1}, \ldots, x_{m}) = \frac{X_{\alpha}^{T}|_{\mathcal{F}}(x_1,\ldots, x_n; y_1, \ldots, y_n)}{{X_{\alpha}^{T}|_{Trop(x_{n+1}, \ldots, x_{m})}}(1,\ldots, 1; y_1, \ldots, y_n)},$$

where:

\begin{itemize}

\item $X_{\alpha}^{T} \in \mathcal{A}_{\mathbf{\bullet}}(S,M)$ and  $x_{\alpha}^T \in \mathcal{A}_{\mathbf{L}}(S,M)$ are the cluster variables corresponding to $\alpha$,
\item $y_j = \displaystyle \prod_{k=n+1}^m x_{k}^{b_{kj}}$ (here $b_{kj} := b_{T}(L_k,\gamma_j)).$

\end{itemize}

\end{thm}

\subsection{The orientable double cover}

Throughout this paper it will often prove very useful to lift our non-orientable surface to its \textit{orientable double cover} - we describe the construction of this cover below. \newline \indent

Let $(S,M)$ be a bordered surface. We construct an orientable double cover of $(S,M)$ as follows. First consider the orientable surface $\tilde{S}$ obtained by replacing each cross-cap with a cylinder, see Figure \ref{surfaceandcylinder}.

\begin{figure}[H]
\begin{center}
\includegraphics[width=10cm]{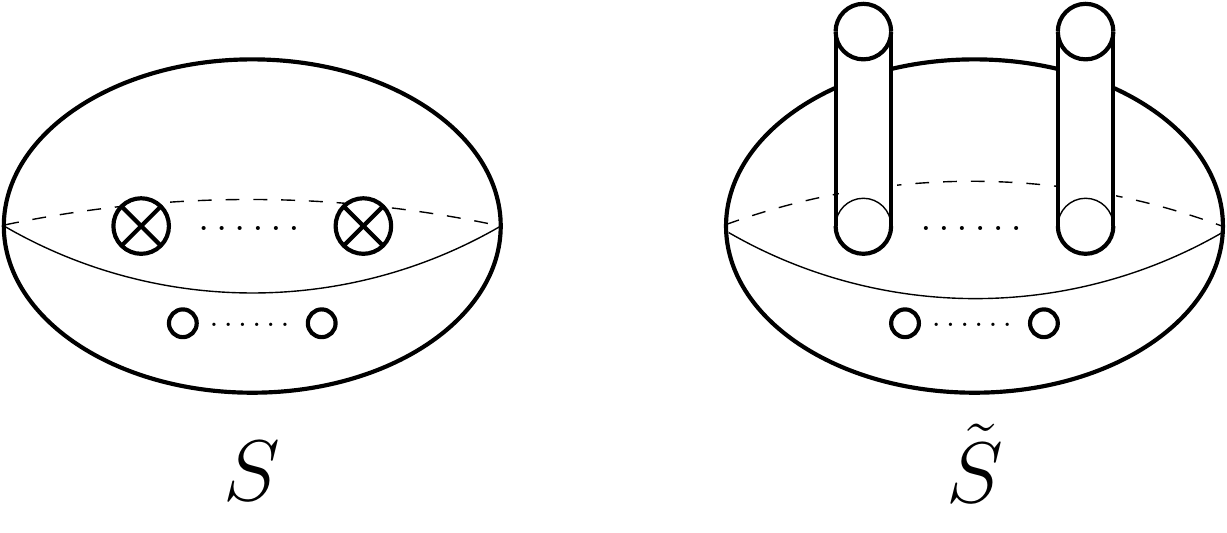}
\caption{A non-orientable surface $S$ and the surface $\tilde{S}$ obtained by replacing each cross-cap with a cylinder. The small circles represent boundary components.}
\label{surfaceandcylinder}
\end{center}
\end{figure}

The orientable double cover $\overline{(S,M)}$ of $(S,M)$ is obtained by gluing together two copies of $\tilde{S}$. Specifically we glue each newly adjoined cylinder in the first copy, with a half twist, to the corresponding cylinder in the second copy, see Figure \ref{surfaceglueing}.

\begin{figure}[H]
\begin{center}
\includegraphics[width=10cm]{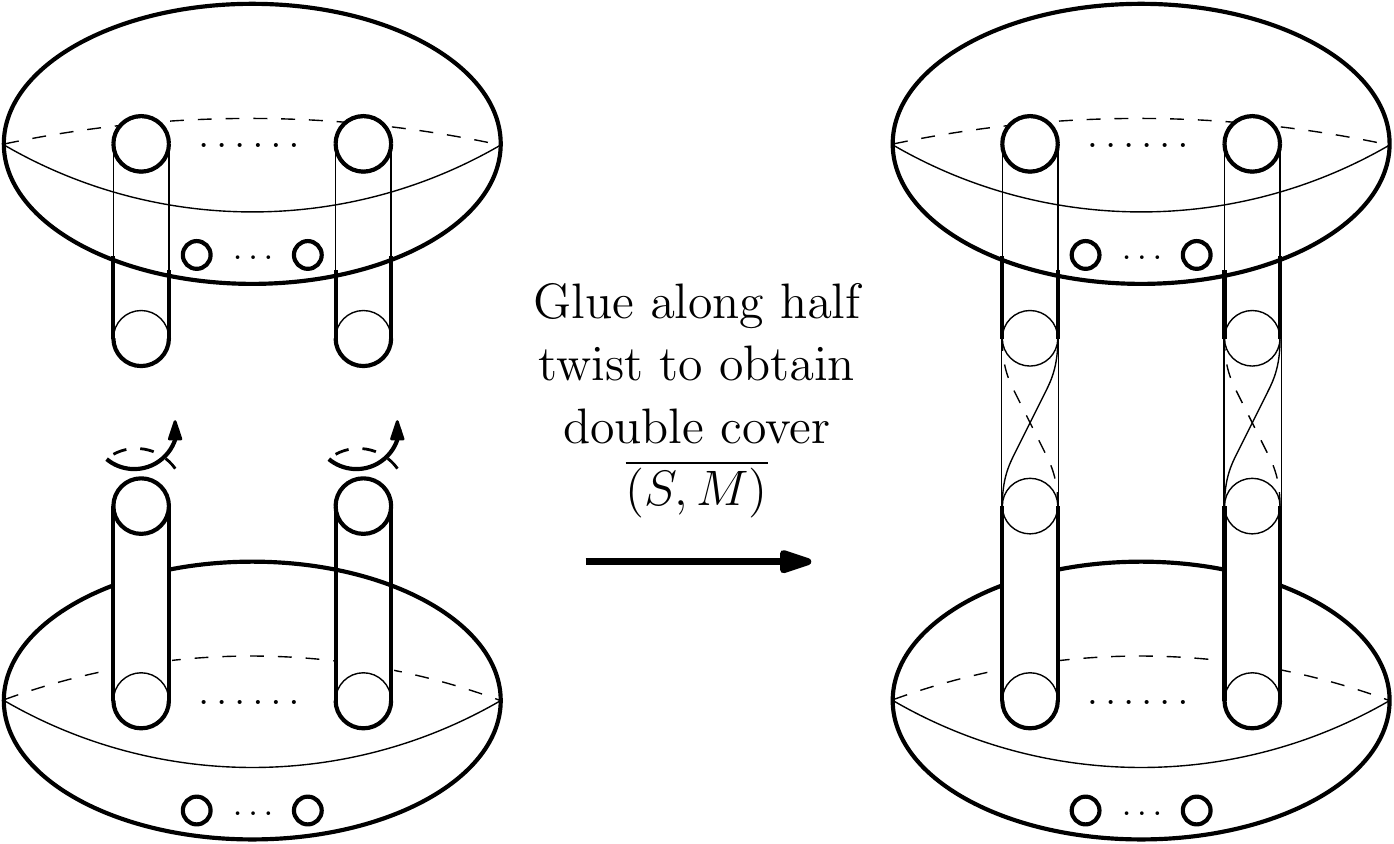}
\caption{The double cover $\overline{(S,M)}$ obtained by glueing two copies of $\tilde{S}$ along the boundaries of the newly adjoined cylinders.}
\label{surfaceglueing}
\end{center}
\end{figure}

\begin{rmk}

Note that each tagged arc $\gamma$ of $(S,M)$ lifts to two tagged arcs $\overline{\gamma}$ and $\tilde{\overline{\gamma}}$ of $\overline{(S,M)}$. Moreover, each tagged triangulation $T$ of $(S,M)$ lifts to a tagged triangulation $\overline{T}$ of the orientable double cover $\overline{(S,M)}$. It is worth noting that a one-sided closed curve $\alpha$ of $(S,M)$ lifts to a single closed curve $\overline{\alpha}$ in $\overline{(S,M)}$.

\end{rmk}

The following key proposition links mutation on $(S,M)$ to mutation on the orientable double cover $\overline{(S,M)}$

\begin{prop}[Proposition 4.2, \cite{wilson2018laurent}]
\label{quiverandflip}
Let $\gamma$ be an arc in a triangulation $T$. If $\mu_{\gamma}(T)$ is a triangulation then $\mu_{\overline{\gamma}}\circ\mu_{\tilde{\overline{\gamma}}}(\overline{T}) = \mu_{\tilde{\overline{\gamma}}}\circ\mu_{\overline{\gamma}}(\overline{T}) = \overline{\mu_{\gamma}(T)}$.

\end{prop}

In Definition \ref{mutation of seeds} the description of mutation requires one to keep track of the geometry of the surface, and, in particular, coefficient systems arise from multi-laminations on $(S,M)$. However, as shown in \cite{wilson2018laurent}, a purely combinatorial description of this process can be described in terms of Lam and Pylyavskyy's \textit{Laurent phenomenon algebras}  \cite{lam2012laurent}. In that context, one can create arbitrary coefficient systems by inserting frozen variables into the exchange polynomials. The following theorem tells us that any such coefficient system arises from a (unique) multi-lamination. Hence, if one wishes to imitate Fomin and Zelevinsky's coefficient systems of \textit{geometric} type, the set-up described in Definition \ref{mutation of seeds} is indeed the broadest one should ask for.

\begin{thm}[\cite{wilson2018laurent}]
\label{coefficientbijection}

Let $T$ be a tagged triangulation of $(S,M)$. Consider the lifted triangulation $\overline{T}:= \{\overline{\gamma}_1,\ldots, \overline{\gamma}_n, \tilde{\overline{\gamma}}_1, \ldots, \tilde{\overline{\gamma}}_n\}$ on the oriented double cover $\overline{(S,M)}$. For each lamination $L = \{L_1, \ldots, L_n \}$ on $(S,M)$, let $\overline{L} = \{\overline{L}_1, \ldots, \overline{L}_n\}$ denote the lifted lamination on $\overline{(S,M)}$. Then there exists the following bijection between laminations on $(S,M)$ and $\mathbb{Z}^n$:

\begin{align*}
 \Gamma_T \hspace{1mm}: \hspace{2mm} &\Big\{ \stackanchor{\text{\hspace{0.5mm} Laminations of \hspace{0.5mm}}}{\text{$(S,M)$}}\Big\} \hspace{5mm} \longrightarrow \hspace{25mm} \mathbb{Z}^n\\
  & \hspace{15mm} L \hspace{24mm}  \hspace{1mm} \mapsto \hspace{7mm} (b_{\overline{T}}(\overline{L},\overline{\gamma}_1), \ldots, b_{\overline{T}}(\overline{L},\overline{\gamma}_n)).
\end{align*}

\noindent where $b_{\overline{T}}(\overline{L},\overline{\gamma}_i)$ denotes the \textit{shear coordinate} of $\overline{\gamma_i}$ with respect to $\overline{T}$ and $\overline{L}$.

\end{thm}

\begin{rmk}
\label{coefficientbijectionrmk}

Coupling Theorem \ref{coefficientbijection} and [Theorem 6.21, \cite{wilson2018laurent}] we arrive at the following bijection:

$$\Large{\Bigg\{\substack{\text{coefficient systems}\\ \text{of $\mathcal{A}(S,M)$} \\ \text{w.r.t (specialised) LP algebras}}\Bigg\}\hspace{2mm} \longleftrightarrow \hspace{2mm} \Big\{ \substack{\text{multi-laminations}\\ \text{of $(S,M)$}}\Big\}}$$

\end{rmk}

\section{Snake and band graphs}

In this section we first recall the notion of an (abstract) \textit{snake} graph, as seen in \cite{canakci2013snake}, \cite{canakci2015snake}. Following this, we define (abstract) \textit{band} graphs. However, it should be noted our notion of band graph differs from that found in the literature \cite{ccanakcci2017snake}, \cite{felikson2017bases}, \cite{musiker2013bases} -- more specifically, edges with different \textit{sgn} may be glued.

\subsection{Snake graphs}

\begin{defn}

A \textit{\textbf{tile}} is a graph consisting of four vertices and four edges, such that each vertex has degree two. \newline 
\indent Embedding this in the plane, we shall always view a tile as a square where edges are parallel to the $x$ and $y$ axes. With respect to this embedding we shall label the edges North (N), East (E), South (S) and West (W) in the obvious way.\end{defn}

\begin{figure}[H]
\begin{center}
\includegraphics[width=3.5cm]{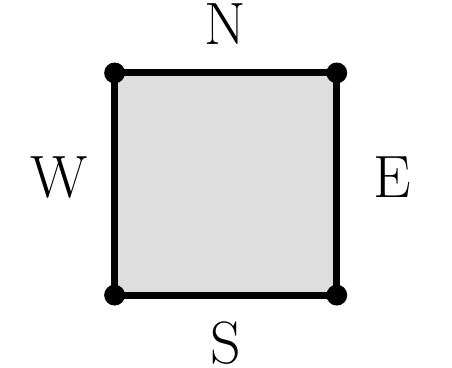}
\caption{A tile with the canonical north, east, south and west edge labelings.}
\label{tileshade}
\end{center}
\end{figure}

We say two tiles are \textit{\textbf{glued}} if they share a common edge. We shall now introduce the combinatorial structure underpinning this paper -- it revolves around the idea of glueing tiles together in a specific way.

\begin{defn}

A \textit{\textbf{snake graph}} $\mathcal{G} = (G_1, \ldots, G_d)$ is a sequence of tiles $G_1, \ldots, G_d$ such that for each $i \in \{1,\ldots, d-1\}$:

\begin{itemize}

\item the north or east edge of $G_i$ is glued to the south or west edge of $G_{i+1}$,

\item $G_i$ and $G_{i+1}$ share precisely one edge. We shall always denote this edge by $e_i$.

\end{itemize}

Moreover, we refer to a sequence of consecutive tiles $G_i, G_{i+1}, \ldots, G_{j-1}, G_j$ appearing in $\mathcal{G}$ as a \textit{\textbf{sub snake graph}} of $\mathcal{G}$.

\end{defn}

\begin{defn}

Let $\mathcal{G} = (G_1, \ldots, G_d)$ be a snake graph.

\begin{itemize} 

\item If the east (resp. north) edge of $G_i$ is glued to the west (resp. south) edge of $G_{i+1}$ for each $i \in \{1,\ldots, n-1\}$, then $\mathcal{G}$ is said to be \textit{\textbf{straight}}.

\item If no three consecutive tiles $G_i, G_{i+1}, G_{i+2}$ form a straight sub snake graph, then $\mathcal{G}$ is said to be \textit{\textbf{zig-zag}}.

\end{itemize}

\end{defn}

\begin{figure}[H]
\begin{center}
\includegraphics[width=10cm]{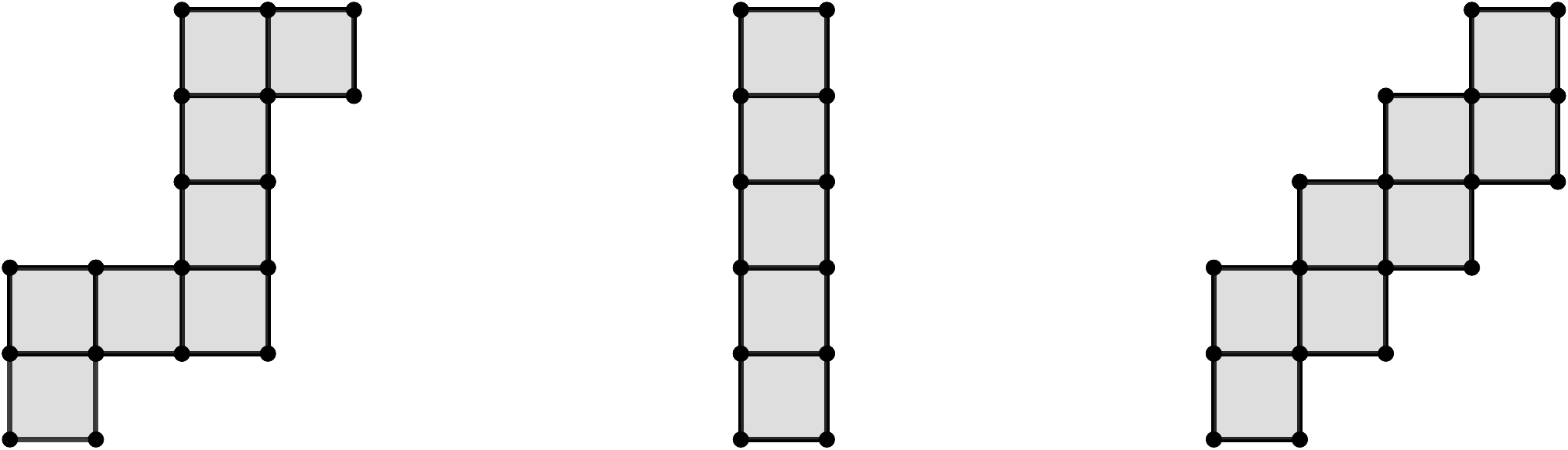}
\caption{An example of a snake graph, straight snake graph, and a zig-zag snake graph, respectively.}
\label{snakegraphshade}
\end{center}
\end{figure}

\begin{defn}

We define a \textbf{\textit{sign function}}, $sgn$, on a snake graph $\mathcal{G}$ to be a map from the set of edges of $\mathcal{G}$ to $\{+,-\}$ such that, for each tile $G_i$, we have: 

\begin{itemize}

\item $sgn(N(G_i)) = sgn(W(G_i)) \text{\hspace{5mm} and \hspace{5mm}} sgn(S(G_i)) = sgn(E(G_i))$

\item $sgn(N(G_i)) \neq sgn(S(G_i)).$

\end{itemize}

\end{defn}

\begin{figure}[H]
\begin{center}
\includegraphics[width=4.5cm]{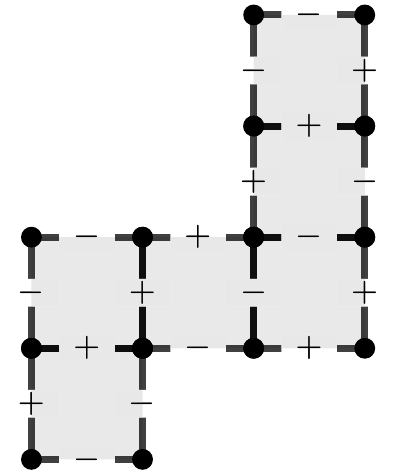}
\caption{An example of a $sgn$ function on a snake graph. Interchanging `$+$' and `$-$' produces the other possible $sgn$ function.}
\end{center}
\end{figure}

\begin{defn}
 
 A \textbf{perfect matching} of a graph $G$ is a collection of edges of $G$ such that \underline{every} vertex of $G$ is contained in \underline{exactly one} of these edges.
 
\end{defn}

\subsection{Band graphs}

Roughly speaking, a band graph is the result of glueing the ends of a snake graph together. Just as snake graphs enable us to obtain expansion formulae for cluster variables corresponding to arcs, band graphs will provide us with the framework to write expansion formulae for the variables corresponding to one-sided closed curves.

\begin{defn}
\label{bandgraph}

Let $\mathcal{G} = (G_1, \ldots, G_d)$ be a snake graph. \newline \indent Choose an edge in $\{S(G_1),W(G_1)\}$ and label it by $b$. Furthermore, let $x \in b$ denote the south-west vertex of $G_1$, and let $y$ denote the remaining vertex of $b$. \newline \indent Similarly, choose an edge in $\{N(G_d),E(G_d)\}$ and label it by $b'$. We denote by $y' \in b'$ the north-east vertex of $G_d$, and we let $x'$ denote the remaining vertex of $b'$.

The \textbf{\textit{band graph}}, $\mathcal{G}^{b}$, \textit{\textbf{glued along $\textbf{b}$}} is then defined to be the snake graph $\mathcal{G}$ glued along $b$ and $\tilde{b}$, such that $x$ (resp. $y$) is glued to $x'$ (resp. $y'$). \newline \indent Following the terminology of \cite{musiker2013bases}, we refer to the glued edge in $\mathcal{G}^{b}$ as the \textit{\textbf{cut edge}}.

\end{defn}

\begin{figure}[H]
\begin{center}
\includegraphics[width=4.5cm]{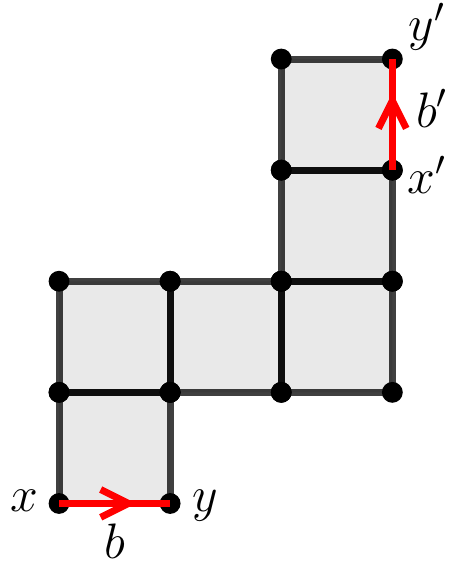}
\caption{An example of a band graph glued along $b$ and $b'$.}
\end{center}
\end{figure}

\begin{rmk}

Note that, unlike in \cite{ccanakcci2017snake}, \cite{felikson2017bases}, \cite{musiker2013bases}, we do not impose the condition that $sgn(b) = sgn(b')$. In fact, all band graphs found in this paper are completely opposite in the sense that they satisfy $sgn(b) \neq sgn(b')$. However, in our upcoming work on `Skein relations for non-orientable surfaces' and `Bases for quasi-cluster algebras' we shall encounter both situations, due to the added consideration of two-sided closed curves \cite{wilson2020bases}, \cite{wilson2020skein}.

\end{rmk}

\begin{defn}
\label{good matching}
 
As in Definition \ref{bandgraph}, let $\mathcal{G}^{b}$ be the band graph formed by glueing a snake graph $\mathcal{G}$ along $b$ and $\tilde{b}$. A \textbf{good matching} of $\mathcal{G}^{b}$ is a perfect matching which can be extended to a perfect matching of $\mathcal{G}$.
\end{defn}

\begin{rmk} Note that Definition \ref{good matching} may be restated as follows. A perfect matching $P$ of $\mathcal{G}^{b}$ is a good matching if the edges matching the vertices of the glued edge both lie on the same side of the cut -- specifically, when viewed as a matching of $\mathcal{G}$, $P$ must contain edge(s) matching both $x$ and $y$ or both $x'$ and $y'$. 

\begin{itemize}

\item  $P$ is a \textbf{\textit{right cut}} with respect to $b$ if $x$ and $y$ are matched in $\mathcal{G}$.

\item $P$ is a \textbf{\textit{left cut}} with respect to $b$ if $x'$ and $y'$ are matched in $\mathcal{G}$.

\item $P$ is a \textbf{\textit{centre cut}} with respect to $b$ if $P$ contains $b$.

\end{itemize}

\end{rmk}

\begin{figure}[H]
\begin{center}
\includegraphics[width=5cm]{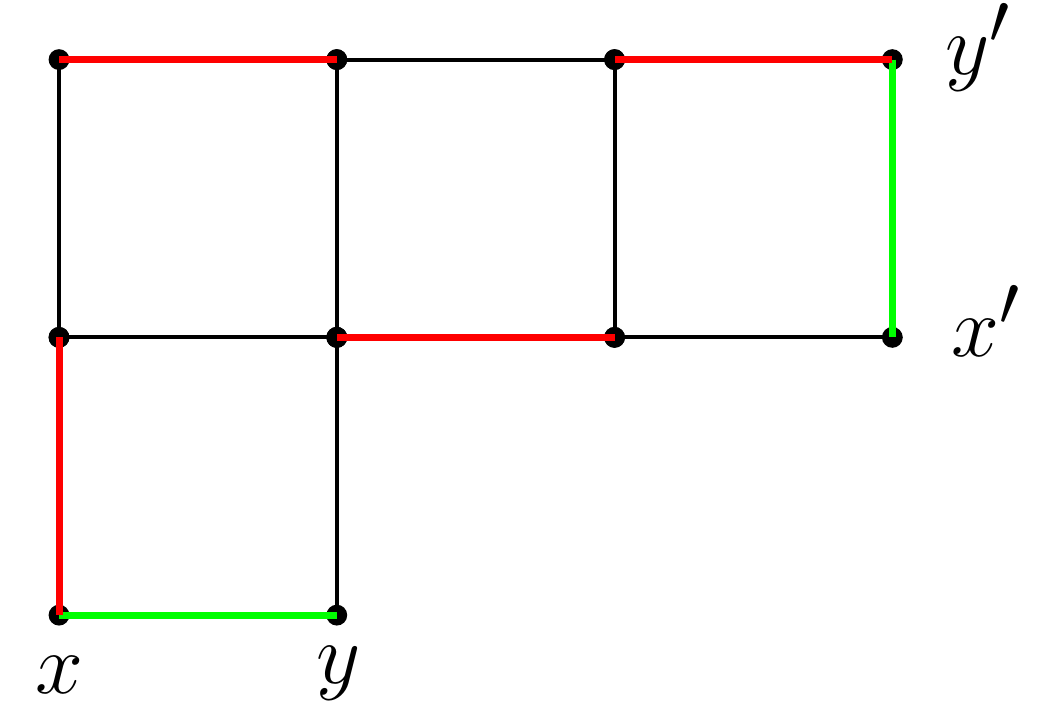}
\caption{A perfect matching but \underline{not} a good matching.}
\end{center}
\end{figure}

\section{Snake and band graphs from surfaces}
\label{snakebandconstruction}

In this section we explain how to associate snake and band graphs to quasi-arcs of $(S,M)$. To describe this procedure we shall lift ourselves to the orientable double cover $\overline{(S,M)}$, however, this is purely for convenience; it is not technically necessary to do so. \newline \indent Throughout this section we fix a triangulation $T$ of $\overline{(S,M)}$ which is the lift of a triangulation of $(S,M)$. \newline

The basic idea will be that tiles in our snake/band graph correspond to certain quadrilaterals on our surface.

\subsection{Snake graphs associated to arcs}

\begin{defn}
\label{relativeorientation}

Let $\gamma$ be a directed arc in $\overline{(S,M)}$, and denote by $p_1, \ldots, p_d$ the intersection points of $\gamma$ with our fixed triangulation $T$ -- listed in order of intersection. To this end, for $k \in \{1,\ldots, d\}$, we let $\tau_{i_k}$ denote the arc in $T$ containing the point $p_k$.

Let $Q_{i_j}$ be a quadrilateral in $T$ with diagonal labelled by $\tau_{i_j}$. We shall denote the triangles in $T$ either side of $\tau_{i_j}$ by $\Delta_j$ and $\Delta_{j+1}$. Moreover, we label them in such a way that, with respect to the orientation of $\gamma$ through $p_j$, $\Delta_j$ precedes $\Delta_{j+1}$. \newline
We view $Q_{i_j}$ as a tile $G_j$ by embedding it in the plane such that:

\begin{itemize}

\item the diagonal $\tau_{i_j}$ of $G_j$ connects the north-west and south-east vertices.

\item $\Delta_j$ (resp. $\Delta_{j+1}$) is the lower (upper) half of $G_j$.

\end{itemize}

Note that, following this procedure, there are two different ways to embed $Q_{i_j}$ as a tile $G_j$, and they differ in orienation. If the orientation on $Q_{i_j}$ (induced by $(S,M)$) agrees with the orientation on $G_j$ (induced by the clockwise orientation of the plane) then we say $rel(G_j) =1$. If it disagrees then $rel(G_j) = -1$. \newline \indent We call $rel(G_j)$ the \textbf{\textit{relative orientation}} of the tile $G_j$ with respect to $Q_{i_j}$.

\end{defn}

\begin{figure}[H]
\begin{center}
\includegraphics[width=13cm]{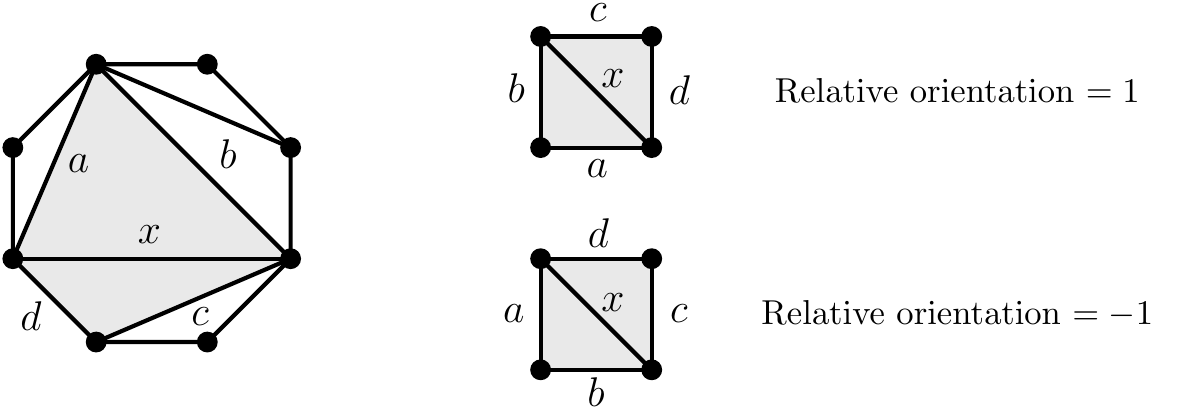}
\caption{An illustration of relative orientation}
\end{center}
\end{figure}

\begin{defn}
\label{arcsnakegraph}

Let $\gamma$ be a directed arc in $(S,M)$, and chose a lift $\overline{\gamma}$ in $\overline{(S,M)}$. Following the set-up of Definition \ref{relativeorientation}, we associate a tile $G_i$ for each intersection point, $p_i$, of $\overline{\gamma}$ with $\overline{T}$, such that $rel(G_i) \neq rel(G_{i+1})$. To define a snake graph $(G_1,\ldots, G_d)$ we just need to decide how these tiles are glued together. \newline \indent To explain the glueing process, note that $\tau_{i_j}$ and $\tau_{i_{j+1}}$ form two sides of the triangle $\Delta_j$; we denote the remaining side by $\tau_{[i_j]}$. The \textbf{\textit{snake graph of $\gamma$ with respect to $T$}} is then defined by glueing $G_j$ to $G_{j+1}$ along $\tau_{[i_j]}$. \newline \indent Up to isomorphism, this snake graph is independent of choice of lift $\overline{\gamma}$, consequently, we denote it by $\mathcal{G}_{\gamma, T}$

\end{defn}

\begin{figure}[H]
\begin{center}
\includegraphics[width=10cm]{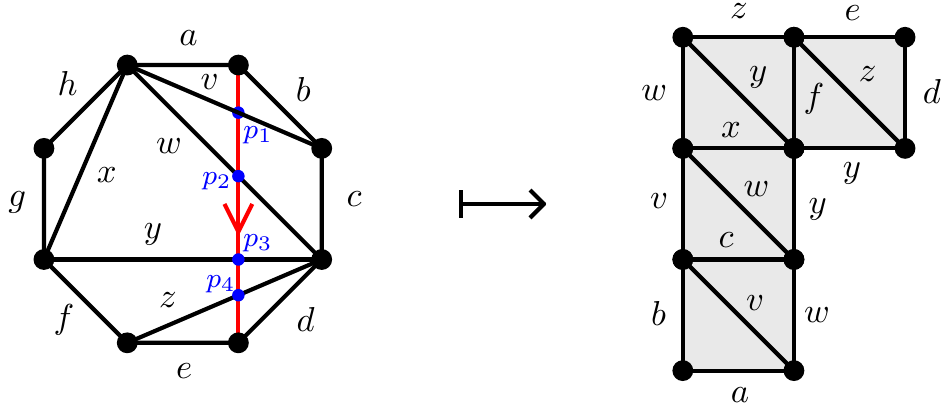}
\caption{The construction of a snake graph from an arc.}
\end{center}
\end{figure}

\subsection{Band graphs associated to one-sided closed curves}

Let $T$ denote a triangulation of $(S,M)$ and let $\overline{T}$ denote the lifted triangulation in the orientable double cover $\overline{(S,M)}$.

\begin{defn}
\label{surfacebandgraph}

Consider a one-sided closed curve $\alpha$ in $(S,M)$ and fix an orientation on $\alpha$ as well as a point $x \in \alpha$, such that $x$ is not a point of any arc in $T$. Note that $\alpha$ lifts to an oriented (two-sided) closed curve $\overline{\alpha}$ in $\overline{(S,M)}$, and $x$ lifts to two points $\overline{x}$ and $\tilde{\overline{x}}$. \newline
\indent Let $\mathcal{G}_{\alpha, T, x} = (G_1,\ldots, G_d)$ denote the snake graph obtained from following the curve from $\overline{x}$ to $\tilde{\overline{x}}$ along $\overline{\alpha}$, under the procedure described in Definition \ref{arcsnakegraph}. Note that if we continue along $\overline{\alpha}$, and apply the procedure to the next quadrilateral, we obtain a snake graph $(G_1,\ldots, G_d, G_{d+1})$ where $G_1$ and $G_{d+1}$ are the two lifts of a quadrilateral in $T$. As such, if we denote the glued edge of $G_d$ and $G_{d+1}$ by $\overline{b}$, then $\tilde{\overline{b}} \in \{S(G_1),W(G_1)\}$ and $sgn(\overline{b}) \neq sgn(\tilde{\overline{b}})$.

The \textbf{\textit{band graph of $\alpha$ with respect to $T$}} is defined as the band graph $\mathcal{G}_{\alpha, T, x}^{\overline{b}}$ (see Definition \ref{bandgraph}). Up to isomorphism, this band graph is independent of starting point $x$, consequently, we denote it by $\mathcal{G}_{\alpha, T}$.

\end{defn}

\begin{figure}[H]
\label{M4lift}
\begin{center}
\includegraphics[width=9cm]{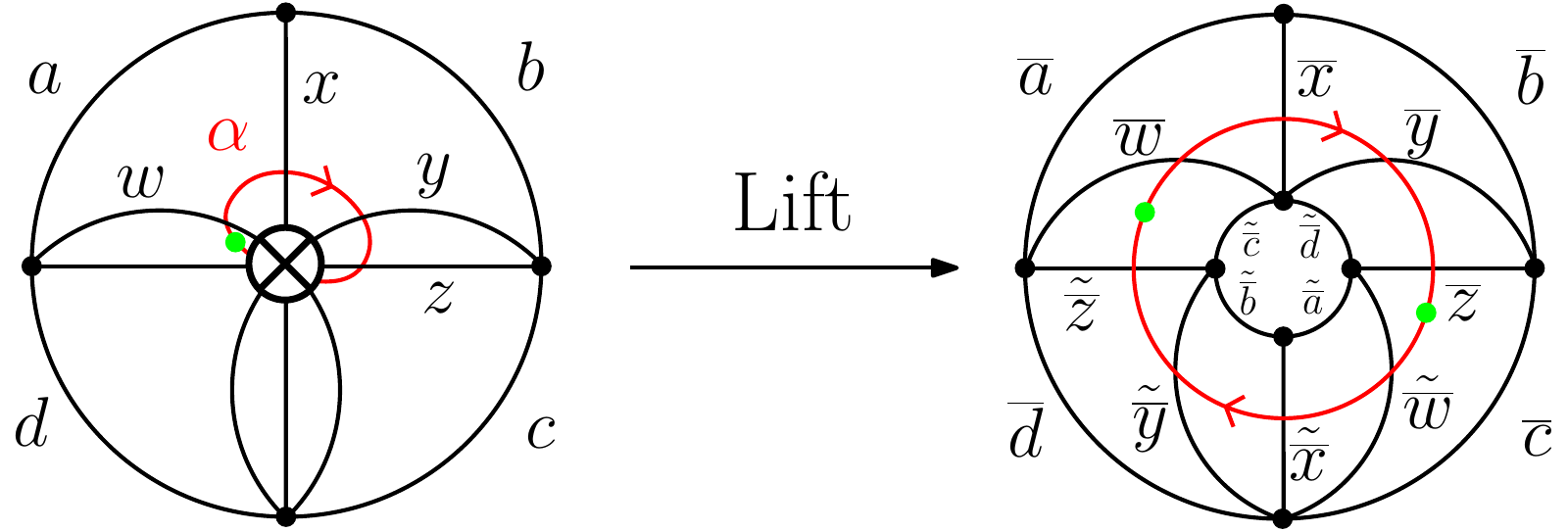}
\caption{The lift of a triangulation $T$ and an oriented one-sided closed curve $\alpha$.}
\end{center}
\end{figure}

\begin{figure}[H]
\label{bandgraph}
\begin{center}
\includegraphics[width=7cm]{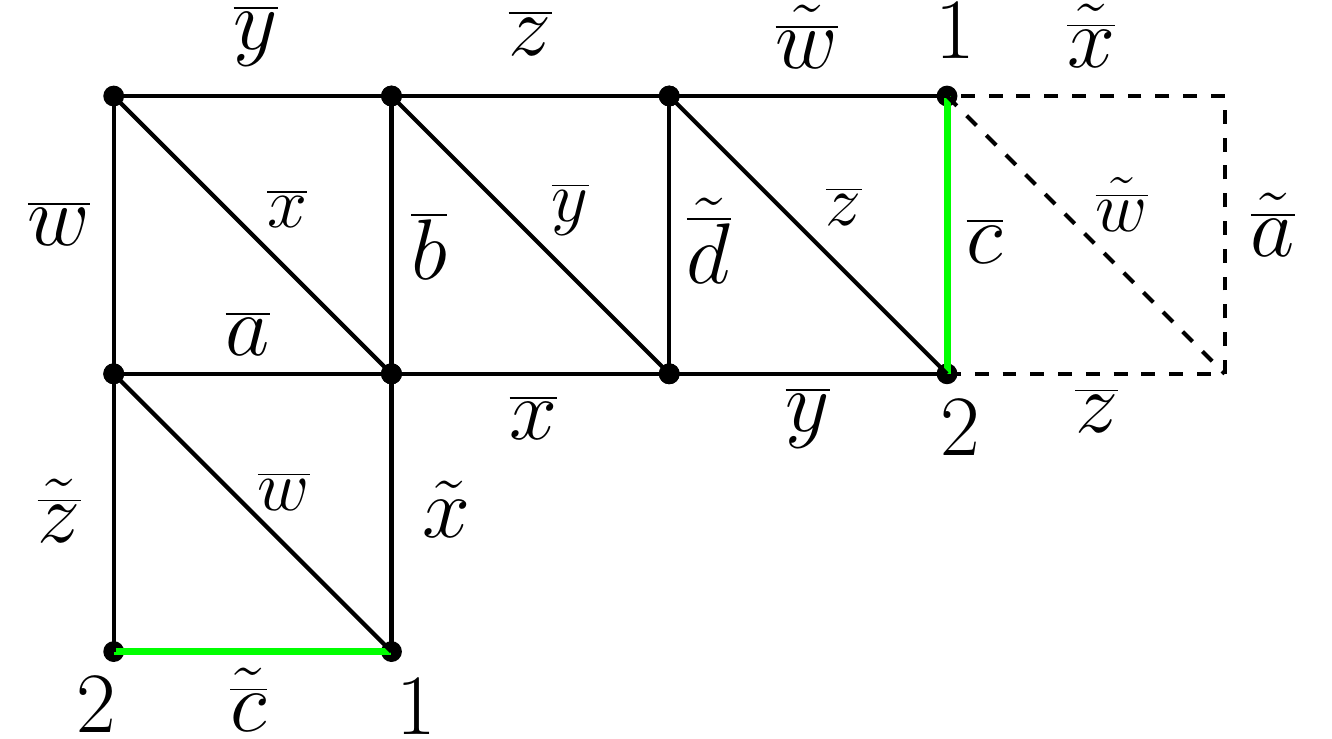}
\caption{The construction of a band graph with respect to $\alpha$ and $T$ found in Figure \ref{M4lift}.}
\end{center}
\end{figure}

\section{Expansion formulae for quasi-arcs with respect to triangulations without self-folded triangles}

\subsection{Expansion formulae for coefficient-free quasi-cluster algebras}

\begin{defn}

Let $T$ be an ideal triangulation. For a directed arc $\gamma$ not in $T$, let $\gamma_{i_1}, \ldots, \gamma_{i_d}$ be the sequence of arcs in $T$ which $\gamma$ intersects. The \textit{\textbf{crossing monomial}} of $\alpha$ with respect to $T$ is defined as: $$cross(\gamma, T) := {\displaystyle \prod_{j=1}^{d} x_{\gamma_{i_j}}}.$$

\noindent If $\gamma$ is an arc in $T$ then $cross(\gamma, T) := \frac{1}{x_{\gamma}}$.

\end{defn}

\begin{defn}

Let $T$ be an ideal triangulation without self-folded triangles and let $\gamma$ be a quasi-arc not in $T$. For $P \in \mathcal{P}_{T,\gamma}$ we define the \textit{\textbf{weight monomial}} $x(P)$ as follows:

$$x(P) := {\displaystyle \prod_{\gamma_i \in P} x_{\gamma_i}}.$$

\noindent If $\gamma$ is an arc in $T$ then $\mathcal{P}_{T,\gamma} = \{*\}$ and $x(*) := 1$.

\end{defn}

The following theorem was obtained by Musiker, Schiffler and Williams in the setting of orientable surfaces [Theorem 4.10, \cite{musiker2011positivity}]. However, the result easily extends to non-orientable surfaces. Indeed, any arc $\gamma$ on a non-orientable surface $(S,M)$ lifts to two arcs $\overline{\gamma}$ and $\tilde{\overline{\gamma}}$ on the orientable double cover $\overline{(S,M)}$. Moreover, each triangulation $T = \{\gamma_1,\ldots, \gamma_n\}$ of $(S,M)$ lifts to the the triangulation $\overline{T} = \{\overline{\gamma}_1, \tilde{\overline{\gamma}}, \ldots, \overline{\gamma}_1, \tilde{\overline{\gamma}}\}$. Therefore, for any arc $\alpha$ of $(S,M)$ we may obtain an expansion formula for $x_{\overline{\alpha}}$ (and of course $x_{\tilde{\overline{\alpha}}}$) in terms of $x_{\overline{\gamma}_1},x_{\tilde{\overline{\gamma}}_1}, \ldots, x_{\overline{\gamma}_n}, x_{\tilde{\overline{\gamma}}_n}$. Making the specialisations $x_{\overline{\gamma}_1} = x_{\tilde{\overline{\gamma}}_1}, \ldots, x_{\overline{\gamma}_n} = x_{\tilde{\overline{\gamma}}_n}$ recovers the geometry of $(S,M)$, whence providing the desired expansion formula of $x_{\alpha}$ in terms of $x_{\overline{\gamma}_1}, \ldots, x_{\overline{\gamma}_n}$.

\begin{thm}
\label{coefficientfreeexpansionarcs}
Let $T$ be a triangulation, of a bordered surface $(S,M)$, without self-folded triangles. Then for any arc $\alpha$ in $(S,M)$ (including loops around punctures) we have:

$$x_{\alpha} = \frac{1}{cross(T,\alpha)}\sum_{P \in \mathcal{P}_{T,\alpha}} x(P). $$

\end{thm}

We now embark on the task of finding expansion formulae for one-sided closed curves.

\begin{defn}

Let $\gamma \in (S,M)$ be an arc enclosing some $M_1$ and $\delta \in (S,M)$ be a quasi-arc which has non-trivial intersection with $\gamma$. There may be many intersections between $\delta$ and $\gamma$, however they split into combinations of three natural types (see Figure):

\begin{itemize}

\item (Type I) $\delta$ intersects $\gamma$ twice and passes through the crosscap of $M_1$. 

\item (Type II) $\delta$ intersects $\gamma$ twice but does not pass through the crosscap of $M_1$.

\item (Type III) $\delta$ intersects $\gamma$ once (hence it has an endpoint on $M_1$ and passes through the crosscap).

\end{itemize}

\end{defn}

\begin{figure}[H]
\begin{center}
\includegraphics[width=12cm]{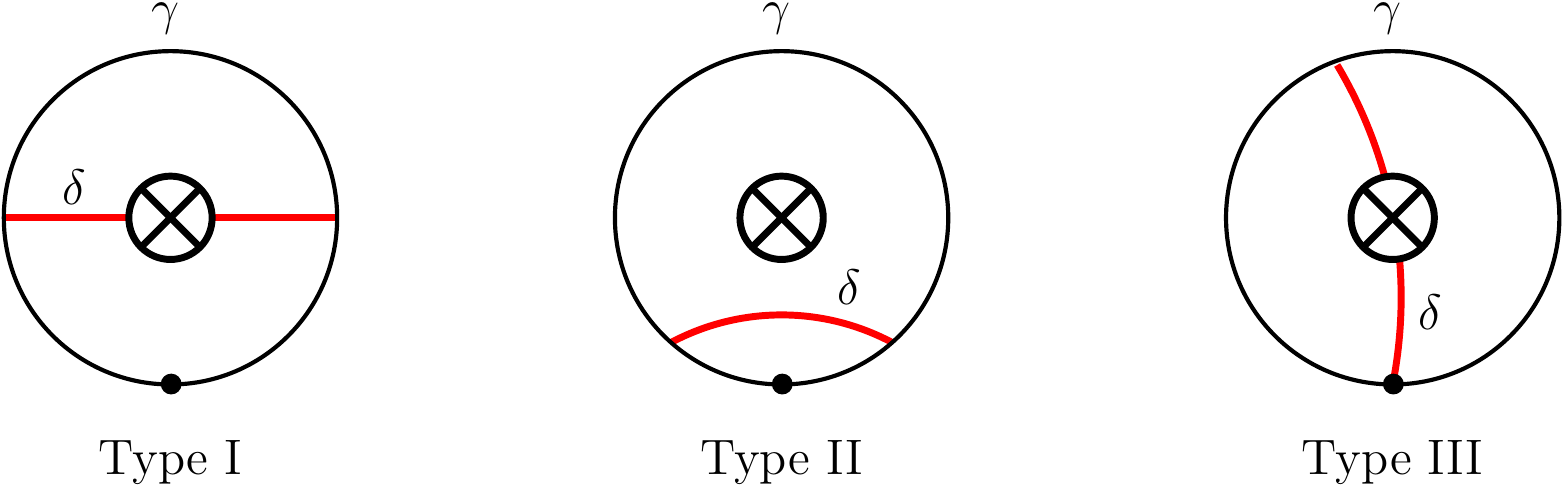}
\caption{The three types of intersections of an arc $\delta$ and an arc $\gamma$ bounding a M\"{o}bius strip with 1 marked point.}
\end{center}
\end{figure}

\begin{rmk}

Note that a quasi-arc $\delta$ can have varying types of intersections with $\gamma$, however it cannot have both Type II and Type III intersections.

\end{rmk}

\begin{lem}
\label{arcexistence}
Let $T$ be an ideal triangulation of $(S,M)$ and $\alpha$ a one-sided closed curve. Then there exists an arc $\gamma$ enclosing $\alpha$ in $M_1$ such that there are no arcs $\delta \in T$ having a Type II intersection with $\gamma$.

\end{lem}

\begin{proof}

Since $\alpha \notin T$ there exists $\eta \in T$ which intersects $\alpha$. After fixing some orientation on $\eta$, let $x$ denote the first (and possibly only) intersection point of $\eta$ with $\alpha$. Define $\beta$ to be the arc isotopic to the following curve:

$$s(\eta) \xrightarrow{\text{$\eta$}} x \xrightarrow{\text{$\alpha$}} x \xrightarrow{\text{$\eta^{-1}$}} s(\eta)$$

where $s(\eta)$ denotes the starting endpoint of $\eta$, and $\eta^{-1}$ indicates we are travelling against the orientation on $\eta$. Define $\gamma$ to be the (unique) arc enclosing $\alpha$ and $\beta$ in some $M_1$ (see figure). The lemma is then proved by noting any arc $\delta \in T$ cannot have Type II intersection with $\gamma$, since this would imply $\delta$ intersects $\eta \in T$.

\end{proof}

Recall that our aim is to prove the expansion formula for $x_{\alpha}$ by embedding the respective band and snake graphs of $\alpha$ and $\beta$ into the snake graph of $\gamma$. The previous lemma tells us that for a given triangulation $T$ and a one-sided closed curve $\alpha$, we can choose $\gamma$ (or equivalently $\beta$) sufficiency nicely. In particular, under this choice, we are able to easily describe $S_{\gamma,\overline{T}^{\circ}}$ in terms of $S_{\alpha,\overline{T}^{\circ}}$ and $S_{\beta,\overline{T}^{\circ}}$.

\begin{lem}
\label{arcproperty}
Let $T$ be an ideal triangulation and $\gamma$ be an arc enclosing some $M_1$ such that no $\delta \in T$ has Type II intersections with $\gamma$. After fixing an orientation on $\gamma$, let $$(\gamma_{i_1}, \ldots, \gamma_{i_d})$$ be all arcs in $T$ that intersect $\gamma$ (including multiplicity), which are listed in order of their intersections.

Then there exist $s,t \in \{1,\ldots, d\}$ such that:

\begin{itemize}

\item $s+t = d$

\item $\{\gamma_{i_{s+1}}, \ldots, \gamma_{i_{t}}\}$ are arcs whose corresponding intersection points $p_{s+1}, \ldots, p_{t}$ are Type III intersections.

\item $\{\gamma_{i_{1}} = \gamma_{i_{t+1}}, \ldots, \gamma_{i_{s}} = \gamma_{i_{t+s}}\}$ are arcs whose corresponding intersection points $(p_{1},p_{t+1}) \ldots, (p_{s}, p_{t+s})$ form Type I intersections.

\end{itemize}

\end{lem}

\begin{proof}

This follows directly from the assumption that there are no Type II intersections. In particular, $s$ and $t-s$ are the number of Type I and Type III intersections, respectively.

\end{proof}

\begin{figure}[H]
\begin{center}
\includegraphics[width=11cm]{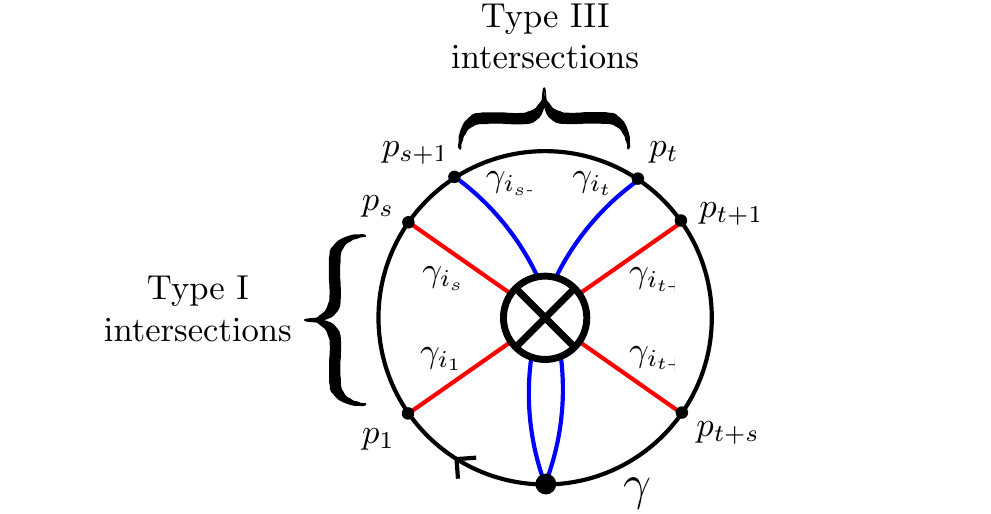}
\caption{The form of the ideal triangulation discussed in Lemma \ref{arcproperty}. The red arcs and blue arcs indicate the arcs in $T$ with type I and type III intersections, respectively.}
\end{center}
\end{figure}

Lemma \ref{arcproperty} has two immediate corollaries.

\begin{cor}
\label{3tiles}

Following the setup of Lemma \ref{arcexistence} and Lemma \ref{arcproperty}, in the snake graph $S_{\gamma,\overline{T}}$ we have the following:

\begin{itemize}

\item If $s+1 \neq t$ then the tiles corresponding to $\gamma_{i_{s}},\gamma_{i_{s+1}}, \gamma_{i_{s+2}}$ form a straight line in $S_{\gamma,\overline{T}^{\circ}}$. Likewise, the tiles corresponding to $\gamma_{i_{t-1}},\gamma_{i_{t}}, \gamma_{i_{t+1}}$ also form a straight line in $S_{\gamma,\overline{T}^{\circ}}$. \newline Moreover, the consecutive tiles corresponding to $\gamma_{i_{s+1}}, \ldots, \gamma_{i_{t}}$ form a zig-zag in $S_{\gamma,\overline{T}^{\circ}}$.

\item If $s+1 = t$ then the tiles corresponding to $\gamma_{i_s},\gamma_{i_{s+1}}=\gamma_{i_{t}}$, and $\gamma_{i_{t+1}}$ form a zig-zag in $S_{\gamma,\overline{T}^{\circ}}$.

\end{itemize}

\end{cor}

\begin{proof}

When $s+1 \neq t$, the first statement follows from the fact $\gamma_{i_{t+s}}$ corresponds to a Type I intersection. So the sequence of arcs $(\gamma_{i_{t+s}}=\gamma_{i_{s}},\gamma_{i_{s+1}}, \gamma_{i_{s+2}})$ form a zig-zag in $\overline{T}^{\circ}$. Hence their corresponding tiles in $S_{\gamma,\overline{T}^{\circ}}$ form a straight line. Similarly, for the second statement, since $\gamma_{i_{1}}$ corresponds to a Type I intersection, then $(\gamma_{i_{t-1}},\gamma_{i_{t}}, \gamma_{i_{t+1}} =\gamma_{i_{1}})$ form a zig-zag in $S_{\gamma,\overline{T}^{\circ}}$. \newline \indent Moreover, the consecutive tiles corresponding to $\gamma_{i_{s+1}}, \ldots, \gamma_{i_{t}}$ form a zig-zag in $S_{\gamma,\overline{T}}$ since $\gamma_{i_{s+1}}, \ldots, \gamma_{i_{t}}$ form a fan in $S_{\gamma,\overline{T}^{\circ}}$\newline

When $s+1 = t$ then $(\gamma_{i_{t+s}}=\gamma_{i_{s}},\gamma_{i_{s+1}} = \gamma_{i_{t}}, \gamma_{i_{t+1}} =\gamma_{i_{1}})$ form a fan in $T$. Hence their corresponding tiles in $S_{\gamma,\overline{T}^{\circ}}$ form a zig-zag.

\end{proof}

\begin{cor}
\label{crossings}
Fix an ideal triangulation $T$ and a one-sided closed curve $\alpha$. Let $\beta$ and $\gamma$ be the arcs guaranteed by Lemma \ref{arcexistence}. Then the following equality holds:

$$cross(T,\gamma) = cross(T,\alpha)cross(T,\beta)$$

\begin{rmk}

Corollary \ref{crossings} of course holds for any arc $\gamma$ (and accompanying arc $\beta$) which encloses $\alpha$ in some $M_1$. However, we shall only need the version stated above.

\end{rmk}

\end{cor}

\underline{\textbf{Notation}}: Given a tile T in a snake graph we denote by $N(T), E(T), S(T), W(T)$ the north, east, south and west edges in $T$, respectively.

\begin{lem}[Zig-zag lemma]
\label{zigzag}

Let $Z_n$ be a zig-zag snake graph with $n$ tiles $G_1, \ldots, G_n$. Consider the $sgn$ function on $Z_n$ such that the sign of each glued edge $e_i$ is $-$.

Define $x$ to be the unique edge in $G_1$ such that $x \in \{W(G_1), S(G_1)\}$ and $sgn(x) = +$.
Similarly, define $y$ to be the unique edge in $G_n$ such that $y \in \{E(T_1), N(T_1)\}$ and $sgn(y) = +$.

Then any perfect matching of $Z_n$ contains $x$ or $y$.

\end{lem}

\begin{proof}

We apply induction on $n$. The base case of $n=1$ trivially holds since $x$ and $y$ share a common vertex, so assume the result is true for $Z_{n-1}$ where $n>1$.

Given a perfect matching $P$ of $Z_n$, if $x \notin P$ then either $N(G_1), S(G_1) \in P$ (when $x = W(G_1)$) or $W(G_1), E(G_1) \in P$ (when $x = S(G_1)$). Hence $P$ descends to a perfect matching $P'$ of $Z_{n-1}$ -- the sub snake graph of $Z_n$ on tiles $G_2, \ldots, G_n$. Specifically, if $x = W(G_1)$ then $$P' := P\setminus \{N(G_1), S(G_1)\} \cup \{W(G_2)\}.$$ Likewise, if $x = E(G_1)$ then $$P' := P\setminus \{W(G_1), E(G_1)\} \cup \{S(G_2)\}.$$ Consequently, the edge in $\{W(G_2), S(G_2)\}$ with sgn '$+$' is not in $P'$. So by induction, $y \in P$.

\end{proof}

\underline{\textbf{Notation}}: Given a snake graph $G$ with tiles labelled by ${i_1}, \ldots {i_d}$ then:

\begin{itemize}

\item $G_{[j,k]}$ denotes the sub snake graph of $G$ with consecutive tiles $i_j, i_{j+1} \ldots i_{k-1}, i_k$.

\end{itemize}

Moreover, for any $r \in \{1,\ldots, d\}$ let $e_r$ denote the shared edge of the tiles corresponding to $i_r$ and $i_{r+1}$. Using this notation:

\begin{itemize}

\item For $j >1$, $G_{(j,k]} := G_{[j,k]}\setminus \{e_{j-1}\}$.

\item For $k < d$, $G_{[j,k)} := G_{[j,k]}\setminus \{e_{k+1}\}$.

\end{itemize}

\begin{rmk}

For a snake graph $G$ and edge $e$, $G\setminus \{e\}$ indicates that we have removed the vertices of the edge $e$, and consequently any edges in $G$ connected to either of these vertices. Therefore, $G_{(j,k]}$ and $G_{[j,k)}$ may not actually be snake graphs. Nevertheless, since they arise from $G_{[j,k]}$, our snake graph terminology still makes sense.

\end{rmk}

\begin{figure}[H]
\begin{center}
\includegraphics[width=13cm]{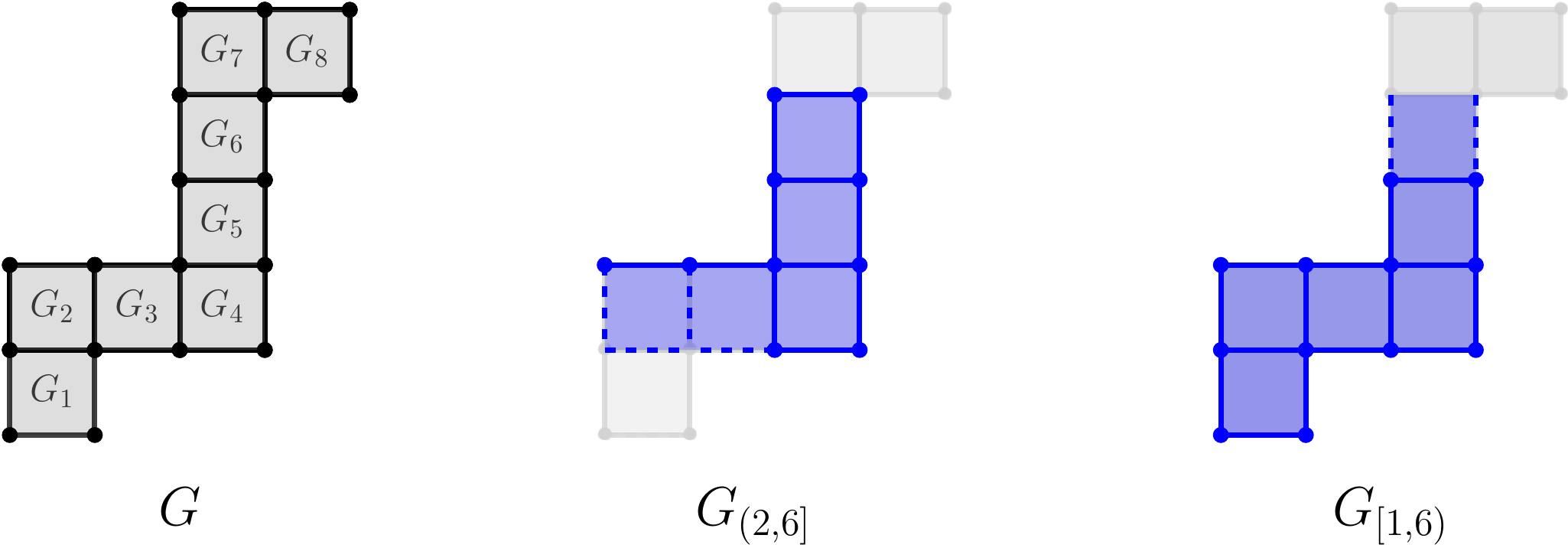}
\caption{An illustration of the subgraphs $G_{(2,6]}$ and $G_{[1,6)}$ of a snakegraph $G = (G_1, \ldots, G_8)$. The dotted lines indicate which edges are not included in the respective graphs. However, as indicated by the blue shading, we consider $G_{(2,6]}$ and $G_{[1,6)}$ arising from subsnakegraphs $G_{[2,6]}$ and $G_{[1,6]}$, respectively. This viewpoint is needed later when considering coefficients, since tiles will encode frozen variables -- so it is important we do not `forget' them.}
\end{center}
\end{figure}

\begin{defn} 

Let $G$ be a snake graph. We let $inv(G)$ denote the reflection of $G$ through the line $x =y$ (with respect to the $x$ and $y$ axes).

\end{defn}

\begin{prop}
\label{decomposition}

Let $\alpha$ be any one-sided closed curve and $T$ an ideal triangulation without self-folded triangles. Choose $\gamma$ (and $\beta$) as in Lemma \ref{arcexistence}, and adhere to the notation used in \ref{arcproperty}.

\begin{itemize}

\item If $\beta$ is \underline{not} in $T$ then:

\[  (S_{\gamma,\overline{T}^{\circ}})_{[1,s]} = S_{\beta,\overline{T}^{\circ}} \hspace{10mm} \text{and} \hspace{10mm} (S_{\gamma,\overline{T}^{\circ}})_{(s+1,s+t]}  = \left\{
\begin{array}{ll}
        S_{\alpha,\overline{T}^{\circ}} \hspace{2mm}, & s \hspace{3mm} \text{is even} \\ 
        inv(S_{\alpha,\overline{T}^{\circ}}) \hspace{2mm}, & s \hspace{3mm} \text{is odd} \\
\end{array} 
\right. \]

\[  (S_{\gamma,\overline{T}^{\circ}})_{[1,t)} = S_{\alpha,\overline{T}^{\circ}} \hspace{10mm} \text{and} \hspace{10mm} (S_{\gamma,\overline{T}^{\circ}})_{[t+1,t+s]}  = \left\{
\begin{array}{ll}
        S_{\tilde{\beta},\overline{T}^{\circ}} \hspace{2mm}, &t \hspace{3mm} \text{is even} \\ 
        inv(S_{\tilde{\beta},\overline{T}^{\circ}}) \hspace{2mm}, &t \hspace{3mm} \text{is odd} \\
\end{array} 
\right. \]

\item If $\beta$ is an arc in $T$ then $s=0$ and $S_{\alpha,\overline{T}^{\circ}} = (G_{\gamma_{i_1}}, \ldots, G_{\gamma_{i_t}}, G_{\gamma_{i_{t+1}}} = G_{\beta})$ has one more tile than $S_{\gamma,\overline{T}^{\circ}} = (G_{\gamma_{i_1}}, \ldots, G_{\gamma_{i_{t+1}}})$; corresponding to $\alpha$'s intersection with $\beta$). Specifically, $$S_{\gamma,\overline{T}^{\circ}} = (S_{\alpha,\overline{T}^{\circ}})_{[1,t]}$$.

\end{itemize}

\end{prop}

\begin{proof}

This follows directly from the geometry of $T, \alpha, \beta$ and $\gamma$.

\end{proof}

\begin{rmk}

Note that Proposition \ref{decomposition} tells us that, when $\beta$ is not in $T$, $S_{\gamma,\overline{T}^{\circ}}$ decomposes into the snake and band graphs $S_{\alpha,\overline{T}^{\circ}}$ and $S_{\beta,\overline{T}^{\circ}}$ in two canonical ways. Moreover, Corollary \ref{3tiles} indicates precisely how they fit together in $S_{\gamma,\overline{T}^{\circ}}$.

\end{rmk}

\begin{figure}[H]
\label{liftabc}
\begin{center}
\includegraphics[width=12cm]{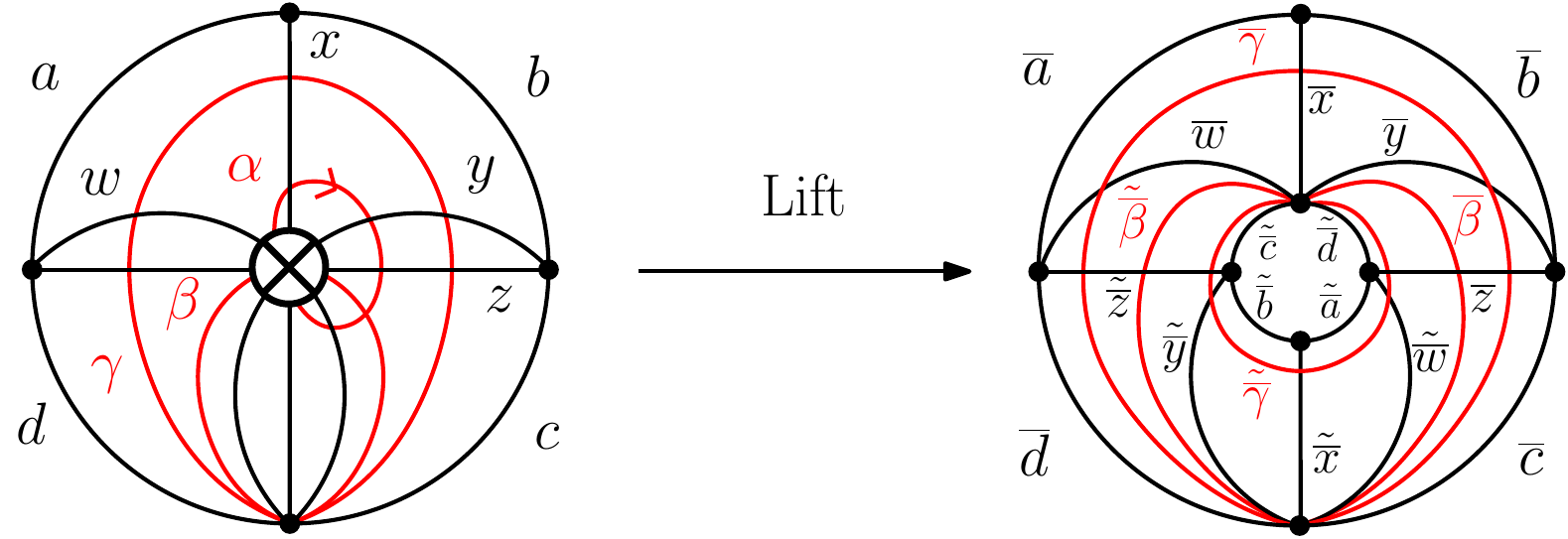}
\caption{The lift of a triangulation $T$ as well as the arcs $\beta$ and $\gamma$.}
\end{center}
\end{figure}

\begin{figure}[H]
\begin{center}
\includegraphics[width=15cm]{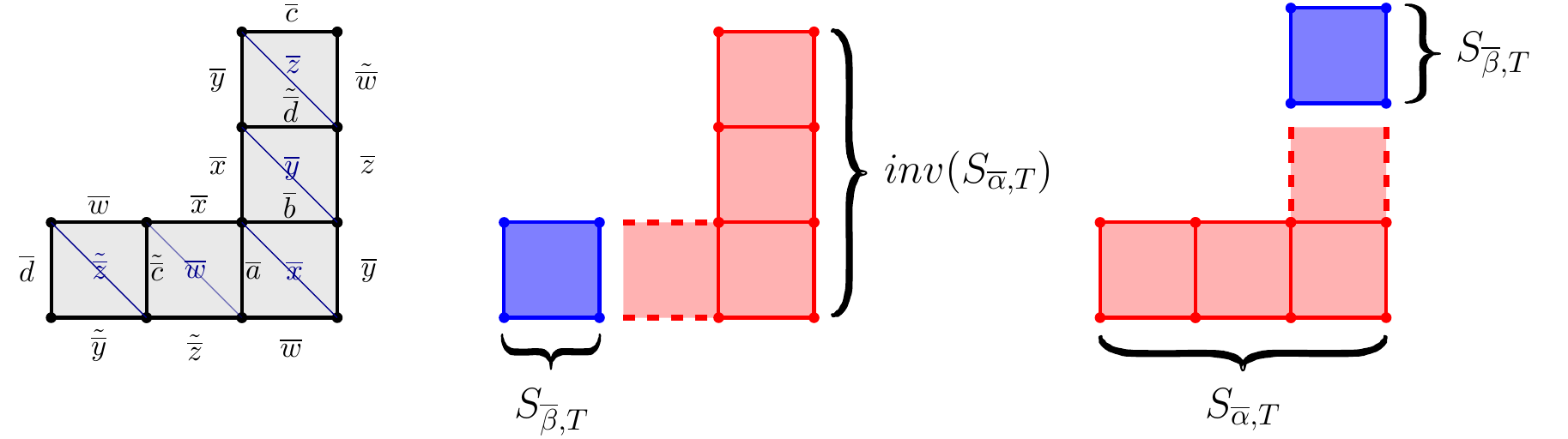}
\caption{The snake graph decomposition of $S_{\gamma,T}$, with respect to the triangulation in Figure \ref{liftabc}.}
\end{center}
\end{figure}

\underline{\textbf{Notation}}: Let $P$ be a perfect matching of a snake graph $G$ with tiles labelled by ${i_1}, \ldots {i_d}$. Following the spirit of the notation used above, if $P$ restricts to a perfect matching on consecutive tiles $i_{j}, \ldots, i_{k}$ (i.e. restricts to a perfect matching on $G_{[j,k]}$) then we denote this perfect matching by $P_{[j,k]}$. \indent
Similarly, if $P$ restricts to a perfect matching on $G_{(j,k]}$ (resp. $G_{[j,k)}$) then we denote this perfect matching by $P_{(j,k]}$ (resp. $P_{[j,k)}$).

\begin{defn}
For a quasi-arc $\gamma$ and an ideal triangulation $T$ without self-folded triangles, let $\mathcal{P}_{T,\gamma}$ denote the set of perfect (resp. good) matchings of the associated snake (resp. band) graph $S_{\gamma,T}$. If $\gamma$ is in $T$ then we formally set $\mathcal{P}_{T,\gamma} := \{*\}$ as a singleton set.

\end{defn}

\begin{prop}
\label{welldefined}
Fix an ideal triangulation $T$ without self-folded triangles. Let $\alpha$ be a one-sided closed curve and let $\beta$ and $\gamma$ be the arcs guaranteed by Lemma \ref{arcexistence}. Then the following map is well defined:

\begin{align*}
 \Phi \hspace{1mm}: \hspace{2mm} \mathcal{P}_{T,\gamma} &\longrightarrow \mathcal{P}_{T,\alpha} \times \mathcal{P}_{T,\beta}\\
  P \hspace{3mm} & \hspace{1mm} \mapsto \hspace{7mm} (P_1 \hspace{1mm},\hspace{0.5mm}P_2).
\end{align*}

where, \begin{itemize}

\item if $\beta$ is not in $T$ then \[  (P_1 \hspace{1mm},\hspace{0.5mm}P_2)  := \left\{
\begin{array}{ll}
        (\hspace{1mm} P_{[1,s]} \hspace{1mm},\hspace{0.5mm}P_{(s+1,t]} \hspace{1mm}), &\text{if there is a left or centre cut at $e_s$ in $P$}. \\
        \\
        (\hspace{1mm} P_{[t+1,t+s]} \hspace{1mm},\hspace{0.5mm}P_{[1,t)} \hspace{1mm}), &\text{if there is a right cut at $e_s$ in $P$}.\\
\end{array} 
\right. \]

\item if $\beta$ is an arc in $T$, then $S_{\gamma,\overline{T}^{\circ}} = (S_{\alpha,\overline{T}^{\circ}})_{[1,t]}$ and $\mathcal{P}_{T,\beta} = \{*\}$. In which case, $$(P_1 \hspace{1mm},\hspace{0.5mm}P_2)  := (P, *).$$

\end{itemize}

\end{prop}

\begin{proof}

We will split the proof into two cases. \newline 

\textbf{Case 1}: $\beta$ is not in $T$. When there is a left or centre cut at $e_s$ in $P$, then $P_{[1,s]} \in \mathcal{P}_{T,\beta}$ by restriction of perfect matchings. Moreover, $P_{(s+1,s+t]}$ induces a good perfect matching of $S_{\alpha, T}$, since $S_{\alpha, T}$ is the band graph obtained from gluing the snake graph $(S_{\gamma,T})_{[s+1,s+t]}$ along $e_s$, and the vertices of $e_s$ are unmatched in $P_{[s+1,s+t]}$. \newline

\indent Similarly, when there is a right cut at $e_s$ in $P$ then, by Corollary \ref{3tiles} and Lemma \ref{zigzag}, we know there must be a right or centre cut at $e_t$. Hence $P_{[t+1,s+t]} \in \mathcal{P}_{T,\beta}$ by restriction. Analogous to the paragraph above, $P_{(1,t]}$ induces a good perfect matching on $S_{\alpha, T}$ since $S_{\alpha, T}$ is obtained from gluing $(S_{\gamma,T})_{[1,t]}$ along $e_t$, and the vertices of $e_t$ in $P_{[1,t]}$ are unmatched. \newline

\textbf{Case 2}: $\beta$ is an arc in $T$. In this case $s=0$, and $\alpha$ has one more intersection with $\beta$ than $\gamma$ does. Consequently, $S_{\alpha, T} = (G_1, \ldots, G_{t+1})$ has one more tile than $S_{\gamma, T} = (G_1, \ldots, G_t)$. It suffices to show the glued edge in $S_{\alpha, T}$ does not contain the north-west edge of the tile $G_t$. This follows from the fact $\gamma_{i_t}, \beta, \gamma_{i_1}$ form a zig-zag in $T$.

\end{proof}

\begin{prop}
\label{splitproduct}
The map $\Phi$ above is a bijection. Moreover, $$x(P) = x(P_1)x(P_2).$$

\end{prop}

\begin{proof}

If $\beta$ is in $T$ then it suffices to show all good matchings of $S_{\alpha,T}$ have right or centre cuts at $e_{t+1}$. This follows from the Zig-zag Lemma \ref{zigzag}. \newline \indent So suppose $\beta$ is not an arc in $T$. Let $\mathcal{P}_{T,\gamma}^{l,c}$ denote all the perfect matchings in $\mathcal{P}_{T,\gamma}$ with a left or centre cut at $e_s$. We see:

$$ \Phi(\mathcal{P}_{T,\gamma}^{l,c}) = \text{ \Bigg\{ \parbox{12em}{all good matchings\\ of $S_{\alpha,\overline{T}^{\circ}}$ with a left\\ or centre cut at $e_s$.} \hspace{-15mm} \Bigg\} $ \times \hspace{2mm} \mathcal{P}_{\beta,T}$}$$

Similarly, let $\mathcal{P}_{T,\gamma}^{r}$ denote all the perfect matchings in $\mathcal{P}_{T,\gamma}$ with a right cut at $e_s$. By Corollary \ref{3tiles} and Lemma \ref{zigzag}, we know there must be a right or centre cut at $e_t$, which consequently induces a perfect matching on $(S_{\gamma,\overline{T}^{\circ}})_{[t+1,t+s]}$. Moreover, in this way, we can obtain all perfect matchings of $(S_{\gamma,\overline{T}^{\circ}})_{[t+1,t+s]}$. Hence we see:

$$ \Phi(\mathcal{P}_{T,\gamma}^{r}) = \text{ \Bigg\{ \parbox{12em}{all good matchings\\ of $S_{\alpha,\overline{T}^{\circ}}$ with a\\ right cut at $e_s$.} \hspace{-15mm} \Bigg\} $ \times \hspace{2mm} \mathcal{P}_{\beta,T}$}$$

Therefore, $\Phi$ is surjective. To recognise injectivity, note that $\Phi$ is trivially injective on $\mathcal{P}_{T,\gamma}^{l,c}$ and $\mathcal{P}_{T,\gamma}^{r}$. And since their images are disjoint, $\Phi$ is indeed injective. \newline

The property $x(P) = x(P_1)x(P_2)$ follows directly from the definition of $\Phi$.

\end{proof}

\begin{thm}
\label{coefficientfreeexpansion}
Let $T$ be an ideal triangulation of a bordered surface $(S,M)$ which has no self-folded triangles. Then for any one-sided closed curve $\alpha$ we have:

$$x_{\alpha} = \frac{1}{cross(T,\alpha)}\sum_{P \in \mathcal{P}_{T,\alpha}} x(P). $$

\end{thm}

\begin{proof}

Recall that for any $\gamma$ enclosing $\alpha$ in $M_1$, with regards to the unique arc $\beta$ contained in $M_1$, we have the relation $$x_{\gamma} = x_{\alpha}x_{\beta}.$$ Moreover, by Theorem \ref{coefficientfreeexpansionarcs} we know,

$$ \frac{x_{\alpha}}{cross(T,\beta)}\sum_{P \in \mathcal{P}_{T,\beta}} x(P) = x_{\alpha}x_{\beta} = x_{\gamma} = \frac{1}{cross(T,\gamma)}\sum_{P \in \mathcal{P}_{T,\gamma}} x(P).$$

By Proposition \ref{splitproduct} we have:

$$\frac{1}{cross(T,\gamma)}\sum_{P \in \mathcal{P}_{T,\gamma}} x(P) = \frac{1}{cross(T,\gamma)}\Big(\sum_{P \in \mathcal{P}_{T,\alpha}} x(P)\Big)\Big(\sum_{P \in \mathcal{P}_{T,\beta}} x(P)\Big).$$

Corollary \ref{crossings} then completes the proof.

\end{proof}

\begin{exmp}

Let us continue with our example of the M\"{o}bius strip with four marked points on the boundary, $M_4$. Consider the one-sided closed curve $\alpha$ of $M_4$ and let $T$ denote the triangulation of $M_4$ given in Figure \ref{M4lift}. Theorem \ref{coefficientfreeexpansion} and Figure \ref{bandgraphexample} tell us that the expansion of $x_{\alpha}$ with respect to $T$ is: $$ x_{\alpha} = \frac{x_xx_zx_yx_d + x_cx_ax_yx_d + x_cx_wx_bx_d + x_ax_y^2x_w + x_wx_bx_yx_w + x_cx_wx_xx_z}{x_wx_xx_yx_z}$$

\end{exmp}

\begin{figure}[H]
\label{bandgraphexample}
\begin{center}
\includegraphics[width=14cm]{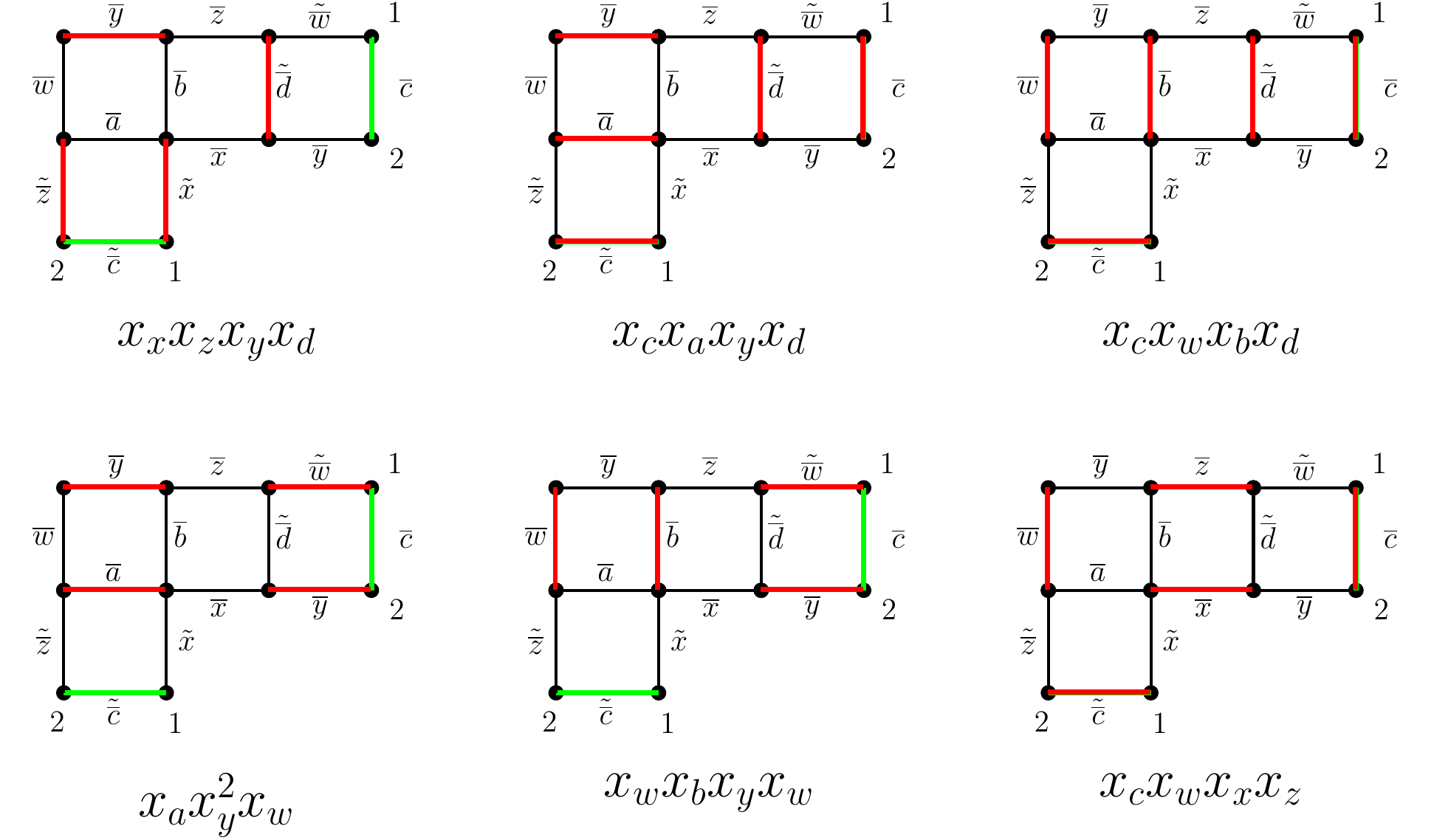}
\caption{All the good matchings of the band graph in Figure \ref{bandgraph} together with their associated weight monomials.}
\end{center}
\end{figure}

\subsection{Principal laminations on orientable marked surfaces.}

To obtain expansion formulae for quasi-cluster algebras with arbitrary coefficients, we shall first restrict our attention to orientable surfaces. In this section we work in the generality of punctured orientable surfaces $(S,M)$.

\begin{defn}[Principal laminations for orientable surfaces]
\label{principal lamination def}
Let $\gamma$ be a tagged arc in $(S,M)$. We define a lamination, $L_{\gamma}$, as follows:

\begin{itemize}

\item If $\gamma$ is a plain arc then $L_{\gamma}$ is taken to be a lamination that runs along $\gamma$ in a small neighbourhood thereof, which consistently spirals around the endpoints of $\gamma$ both clockwise (or anti-clockwise). For endpoints of $\gamma$ which are not punctures $L_{\gamma}$ cannot `spiral', instead we mean it turns clockwise (resp. anti-clockwise) at the marked point, and ends when it reaches the boundary.

\item If $\gamma$ is a tagged arc with some notched endpoints, $L_{\gamma}$ is defined as above, except now, at notched endpoints the direction of spiralling is reversed.

\end{itemize}

\end{defn}

\begin{rmk}

Note that following the rules listed in Definition \ref{principal lamination def} there are precisely two choices for $L_{\gamma}$ for each tagged arc $\gamma$.

\end{rmk}

\begin{defn}
A multi-lamination $\mathbf{L_T}$ of the form $\{L_{\gamma} | \gamma \in T\}$ is called a \textbf{principal lamination}.

\end{defn}

\begin{rmk}
\label{principalcoeffcients}
One may rephrase the above condition in terms of Thurston's shear coordinates. Namely, $\mathbf{L_T}$ is the multi-lamination such that $b_{\gamma}(L_{\gamma},T) = \pm1$ for each $\gamma \in T$. In particular, choosing a principal lamination $\mathbf{L_T}$ such that $b_{\gamma}(L_{\gamma},T) = 1$ for each $\gamma \in T$ recovers the notion of principal coefficients.

\end{rmk}

\begin{defn}
\label{alphaoriented}
Let $\mathbf{L_T}$ be a principal lamination and $\alpha$ a curve in $(S,M)$. A diagonal $\gamma_{i_k}$ of a tile in $S_{\alpha,T}$ is \textbf{$\mathbf{L_T}$-oriented} with respect to a good matching $P$ if: \vspace{3mm}

\begin{itemize}

\item The tile is \underline{odd} and either:

\begin{itemize}

\item $b_{T}(L_{\gamma_{i_k}},\gamma_{i_k}) = 1$ and the diagonal is oriented \underline{down}.

\item $b_{T}(L_{\gamma_{i_k}},\gamma_{i_k})  = -1$ and the diagonal is oriented \underline{up}.

\end{itemize}

\item The tile is \underline{even} and either:

\begin{itemize}

\item $b_{T}(L_{\gamma_{i_k}},\gamma_{i_k})  = 1$ and the diagonal is oriented \underline{up}.

\item $b_{T}(L_{\gamma_{i_k}},\gamma_{i_k})  = -1$ and the diagonal is oriented \underline{down}.

\end{itemize}

\end{itemize}

\end{defn}

\begin{defn} 
\label{ymonomial}
Given a good matching $P$ of $S_{\alpha,T}$ we define the following \textbf{coefficient monomial}:

\begin{center}
\large{$y_{\mathbf{L_T}}(P) := {\displaystyle \prod_{\gamma_{i_k} \text{is} \hspace{1mm} \mathbf{L_T}\text{-oriented}} y_{\gamma_{i_k}}}$}

\end{center}

\end{defn}

\begin{thm}[Theorem 4.9, \cite{musiker2011positivity}]
\label{orientedarcexpansioncoefficients}
Let $(S,M)$ be an orientable bordered surface and let $T$ be an ideal triangulation without self-folded triangles. If $\mathbf{L_{T}^{\bullet}}$ is the principal lamination corresponding to principal coefficients at $\Sigma_T$, then for any arc $\alpha$ we have:

\begin{equation}
x_{\mathbf{L_{T}^{\bullet}}}(\alpha) = \frac{1}{cross(\alpha, T)} {\displaystyle \sum_{P} x(P)y_{\mathbf{L_{T}^{\bullet}}}(P)}
\end{equation}

Where:

\begin{itemize}

\item The sum is over all perfect matchings of the snake graph $\mathcal{G}_{\alpha, T}$.

\item $x_{\mathbf{L_{T}^{\bullet}}}(\gamma)$ is the cluster variable (corresponding to $\gamma$) in the cluster algebra $\mathcal{A} := \mathcal{A}_{\mathbf{L_{T}^{\bullet}}}(S,M)$ with initial seed $\Sigma_T$.
 
 \end{itemize}

\end{thm}

\begin{prop}[Corollary 6.4, \cite{schiffler2010cluster} and Theorem 5.1, \cite{musiker2010cluster}]
\label{maximalmatchingtriangulation}
Following the set-up of Theorem \ref{orientedarcexpansioncoefficients}, let $S_{\alpha,T}$ be the snake graph of $\alpha$ with respect to $T$. Then there exists a unqiue perfect matching $P_{+}$ of $S_{\alpha,T}$ such that:

$$y_{\mathbf{L_{T}^{\bullet}}}(P_{+}) = \displaystyle \prod_{\substack{\text{$\gamma$ is a} \\ \text{diagonal of a}\\ \text{tile in $S_{\alpha,T}$}}} y_{\gamma}$$

\noindent We call $P_+$ the \textit{\textbf{maximal matching}} of $S_{\alpha,T}$.

\end{prop}

We now generalise Theorem \ref{orientedarcexpansioncoefficients} to principal laminations on orientable surfaces.

\begin{thm}
\label{arcexpansioncoefficients}
Let $T$ be a triangulation of an orientable surface $(S,M)$ which contains no self-folded triangles. Then for any plain arc $\gamma$ of $(S,M)$ we have:

\begin{center}
 
 \large{$x_{\mathbf{L_T}}(\gamma) = \frac{1}{cross(\gamma, T)} {\displaystyle \sum_{P} x(P)y_{\mathbf{L_T}}(P)}$}

\end{center}

Where the sum is over all perfect matchings of the snake graph $S_{\gamma, T}$.

\end{thm}

\begin{proof}

When $\mathbf{L_T}$ is the principal lamination, $\mathbf{L_{T}^{\bullet}}$, corresponding to principal coefficients at $\Sigma_T$, the statement is precisely Theorem \ref{orientedarcexpansioncoefficients}.

To prove the result for an arbitrary principal lamination $\mathbf{L_T}$ we use Theorem \ref{orientablesep} -- the `separation of addition' formula of Fomin and Zelevinsky. For our principal lamination $\mathbf{L_T}$, note that $m= 2n$ and $y_j = x_{n+j}^{\pm 1}$ for each $j \in \{1,\ldots n\}$. Specifically, adopting the notation used there and applying Proposition \ref{maximalmatchingtriangulation} we see \newline

$$X_{\alpha}^{T}|_{Trop(x_{n+1}, \ldots, x_{m})}(1,\ldots, 1; y_1, \ldots, y_n) = \displaystyle \prod_{j=1}^n x_{j+n}^{-a_j} $$

where

\[  a_j = \left\{
\begin{array}{ll}
      0, & y_j = x_{n+j} \\
      \# \{\text{tiles in $S_{\alpha,T}$ corresponding to $\gamma_j$}\}, & y_j = x_{n+j}^{-1} \\
\end{array} 
\right. \]

Therefore, for a given term $x(P)y_{\mathbf{L^{\bullet}_T}}(P)$ in the expansion of $X_{\alpha}^T$, the corresponding term in $x_{\alpha}^T$ will be

$$\frac{x(P)y_{\mathbf{L^{\bullet}_T}}(P)|_{y_j \leftarrow x_{n+j}^{\pm 1}}}{\displaystyle \prod_{j=1}^n x_{j+n}^{-a_j}} = x(P)\displaystyle \prod_{j=1}^n x_{j+n}^{b_j}$$ for some $b_j \in \mathbf{Z}_{\geq 0}$.

Using the notion of $\mathbf{L^{\bullet}_T}$-oriented arcs (see Definition \ref{alphaoriented}) we may describe the exponents $b_j$, for $j \in \{1,\ldots, n\}$, as follows:

\[  b_j = \left\{
\begin{array}{ll}
      \# \{ \hspace{1mm}\text{\parbox{15em}{\hspace{5mm} tiles in $S_{\alpha,T}$ whose\\ diagonal $\gamma_j$ is $\mathbf{L^{\bullet}_T}$-oriented}}\hspace{-15mm}\} , \hspace{2mm} & y_j = x_{n+j} \vspace{6mm}\\ 
      \# \{ \hspace{1mm}\text{\parbox{15em}{tiles in $S_{\alpha,T}$ whose \\ diagonal is $\gamma_j$}\hspace{-25mm}\}} - \# \{\hspace{1mm}\text{\parbox{15em}{\hspace{5mm} tiles in $S_{\alpha,T}$ whose \\ diagonal $\gamma_j$ is $\mathbf{L^{\bullet}_T}$-oriented}}\hspace{-13mm}\}, \hspace{2mm} & y_j = x_{n+j}^{-1} \\
\end{array} 
\right. \]

This coincides with the notion of $\mathbf{L_T}$-oriented arcs, which concludes the proof.

\end{proof}

\begin{rmk}

For this paper it is crucial for us to understand the behaviour of expansion formulae for principal laminations, and not just principal coefficients. Indeed, given a lamination on a non-orientable surface, the lift to the double cover will have both `S' and `Z' shape intersections.

\end{rmk}

\subsection{Expansion formulae for quasi-cluster algebras with principal laminations}

Recall from \cite{wilson2018laurent} that a tagged arc $\gamma$ of $(S,M)$ is called \textit{\textbf{orientable}} if it has an orientable neighbourhood. Otherwise $\gamma$ is said to be \textit{\textbf{non-orientable}}.

\begin{defn}
\label{principalquasi}
Let $\gamma$ be a tagged arc of $(S,M)$. We define a lamination $L_{\gamma}$ as follows:

\begin{itemize}

\item If $\gamma$ is an oriented tagged arc then, by definition, it has an orientable neighbourhood and we define $L_{\gamma}$ as in Definition \ref{principal lamination def}.

\item Otherwise $\gamma$ is non-orientable with unique endpoint $m$, and the definition is split into two cases.

\begin{itemize}

\item If $m$ is a marked point on the boundary then $L_{\gamma}$ can be chosen to be either: the one-sided closed curve compatible with $\gamma$; or the lamination that runs along $\gamma$ in a small neighbourhood thereof, with endpoints on the boundary (contained within a small neighbourhood of $m$).

\item If $m$ is a puncture then $L_{\gamma}$ is chosen to be a one-sided closed curve that intersects $\gamma$ and is also compatible with $\gamma$ (when viewed as a quasi-arc).

\end{itemize}

\end{itemize}

\end{defn}

\begin{figure}[H]
\begin{center}
\includegraphics[width=13cm]{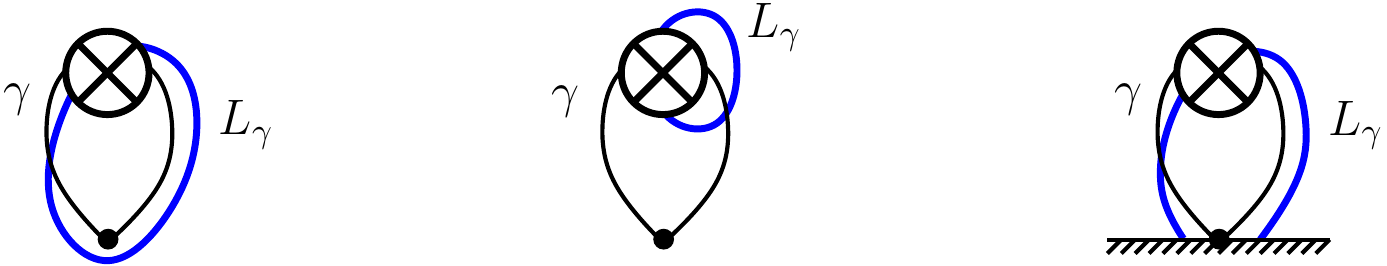}
\caption{A non-orientable arc $\gamma$ and the possible choices of the associated principal lamination $L_{\gamma}$.}
\label{onesidedlamination}
\end{center}
\end{figure}

\begin{rmk}

For each tagged arc $\gamma$ it is easily seen there are precisely two choices for $L_{\gamma}$.

\end{rmk}

\begin{defn}[Principal laminations for quasi-cluster algebras]

We say a multi-lamination of the form $\mathbf{L}_T := \{L_{\gamma} | \gamma \in T\}$ is a \textit{\textbf{principal lamination}} where $T$ is a triangulation of $(S,M)$ and $L_{\gamma}$ is the lamination defined in Definition \ref{principalquasi}.

\end{defn}

\subsubsection{Case 1: $\mathbf{L_T}$ contains no one-sided closed curves}

\begin{prop}
\label{arcexpansioncoefficientspuncture}
Let $T$ be an ideal triangulation of a non-orientable surface $(S,M)$, and let $\mathbf{L_T}$ be a principal lamination containing no one-sided closed curves. Then for any plain arc $\gamma$ of $(S,M)$ we have:

\begin{center}
 
 \large{$x_{\mathbf{L_T}}(\gamma) = \frac{1}{cross(\gamma, T)} {\displaystyle \sum_{P} x(P)y_{\mathbf{L_T}}(P)}$}

\end{center}

\noindent where the sum is over all perfect matchings of the snake graph $S_{\overline{\gamma}, \overline{T}}$.

\end{prop}

\begin{proof}

Since $\mathbf{L_T}$ is a principal lamination of $(S,M)$ which contains no one-sided closed curves, its lift to the double cover will also be a principal lamination of $\overline{(S,M)}$. The result then follows from Proposition \ref{arcexpansioncoefficients}.

\end{proof}

Analogous to Proposition \ref{splitproduct} we have the following result.

\begin{prop}
\label{splitproductcoeffcient}

With respect to the map $\Phi$ defined in Proposition \ref{welldefined}, for any principal lamination $\mathbf{L_T}$ we have:

$$y_{\mathbf{L_T}}(P) = y_{\mathbf{L_T}}(P_1)y_{\mathbf{L_T}}(P_2).$$

\end{prop}

\begin{proof}

If $\beta$ is not in $T$ then the proposition follows from Proposition \ref{decomposition}. Indeed, following the notation used there, if $s$ is even then $S_{\overline{\gamma},\overline{T}}$ decomposes as $S_{\overline{\beta},\overline{T}}$ and $S_{\alpha,\overline{T}}$. Recall that any perfect matching $P$ with a left or centre cut at $e_s$ induces perfect/good matchings $P_1$ and $P_2$ of $S_{\overline{\beta},\overline{T}}$ and $S_{\alpha,\overline{T}}$, respectively. Consequently, for any such $P$, a diagonal of $S_{\overline{\gamma},\overline{T}}$ is $\mathbf{L_T}$-oriented with respect to $P$ \textit{if an only if} the corresponding diagonal in $S_{\overline{\beta},\overline{T}}$ or $S_{\alpha,\overline{T}}$ is $\mathbf{L_T}$-oriented with respect to $P_1$ or $P_2$, respectively. \newline \indent Similarly, if $s$ is odd then $S_{\overline{\gamma},\overline{T}}$ decomposes as $S_{\overline{\beta},\overline{T}}$ and $inv(S_{\alpha,\overline{T}})$. Moreover, each $P$ of $S_{\overline{\gamma},\overline{T}}$ with a left or centre cut at $e_s$ induces perfect/good matchings $P_1$ and $P_2$ of $S_{\overline{\beta},\overline{T}}$ and $inv(S_{\alpha,\overline{T}})$, respectively. Since $s$ is odd then the relative orientation of the tile $s+1$ in $S_{\overline{\gamma},\overline{T}}$ is $-1$. Therefore, a diagonal of $S_{\overline{\gamma},\overline{T}}$ is $\mathbf{L_T}$-oriented with respect to $P$ \textit{if an only if} the corresponding diagonal in $S_{\overline{\beta},\overline{T}}$ or $inv(S_{\alpha,\overline{T}})$ is $\mathbf{L_T}$-oriented with respect to $P_1$ or $P_2$, respectively. \newline \indent This verifies that $y_{\mathbf{L_T}}(P) = y_{\mathbf{L_T}}(P_1)y_{\mathbf{L_T}}(P_2)$ for all perfect matchings of $S_{\overline{\gamma},\overline{T}}$ with a left or centre cut at $e_s$. An analogous argument holds when we look at the case when $t$ is even or odd and we consider the other decomposition of $S_{\overline{\gamma},\overline{T}}$ promised by Proposition \ref{decomposition}. In particular, one then sees that $y_{\mathbf{L_T}}(P) = y_{\mathbf{L_T}}(P_1)y_{\mathbf{L_T}}(P_2)$ for all perfect matchings of $S_{\overline{\gamma},\overline{T}}$ with a right cut at $e_s$. \newline \indent 

If $\beta$ is an arc in $T$ then by Proposition \ref{decomposition} we know $S_{\alpha,\overline{T}}$ decomposes as $S_{\overline{\gamma},\overline{T}}$ and a single tile $T_{\overline{\beta}}$ corresponding to $\overline{\beta}$. For any perfect matching $P$ of $S_{\overline{\gamma},\overline{T}}$ we know $P$ is also a good matching of $S_{\alpha,\overline{T}}$. Immediately we see a diagonal of $S_{\overline{\gamma},\overline{T}}$ is $\mathbf{L_T}$-oriented with respect to $P$ \textit{if and only if} the corresponding diagonal in $S_{\alpha,\overline{T}}$ is $\mathbf{L_T}$-oriented with respect to $P$. To prove the proposition it remains to show the diagonal of the tile $T_{\beta}$ is never $\mathbf{L_T}$-oriented with respect to any $P$. \newline \indent By assumption there are no closed curves in the principal lamination $\mathbf{L_T}$, so $L_{\beta}$ is the lamination appearing in the right of Figure \ref{onesidedlamination}. Consequently, $b_{\overline{T}}(\overline{\beta},\overline{L_{\overline{\beta}}}) = -1$. By the Zig-zag Lemma \ref{zigzag}, the diagonal of the tile $T_{\overline{\beta}}$ is never $\mathbf{L_T}$-oriented with respect to any good matching $P$ of $S_{\alpha,\overline{T}}$. Since $y_{\mathbf{L_T}}(P_2) = 1$ (by construction of $\Phi$), we have verified $y_{\mathbf{L_T}}(P) = y_{\mathbf{L_T}}(P_1)y_{\mathbf{L_T}}(P_2)$ for any perfect matching $P$ of $S_{\overline{\gamma},\overline{T}}$. This completes the proof.

\end{proof}

\begin{thm}
\label{closedcurveexpansioncoefficients}
Let $T$ be an ideal triangulation of a bordered $(S,M)$ which contains no self-folded triangles and $\mathbf{L_T}$ a principal lamination containing no one-sided closed curves. Then for any one-sided closed curve $\alpha$ we have:

\begin{center}
 
 \large{$x_{\mathbf{L_T}}(\alpha) = \frac{1}{cross(\alpha, T)} {\displaystyle \sum_{P} x(P)y_{\mathbf{L_T}}(P)}$}

\end{center}

Where:

\begin{itemize}

\item The sum is over all good matchings of the band graph $S_{\alpha, T}$.

\item $x_{\mathbf{L_T}}(\alpha)$ is the cluster variable (corresponding to $\alpha$ in the cluster algebra, $\mathcal{A}_{\mathbf{L_T}}(S,M)$, associated to the principal lamination $\mathbf{L_T}$.
 
 \end{itemize}

\end{thm}

\begin{proof}

Recall that for any $\gamma$ enclosing $\alpha$ in $M_1$, with regards to the unique arc $\beta$ contained in $M_1$, we have the relation $$x_{\mathbf{L}}(\gamma) = x_{\mathbf{L}}(\alpha)x_{\mathbf{L}}(\beta).$$

Moreover, Proposition \ref{arcexpansioncoefficientspuncture} and Proposition \ref{splitproductcoeffcient} allow us to follow an analogous argument to the one given in the proof of Theorem \ref{coefficientfreeexpansion}.

\end{proof}

\subsubsection{Case 2: $\mathbf{L_T}$ is an arbitrary principal lamination}

When a principal lamination $\mathbf{L_T}$ of $(S,M)$ contains one-sided closed curves, obtaining expansion formulae becomes more complicated. Indeed, $\mathbf{L_T}$ will no longer lift to a principal lamination on the double cover. To overcome this difficulty we will replace each one-sided closed curve with a self-intersecting curve $L$. Even though $L$ is self-intersecting it has the nice property that it lifts to two non-self-intersecting curves in the double cover $\overline{(S,M)}$. Moreover, this new collection of curves will actually lift to a principal lamination on $\overline{(S,M)}$. Consequently, for each arc $\gamma$ on the double cover we may obtain expansion formulae using Proposition \ref{arcexpansioncoefficients} with respect to the associated new coefficient system. By specialising this expansion, and dividing out a common multiple, we obtain expansion formulae for all quasi-arcs (with respect to $\mathbf{L_T}$).

\begin{defn}

Let $L_{\beta}$ be a one-sided closed curve which is an elementary lamination of a tagged arc $\beta$. Let $\overline{\beta}$ and $\tilde{\overline{\beta}}$ denote the two lifts of $\beta$ in $\overline{(S,M)}$, and note that $L_{\beta}$ lifts to a single closed curve $\overline{L}_{\beta}$. With respect to any tagged triangulation $T$ containing $\overline{\beta}$ and $\tilde{\overline{\beta}}$ we see there exists $\epsilon \in \{+1,-1\}$ such that: $$b_{\overline{T}}(\overline{\beta}, \overline{L}_{\beta}) = \epsilon, \hspace{10mm} b_{\overline{T}}(\tilde{\overline{\beta}}, \overline{L}_{\beta}) = -\epsilon, \hspace{10mm} b_{\overline{T}}(\gamma, \overline{L}_{\beta}) = 0 \hspace{3mm} \forall \gamma \in \overline{T}\setminus \{\overline{\beta}, \tilde{\overline{\beta}}\}.$$

We define $L_{\beta}^*$ to be the unique curve on $(S,M)$ that lifts to two laminations $\overline{L}_{\beta}^*$ and $\tilde{\overline{L}}_{\beta}^*$ in $\overline{(S,M)}$ where: $$b_{\overline{T}}(\overline{\beta}, \overline{L}_{\beta}^*) = \epsilon, \hspace{10mm} b_{\overline{T}}(\gamma, \overline{L}_{\beta}^*) = 0 \hspace{3mm} \forall \gamma \in \overline{T}\setminus \{\overline{\beta}\}.$$ 

$$b_{\overline{T}}(\tilde{\overline{\beta}}, \tilde{\overline{L}}_{\beta}^*) = \epsilon, \hspace{10mm} b_{\overline{T}}(\gamma, \tilde{\overline{L}}_{\beta}^*) = 0 \hspace{3mm} \forall \gamma \in \overline{T}\setminus \{\tilde{\overline{\beta}}\}.$$ \newline \noindent We say $L_{\beta}^*$ is the \textbf{\textit{quasi-lamination}} associated to $L_{\beta}$. As indicated in Figure \ref{quasi-lamination}, $L_{\beta}^*$ is self-intersecting and therefore not a lamination. However, by construction, its two lifts $\overline{L}_{\beta}^*$ and $\tilde{\overline{L}}_{\beta}^*$ are laminations of the orientable double cover (they are elementary laminations of $\overline{\beta}$ and $\tilde{\overline{\beta}}$, respectively). 

\end{defn}

\begin{figure}[H]
\begin{center}
\includegraphics[width=12cm]{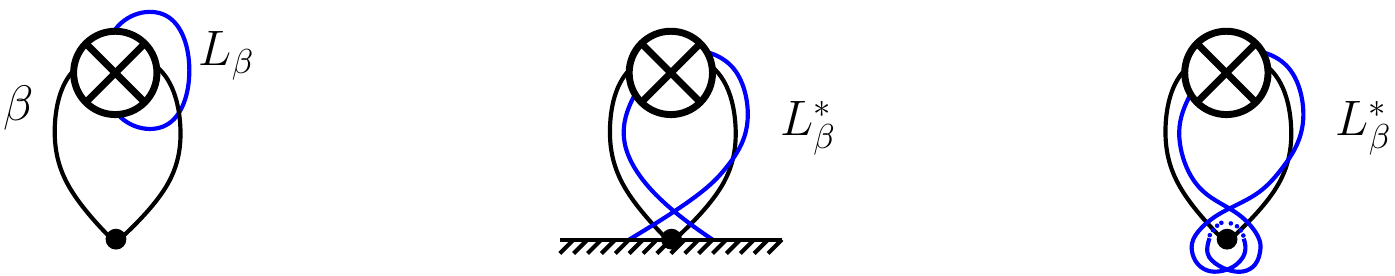}
\caption{On the left we show an elementary lamination of $\beta$ which is a one-sided closed curve. On the right we illustrate the associated quasi-lamination depending on whether the endpoint of $\beta$ is a marked point on the boundary or a puncture.}
\label{quasi-lamination}
\end{center}
\end{figure}

\begin{defn}
Let $\mathbf{L_T}$ be a principal lamination of $(S,M)$. We define $$ \mathbf{L_T^*} := \mathbf{L_T}\setminus \Big\{L_{\beta} | \substack{L_{\beta} \in \mathbf{L_T} \hspace{1mm} \text{is a one-sided } \\ \text{closed curve}}\Big\} \cup \Big\{L_{\beta}^* | \substack{L_{\beta} \in \mathbf{L_T} \hspace{1mm} \text{is a one-sided } \\ \text{closed curve}}\Big\}$$

\noindent and call $\mathbf{L_T^*}$ the \textit{\textbf{quasi-principal lamination}} associated to $\mathbf{L_T}$.

\end{defn}

\begin{defn}
\label{badencounterdefn}
Let $T$ be a triangulation and suppose $L_{\beta}^*$ is a quasi-lamination of an arc $\beta$ in $T$. In this case $\beta$ is a non-orientable arc with unique endpoint $m$. Let $B_m$ be a small neighbourhood of $m$ which is large enough so that no self-intersections of $L_{\beta}^*$ occur outside $B_m$. \newline \indent For any directed quasi-arc $\gamma$ of $(S,M)$ we say $\gamma$ has a \textit{\textbf{bad encounter}} with $L_{\beta}^*$ \textit{if and only if}, one of the following holds:

\begin{itemize}

\item $\gamma$ is a one-sided closed curve which is homotopic to $L_{\beta}$, or

\item within a small neighbourhood of $B_m$, $\gamma$ intersects $\beta, L_{\beta}^*, L_{\beta}^*, \beta$ (in that order).

\end{itemize}

\end{defn}

\begin{figure}[H]
\begin{center}
\includegraphics[width=12cm]{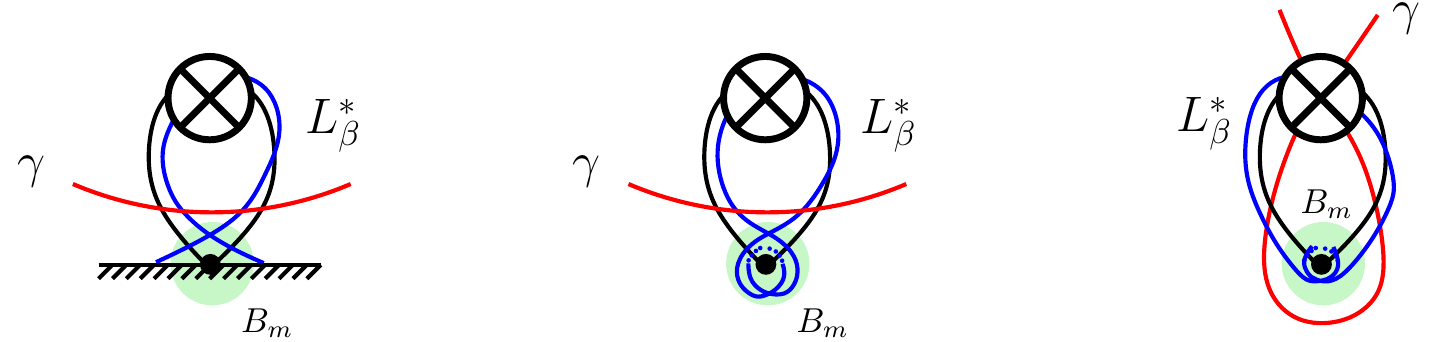}
\caption{Examples of bad encounters of a quasi-arc $\gamma$ with various quasi-laminations $L_{\beta}^*$. Here the green shaded area denotes the neighbourhood $B_m$ described in Definition \ref{badencounterdefn}.}
\label{badencounter}
\end{center}
\end{figure}

\begin{defn}

Let $T$ be a triangulation and $\mathbf{L_T^*}$ be a quasi-principal lamination. For any quasi-arc $\gamma$ of $(S,M)$ we define $$ bad( \mathbf{L_T^*},\gamma) := \displaystyle \prod_{\beta \in T} y_{\beta}^{a_{\beta}}$$ where $a_{\beta}$ is the number of bad encounters of $\gamma$ with $L_{\beta}^* \in  \mathbf{L_T^*}$.

\end{defn}

\begin{thm}
\label{badarcexpansioncoefficients}
Let $T$ be a triangulation and $\mathbf{L_T}$ a principal lamination. Then for any arc $\gamma$ of $(S,M)$ we have:

\begin{equation}
\label{factorproperty}
x_{\mathbf{L_T}}(\gamma) = \frac{x_{\mathbf{L_T^*}}(\gamma)}{bad( \mathbf{L_T^*},\gamma)}
\end{equation}

\end{thm}

\begin{proof}

Let $Q$ be a quadrilateral in $(S,M)$. First we wish to determine when the exchange relations of the diagonals of $Q$ are different with respect to the two coefficient systems arising from $\mathbf{L_T}$ and $\mathbf{L_T^*}$. We shall then prove the proposition by induction on the flipping of arcs; showing the flipped arc also satisfies equation (\ref{factorproperty}). \newline \indent Without loss of generality, for an non-orientable arc $\beta$ in $T$, we may restrict our attention to a single one-sided lamination $L_{\beta}$ of $\beta$ and the associated quasi-lamination $L_{\beta}^*$. There is a difference in the exchange relations of $Q$ only if a self-crossing of $L_{\beta}$ lies inside $Q$ (up to isotopy). In Figure \ref{differentexchangerelations}, up to reflectional symmetry, we list all configurations for which the exchange relations differ. One can verify this list is complete by running through all possible combinations of how the four endpoints of the crossing can exit the quadrilateral.

To prove the proposition we will show that if the sides $a,b,c,d$ and a diagonal $\gamma$ of a quadrilateral $Q$ satisfy equation equation (\ref{factorproperty}), then so does the flipped diagonal $\gamma'$. \newline \indent When the two exchange relations of $Q$ (associated to $L_{\beta}$ and $L_{\beta}^*$) coincide then the only bad encounters of $a,b,c,d, \gamma, \gamma'$ with $L_{\beta}^*$ can occur at the endpoints of $Q$. Specifically, for some $n_1, n_2, n_3, n_4 \geq 0$ we see $a,b,c,d, \gamma, \gamma'$ have precisely $n_1+n_2, n_2 + n_3, n_3+ n_4, n_1+n_3, n_2 + n_4$ bad encounters with $L_{\beta}^*$, respectively. Therefore, if $a,b,c,d, \gamma$ satisfy equation (\ref{factorproperty}) we see 
\begin{align*}
x_{\mathbf{L_T^*}}(\gamma') &= \frac{x_{\mathbf{L_T^*}}(a)x_{\mathbf{L_T^*}}(c) + x_{\mathbf{L_T^*}}(b)x_{\mathbf{L_T^*}}(d)}{x_{\mathbf{L_T^*}}(\gamma)} = \\ &= \frac{y_{\beta}^{n_1+n_2+n_3+n_4}\big(x_{\mathbf{L_T}}(a)x_{\mathbf{L_T}}(c) + x_{\mathbf{L_T}}(b)x_{\mathbf{L_T}}(d)\big)}{y_{\beta}^{n_1+n_3}x_{\mathbf{L_T}}(\gamma)} = \\ &= y_{\beta}^{n_2+n_4}x_{\mathbf{L_T}}(\gamma').
\end{align*}

\noindent Hence $\gamma'$ also satisfies equation (\ref{factorproperty}). It remains to verify the case when the exchange relations of $Q$ are different. As mentioned already, a complete list of such instances are provided in Figure \ref{differentexchangerelations}. Analogous to the case considered above, the endpoints of the quadrilateral $Q$ may also have bad encounters with $L_{\beta}^*$. Specifically, in addition to the bad encounters illustrated in Figure \ref{differentexchangerelations}, for some $n_1, n_2, n_3, n_4 \geq 0$ we see $a,b,c,d, \gamma, \gamma'$ have precisely $n_1+n_2, n_2 + n_3, n_3+ n_4, n_1+n_3, n_2 + n_4$ additional bad encounters with $L_{\beta}^*$, respectively. Without loss of generality, we shall assume $n_1, n_2, n_3, n_4 = 0$ as otherwise the calculation essentially reduces to the one presented above. We will show that for configurations (a) and (b) in Figure \ref{differentexchangerelations} that if $a,b,c,d, \gamma$ satisfy equation (\ref{factorproperty}) then so does $\gamma'$. For configuration (a) we obtain:

\begin{align*}
x_{\mathbf{L_T^*}}(\gamma') &= \frac{y_{\beta}x_{\mathbf{L_T^*}}(a)x_{\mathbf{L_T^*}}(c) + y_{\beta}x_{\mathbf{L_T^*}}(b)x_{\mathbf{L_T^*}}(d)}{x_{\mathbf{L_T^*}}(\gamma)} = \\ &= \frac{y_{\beta}\big(x_{\mathbf{L_T}}(a)x_{\mathbf{L_T}}(c) + x_{\mathbf{L_T}}(b)x_{\mathbf{L_T}}(d)\big)}{x_{\mathbf{L_T}}(\gamma)} = \\ &= y_{\beta}x_{\mathbf{L_T}}(\gamma').
\end{align*}

For configuration (b) we see:

\begin{align*}
x_{\mathbf{L_T^*}}(\gamma') &= \frac{y_{\beta}x_{\mathbf{L_T^*}}(a)x_{\mathbf{L_T^*}}(c) + x_{\mathbf{L_T^*}}(b)x_{\mathbf{L_T^*}}(d)}{x_{\mathbf{L_T^*}}(\gamma)} = \\ &= \frac{y_{\beta}\big(x_{\mathbf{L_T}}(a)x_{\mathbf{L_T}}(c) + x_{\mathbf{L_T}}(b)x_{\mathbf{L_T}}(d)\big)}{x_{\mathbf{L_T}}(\gamma)} = \\ &= y_{\beta}x_{\mathbf{L_T}}(\gamma').
\end{align*}

The remaining configurations (c),(d) and (e) follow in a similar way. This concludes the proof since any two arcs of (S,M) are related by a sequence of flips.

\end{proof}

\begin{figure}[H]
\begin{center}
\includegraphics[width=14cm]{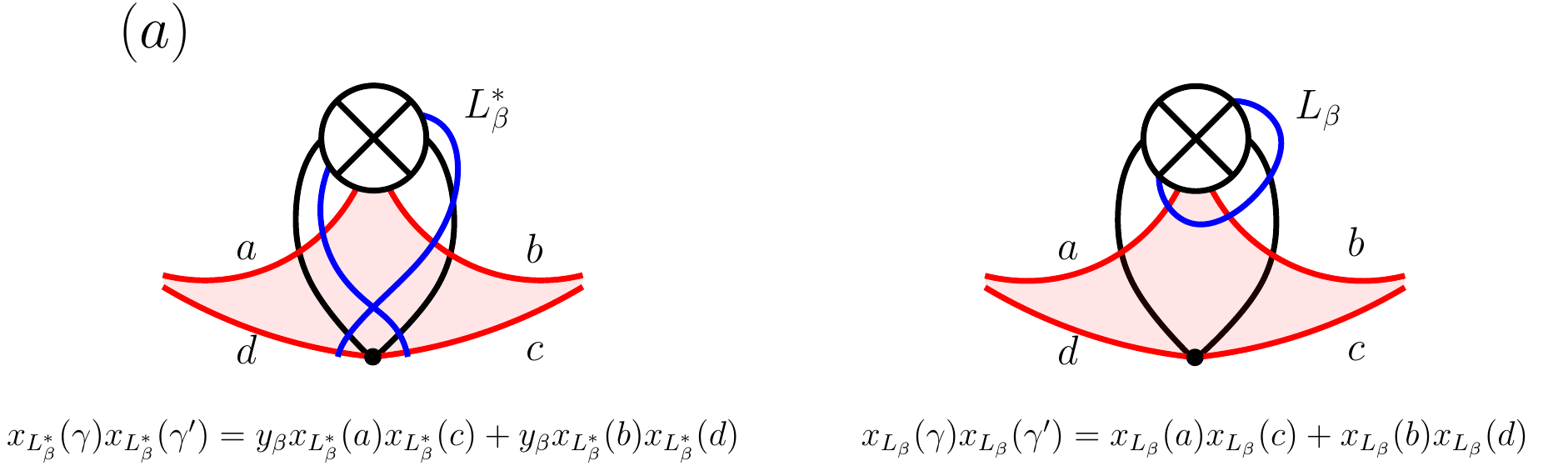}
\end{center}
\end{figure}

\begin{figure}[H]
\begin{center}
\includegraphics[width=14cm]{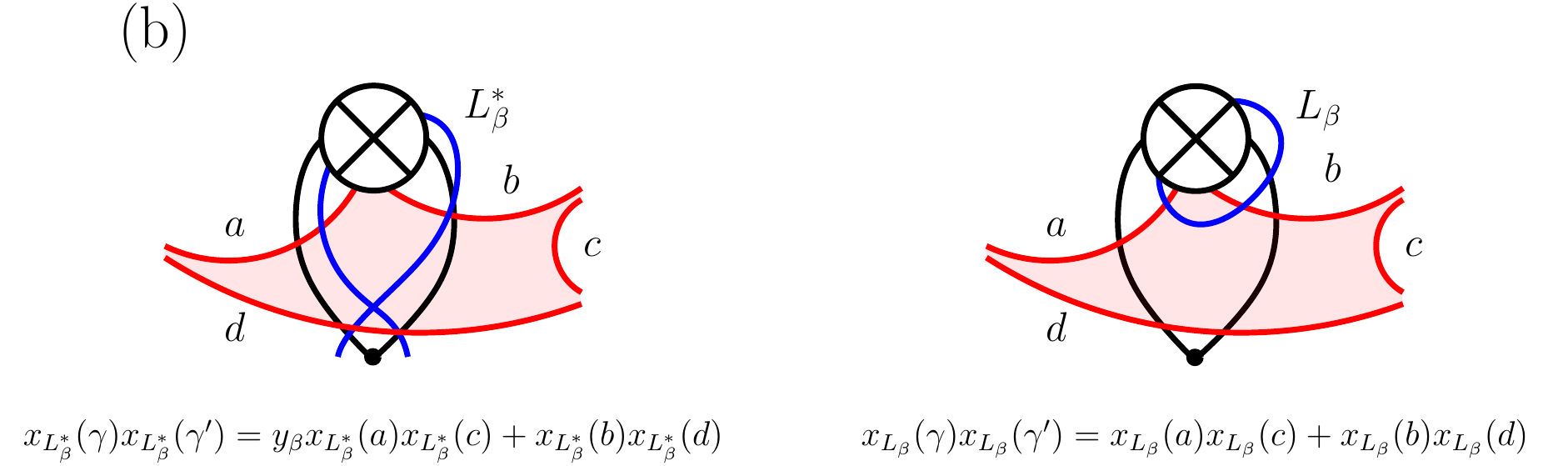}
\end{center}
\end{figure}

\begin{figure}[H]
\begin{center}
\includegraphics[width=14cm]{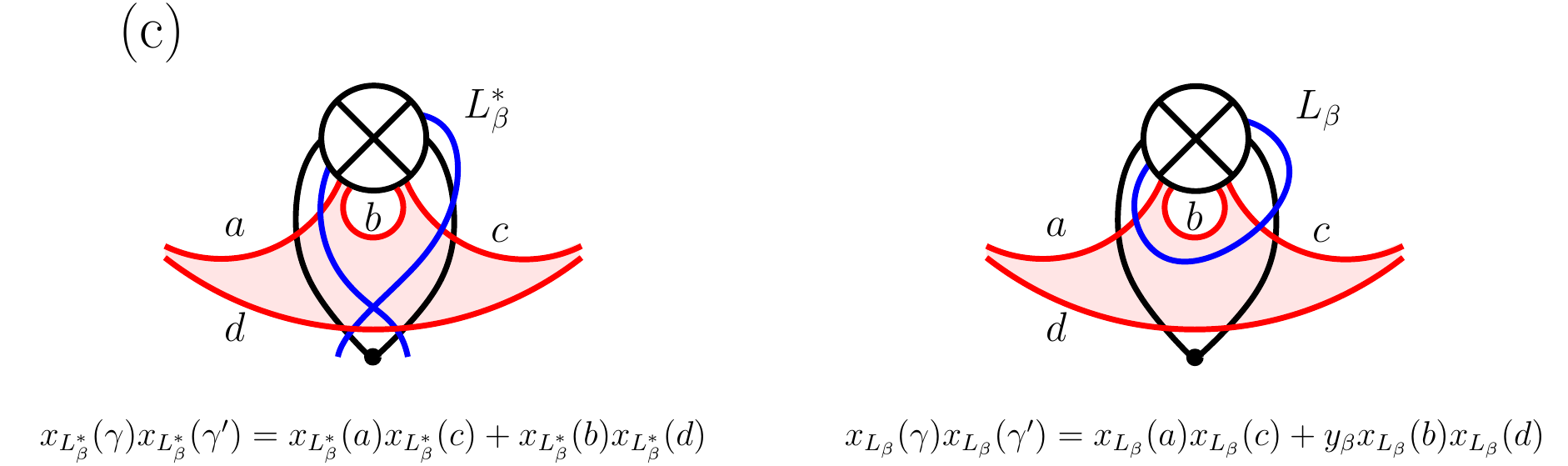}
\end{center}
\end{figure}

\begin{figure}[H]
\begin{center}
\includegraphics[width=14cm]{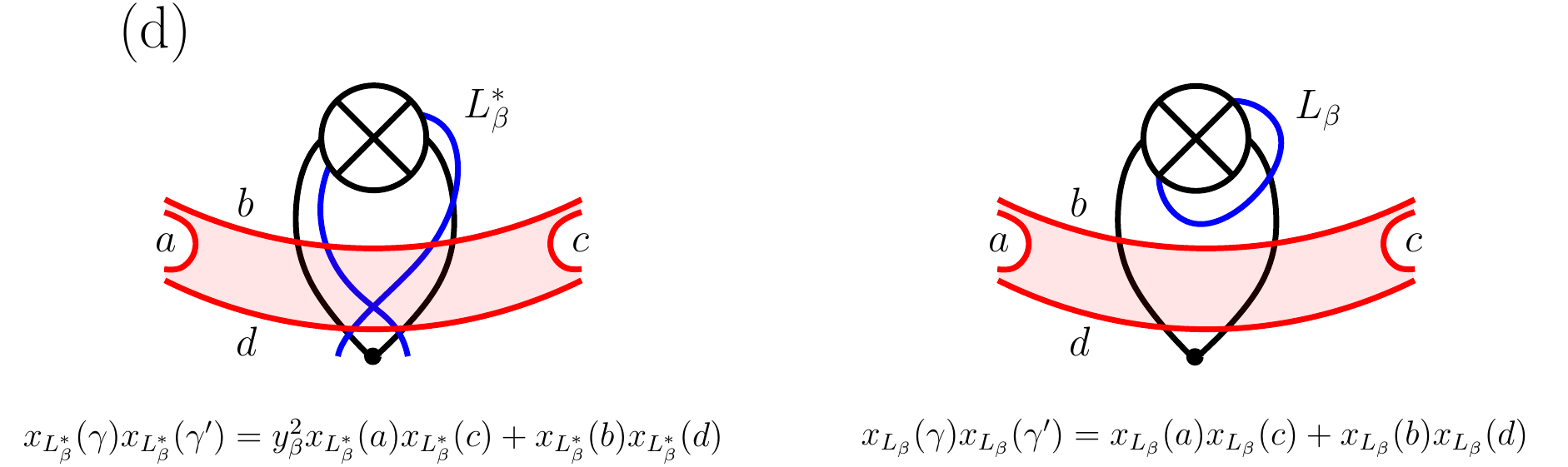}
\end{center}
\end{figure}

\begin{figure}[H]
\begin{center}
\includegraphics[width=14cm]{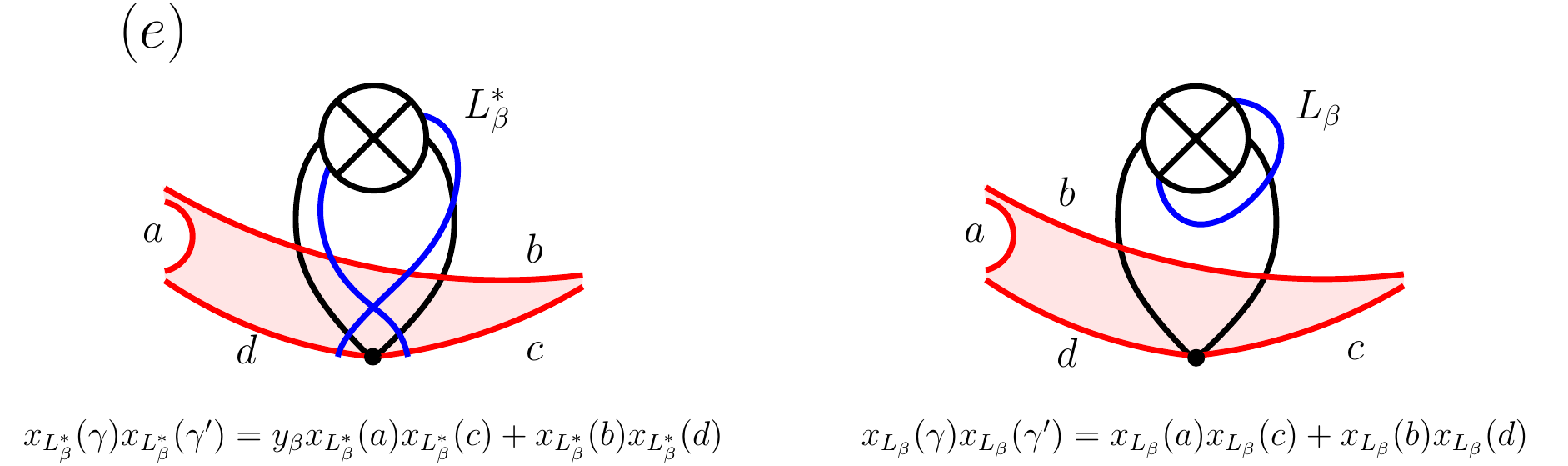}
\caption{Examples of bad encounters of a quasi-arc $\gamma$ with various quasi-laminations $L_{\beta}^*$. Here the green shaded area denotes the neighbourhood $B_m$ described in Definition \ref{badencounterdefn}.}
\label{differentexchangerelations}
\end{center}
\end{figure}

With respect to an initial triangulation $T$ and a principal lamination $\mathbf{L_T}$, we now wish to obtain expansion formulae for each one-sided closed curve $\alpha$ of $(S,M)$.

\begin{lem}
\label{badencountermobius}
Let $\alpha$ be a one-sided closed curves and consider an arc $\gamma$ enclosing $\alpha$ in a M\"{o}bius strip with one marked point on the boundary, $M_1$. Furthermore, let $\beta$ be the arc in this $M_1$. Then for any triangulation $T$ and principal lamination $\mathbf{L_T}$ we have:

\begin{equation}
bad(\gamma, \mathbf{L_T^*}) := \left\{
\begin{array}{ll}
      \frac{bad(\alpha, \mathbf{L_T^*})bad(\beta, \mathbf{L_T^*})}{y_{\beta}} , \hspace{2mm} & \substack{\text{$\beta$ is in $T$, and $L_{\beta}$} \\ \text{ is homotopic to $\alpha$.}} \\ 
      bad(\alpha, \mathbf{L_T^*})bad(\beta, \mathbf{L_T^*}) , \hspace{2mm} & \text{otherwise}. \\
\end{array} 
\right.
\end{equation}

\end{lem}

\begin{proof}

To prove this lemma it suffices to consider verify the identities for a single quasi-lamination $L_{\delta}^*$ (associated to a non-orientable arc $\delta$). In Figure \ref{badencounterclassification} we list all \textit{elementary} configurations (up to reflectional symmetry) which involve at least one of $\alpha$, $\beta$ or $\gamma$ having a bad encounter with $L_{\delta}^*$. By `elementary' we mean that any configuration (in which $\alpha$, $\beta$ or $\gamma$ having a bad encounter with $L_{\delta}^*$) can be written as a union of these configurations. One can verify the list in Figure \ref{badencounterclassification} is complete by supposing there is a bad encounter between $L_{\delta}^*$ and one of the quasi-arcs $\alpha$, $\beta$ or $\gamma$. Then, using the fact $\gamma$ encloses $\alpha$ and $\beta$ in a M\"{o}bius strip with one marked point on the boundary, it is easy to run through the possible configurations involving the remaining two quasi-arcs. Note that in case (d) we have that $\delta = \beta$ and $L_{\beta}$ is homotopic to $\alpha$. Moreover, $bad(\alpha) = y_{\beta}$ and $bad(\beta) = bad(\gamma) = 1$, so $bad(\gamma) = \frac{bad(\alpha)bad(\beta)}{y_{\beta}}$. For all other cases one sees that $bad(\gamma) = bad(\alpha)bad(\beta)$ which completes the proof of the lemma.

\end{proof}

\begin{figure}[H]
\begin{center}
\includegraphics[width=14cm]{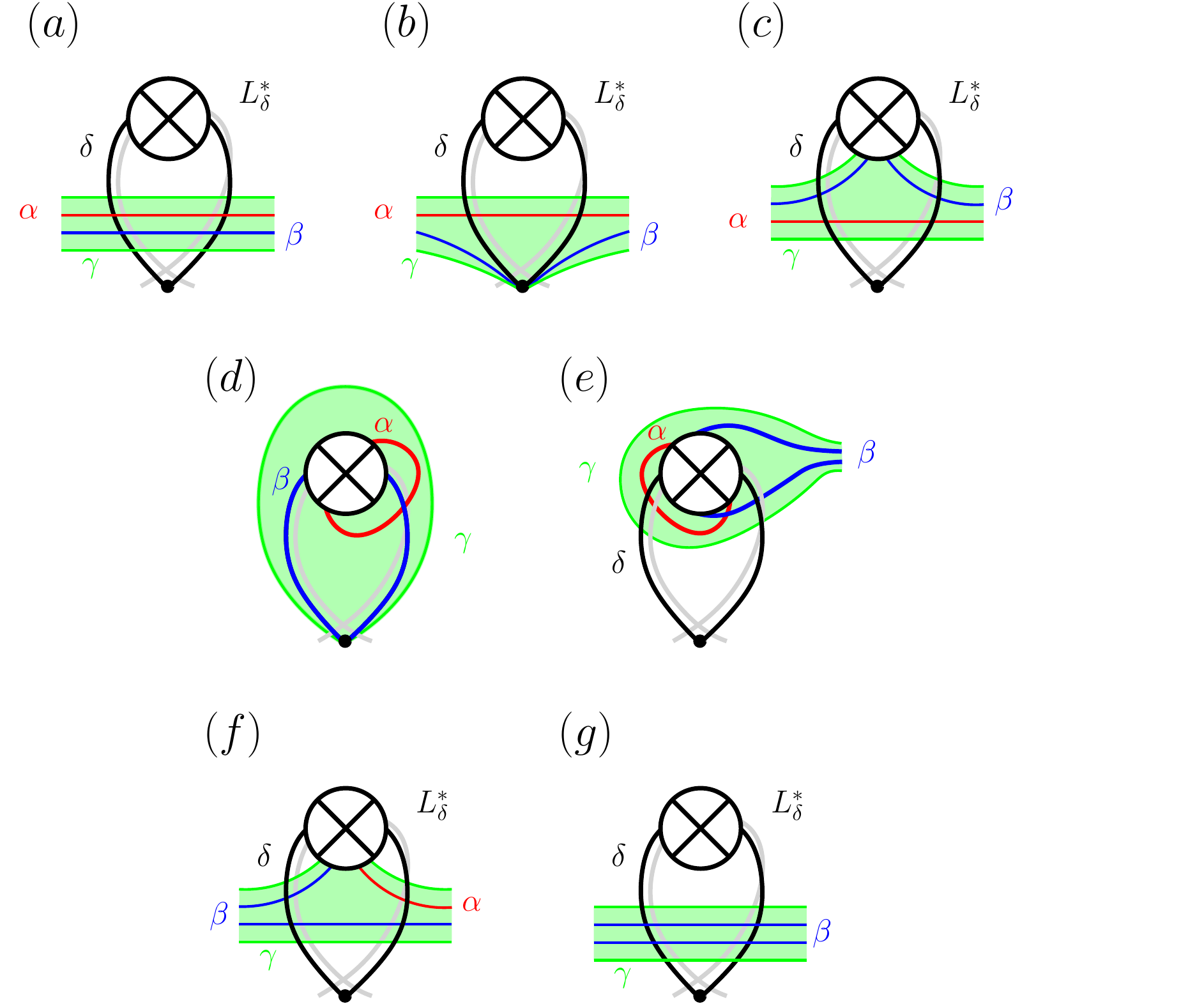}
\caption{With respect to the scenario discussed in Lemma \ref{badencountermobius}, we list the various configurations (up to reflectional symmetry) in which at least one of $\alpha$, $\beta$ or $\gamma$ have a bad encounter with $L_{\delta}^*$ -- the quasi-lamination of a lamination $L_{\delta}$ in $\mathbf{L_{T}}$.}
\label{badencounterclassification}
\end{center}
\end{figure}

\begin{prop}
\label{badsplitproductcoeffcient}

Let $T$ be a triangulation of $(S,M)$ and $\mathbf{L_{T}}$ be a principal lamination of $T$. With respect to the coefficient system of $\mathbf{L_{T}^*}$ and the map $\Phi$ defined in Proposition \ref{welldefined}, we have:

\begin{equation}
y_{\mathbf{L_T^*}}(P) = \left\{
\begin{array}{ll}
      \hspace{5mm} \frac{y_{\mathbf{L_T^*}}(P_1)}{y_{\beta}} , \hspace{2mm} & \substack{\text{some $L_{\beta} \in \mathbf{L_{T}}$ is} \\ \text{homotopic to $\alpha$ (as curves).}} \\ 
      y_{\mathbf{L_T^*}}(P_1)y_{\mathbf{L_T^*}}(P_2) , \hspace{2mm} & \text{otherwise}. \\
\end{array} 
\right.
\end{equation}

\end{prop}

\begin{proof}

We follow the same notation used in Proposition \ref{splitproductcoeffcient}. If $\beta$ is not in $T$ then the proof is the same as in Proposition \ref{splitproductcoeffcient}. If $\beta$ is in $T$ then by the Zig-Zag Lemma \ref{zigzag}, for any good matching $P$ of $S_{\alpha,T}$, the tile $T_{\beta}$ is $\alpha$-oriented when $L_{\beta}$ is homotopic to $\alpha$, and is never $\alpha$-oriented otherwise. This completes the proof. 

\end{proof}

\begin{thm}
\label{closedcurveexpansioncoefficients2}
Let $T$ be an ideal triangulation of a bordered $(S,M)$ which contains no self-folded triangles and let $\mathbf{L_T}$ be a principal lamination. Then for any one-sided closed curve $\alpha$ we have:

\begin{equation}
\label{onesidedbadexpansion}
x_{\mathbf{L_T}}(\alpha) = \frac{1}{bad(\mathbf{L_T},\alpha)cross(\alpha, T)} {\displaystyle \sum_{P} x(P)y_{\mathbf{L_T^*}}(P)}
\end{equation}

Where:

\begin{itemize}

\item The sum is over all good matchings of the band graph $S_{\alpha, T}$.

\item $x_{\mathbf{L_T}}(\alpha)$ is the cluster variable (corresponding to $\alpha$) in the quasi-cluster algebra, $\mathcal{A}_{\mathbf{L_T}}(S,M)$, associated to the principal lamination $\mathbf{L_T}$.
 
 \end{itemize}

\end{thm}

\begin{proof}

Recall that for any $\gamma$ enclosing $\alpha$ in $M_1$, with regards to the unique arc $\beta$ contained in $M_1$, we have the relation: $$x_{\mathbf{L_T}}(\gamma) = x_{\mathbf{L_T}}(\alpha)x_{\mathbf{L_T}}(\beta).$$ Moreover, by combining Theorem \ref{closedcurveexpansioncoefficients} and Theorem \ref{arcexpansioncoefficients} we see:

\begin{align*}
&\frac{x_{\mathbf{L_T}}(\alpha)}{bad(\mathbf{L_T^*},\beta)cross(T,\beta)}\sum_{P \in \mathcal{P}_{T,\beta}} x(P)y_{\mathbf{L_T^*}}(P) = x_{\mathbf{L_T}}(\alpha)x_{\mathbf{L_T}}(\beta) = x_{\mathbf{L_T}}(\gamma) = \\ =& \frac{1}{bad(\mathbf{L_T^*},\gamma)cross(T,\gamma)}\sum_{P \in \mathcal{P}_{T,\gamma}} x(P)y_{\mathbf{L_T^*}}(P).
\end{align*}

By Proposition \ref{splitproduct}, Proposition \ref{badencountermobius} and Proposition \ref{badsplitproductcoeffcient} we have:

\begin{align*}
&\frac{1}{bad(\mathbf{L_T^*},\gamma)cross(T,\gamma)}\sum_{P \in \mathcal{P}_{T,\gamma}} x(P)y_{\mathbf{L_T^*}}(P) = \\ = & \frac{1}{bad(\mathbf{L_T^*},\alpha)bad(\mathbf{L_T^*},\beta)cross(T,\gamma)}\Big(\sum_{P \in \mathcal{P}_{T,\alpha}} x(P)y_{\mathbf{L_T^*}}(P)\Big)\Big(\sum_{P \in \mathcal{P}_{T,\beta}} x(P)y_{\mathbf{L_T^*}}(P)\Big).
\end{align*}

Corollary \ref{crossings} then completes the proof.

\end{proof}

\begin{rmk}

It is important to note that the coefficient monomial in equation (\ref{onesidedbadexpansion}) is defined with respect to the quasi-principal lamination $\mathbf{L_T^*}$.

\end{rmk}

\section{Expansion formulae for tagged quasi-arcs with respect to any triangulations}
\label{alltriangulations}
\subsection{Snake and band graphs for triangulations with self-folded triangles}

In Section \ref{snakebandconstruction}, we saw one can assign a graph $S_{\alpha, T}$ when given a quasi arc $\alpha$ and an ideal triangulation $T$ containing no self-folded triangles. However, if $T$ contains self-folded triangles then ambiguities may arise in this procedure with respect to how tiles are defined and glued -- in particular, ambiguity occurs precisely when passing through a self-folded triangle. Musiker, Schiffler and Williams \cite{musiker2011positivity} addressed these points using the following two definitions.

\begin{defn}
Let $\gamma$ be an oriented plain arc or a one-sided closed curve. As usual, we let $p_1, \ldots, p_d$ denote the intersection points of $\gamma$ and $T$, and let $\tau_{i_j}$ denote the arc in $T$ containing the point $p_j$. \newline \indent  If $\tau_{i_j}$ is not the folded side of a self-folded triangle in $T$ then the \textit{\textbf{ordinary}} tile $G_j$ is defined as in Definition \ref{relativeorientation}. Otherwise, the \textit{\textbf{non-ordinary}} tile $G_j$ is defined by glueing two copies of the triangle $(\tau_{i_{j+1}}, \tau_{i_{j+1}}, \tau_{i_j} = \tau_{i_{j+2}})$ along $\tau_{i_{j+1}}$, such that the labels on the north and west (equivalently south and east) edges of $G_j$ are equal -- there are two such tiles.

\end{defn}

\begin{defn}
\label{puncturedsnakegraph}

The \textit{\textbf{snake graph}} (resp. \textit{\textbf{band graph}}) $\mathcal{G}_{\gamma,T}$ associated to a directed plain arc (resp. one-sided closed curve) $\gamma$ is defined as follows:

\begin{itemize}

\item If $G_j$ and $G_{j+1}$ are both ordinary tiles then the glueing procedure is defined as in Definition \ref{arcsnakegraph} (resp. Definition \ref{surfacebandgraph}).

\item If $G_j$ is a non-ordinary tile then it is glued to $G_{j-1}$ and $G_{j+1}$ by the rules outlined in Figure \ref{nonordinarytile}.

\end{itemize}

For ordinary tiles we have the notion of \textit{\textbf{relative orientation}} given in Definition \ref{relativeorientation}. We extend this to each non-ordinary tile $G_j$ appearing in some $\mathcal{G}_{\gamma,T}$ by defining $rel(G_j) := -rel(G_{j-1})$. Note that by Figure \ref{nonordinarytile} we therefore have that $rel(G_i) := -rel(G_{i-1})$ for any tile appearing in $\mathcal{G}_{\gamma,T}$.

\end{defn}

\begin{figure}[H]
\begin{center}
\includegraphics[width=10cm]{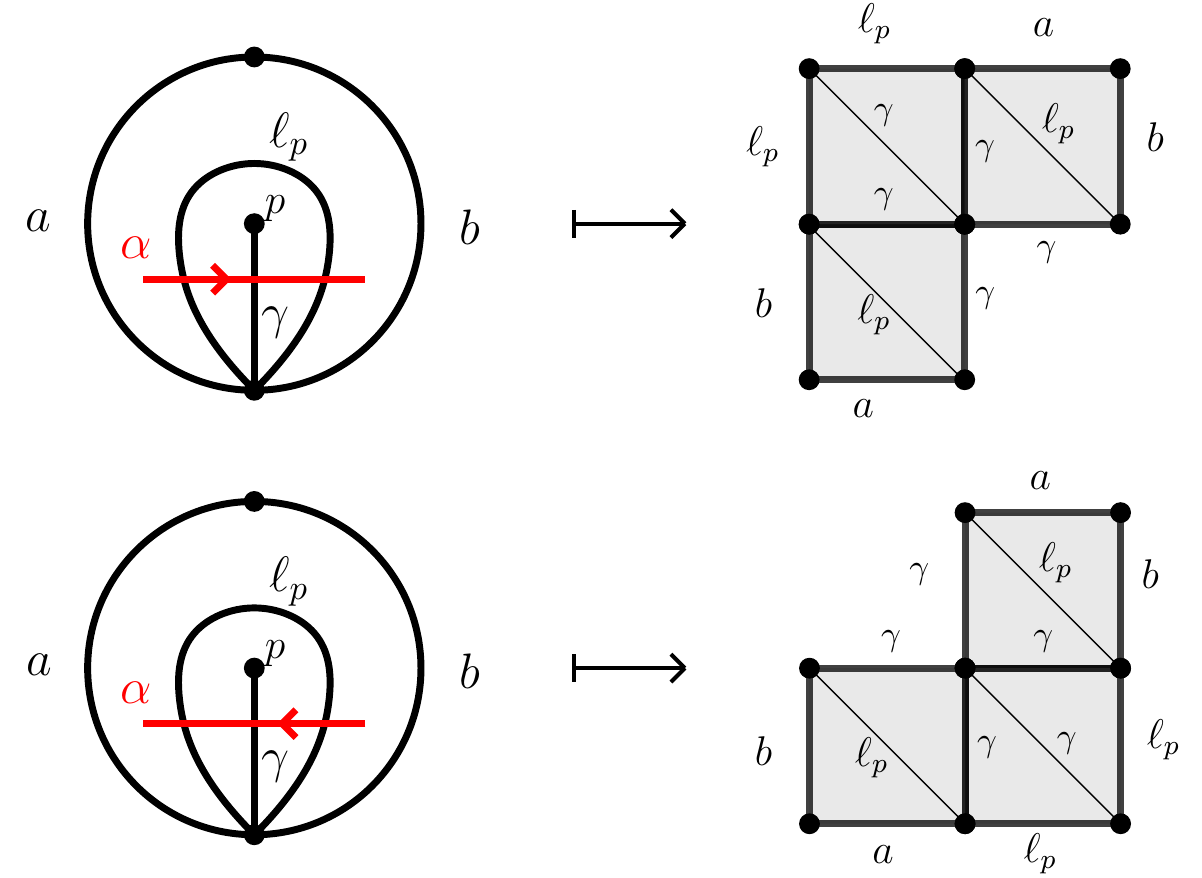}
\caption{The gluing specifications of a non-ordinary tile for the snake and band graph constructions described in Definition \ref{puncturedsnakegraph}.}
\label{nonordinarytile}
\end{center}
\end{figure}

\subsection{Expansion formulae for plain arcs}

Note that to define the snake graph or band graph of a curve $\gamma$ with respect to a tagged triangulation $T$, we actually pass to the associated ideal triangulation $T^{\circ}$. Consequently, when $T^{\circ}$ contains self-folded triangles, the graph is larger than usual, in the sense that formulae found in Theorem \ref{arcexpansioncoefficients} and Theorem \ref{closedcurveexpansioncoefficients} will no longer hold. \newline \indent More specifically, the terms in the right hand side of these formulae contain more frozen variables than required (the crossing monomial $cross(T^{\circ}, \gamma)$ prevents extra cluster variables appearing). In the context of cluster algebras with principal coefficients, Musiker, Schiffler and Williams eliminate these extra frozen variables by a certain specialisation of the coefficient monomials (see [Definition 4.8 \cite{musiker2011positivity}]). In the interest of obtaining expansion formulae for any choice of principal lamination we are required to define the coefficient monomial from a different viewpoint.

\begin{defn}

Let $T$ be an ideal triangulation and $\mathbf{L_T}$ be a principal lamination. Let $S_{\alpha,T}$ be the snake (resp. band) graph of some directed plain arc (resp. one-sided closed curve) $\alpha$. Suppose $G_j$ is a tile in $S_{\alpha,T}$ corresponding to the radius, $\gamma$, of a self-folded triangle in $T$ with puncture $p$ (i.e. $G_j$ is a non-ordinary tile). \newline \indent Recall that every perfect matching $P$ of $S_{\alpha,T}$ induces an orientation on the diagonals of each tile in $S_{\alpha,T}$. We say the radius $\gamma$ associated to $G_j$ is \textit{\textbf{$\mathbf{L_T}$-directed}} with respect to some $P$ if:

\begin{itemize}

\item the relative orientation of the tile $G_j$ is odd and one of the following statements holds:

\begin{itemize}

\item $b_{T}(L_{\gamma^{(p)}}, \gamma^{(p)}) = 1$ and the diagonal of $G_j$ is oriented `down'.

\item $b_{T}(L_{\gamma^{(p)}}, \gamma^{(p)}) = -1$ and the diagonal of $G_j$ is oriented `up'.

\end{itemize}

\item the relative orientation of the tile $G_j$ is even and one of the following statements holds:

\begin{itemize}

\item $b_{T}(L_{\gamma^{(p)}}, \gamma^{(p)}) = 1$ and the diagonal of $G_j$ is oriented `up'.

\item $b_{T}(L_{\gamma^{(p)}}, \gamma^{(p)}) = -1$ and the diagonal of $G_j$ is oriented `down'.

\end{itemize}

\end{itemize}

\end{defn}

\begin{defn}

Let $T$ be an ideal triangulation and $\mathbf{L_T}$ be a principal lamination. For any perfect/good matching $P$ of a snake/band graph $S_{\alpha,T}$ we define the \textit{\textbf{coefficient monomial}}, $y_{\mathbf{L_T}}(P)$, as follows:

\begin{equation}
y_{\mathbf{L_T}}(P) := \displaystyle \Big(\prod_{\substack{\text{$\gamma_{i_j}$ is an} \\ \text{$\mathbf{L_T}$-oriented}\\ \text{diagonal}}} y_{\gamma_{i_j}} \Big) \Big(\prod_{\substack{\text{$\gamma_{i_j}$ is an} \\ \text{$\mathbf{L_T}$-directed}\\ \text{radius}}} y_{\gamma_{i_j}^{(p)}}^{-1} \Big)
\end{equation}

\end{defn}

\begin{rmk}

Note that if a quasi-arc $\alpha$ does not intersect the radius of a self-folded triangle in $T$, then the the definition of the coefficient monomial coincides with the one in Definition \ref{ymonomial}.

\end{rmk}

Musiker, Schiffler and Williams obtained the following two results in the setting of cluster algebras with principal coefficients.

\begin{thm}[Theorem 4.9, \cite{musiker2011positivity}]
\label{orientedpuncturearc}
Let $(S,M)$ be an orientable bordered surface and let $T^{\circ}$ be an ideal triangulation with corresponding tagged triangulation $T = i(T^{\circ})$. If $\mathbf{L_{T}^{\bullet}}$ is the principal lamination corresponding to principal coefficients at $\Sigma_T$, then for any arc $\alpha$ we have:

\begin{equation}
x_{\mathbf{L_{T}^{\bullet}}}(\alpha) = \frac{1}{cross(\alpha, T^{\circ})} {\displaystyle \sum_{P} x(P)y_{\mathbf{L_{T}^{\bullet}}}(P)}
\end{equation}

Where:

\begin{itemize}

\item The sum is over all perfect matchings of the snake graph $\mathcal{G}_{\gamma, T^{\circ}}$.

\item $x_{\mathbf{L_{T}^{\bullet}}}(\gamma)$ is the cluster variable (corresponding to $\gamma$) in the cluster algebra $\mathcal{A} := \mathcal{A}_{\mathbf{L_{T}^{\bullet}}}(S,M)$ with initial seed $\Sigma_T$.
 
 \end{itemize}

\end{thm}

\begin{prop}[Corollary 6.4, \cite{schiffler2010cluster} and Theorem 5.1, \cite{musiker2010cluster}]
\label{maximalpunctureorientable}
Following the set-up of Theorem \ref{orientedpuncturearc}, let $S_{\alpha,T^{\circ}}$ be the snake graph of $\alpha$ with respect to $T^{\circ}$. Then there exists a perfect matching $P_{+}$ of $S_{\alpha,T^{\circ}}$ such that:

$$y_{\mathbf{L_{T}^{\bullet}}}(P_{+}) = \displaystyle \Big(\prod_{\substack{\text{$\gamma$ is a} \\ \text{diagonal of a}\\ \text{tile in $S_{\alpha,T^{\circ}}$}}} y_{\gamma} \Big) \Big(\prod_{\substack{\text{$\gamma$ is a diagonal} \\ \text{of a non-ordinary} \\ \text{tile in $S_{\alpha,T^{\circ}}$}}} y_{\gamma_{i_j}^{(p)}}^{-1} \Big)$$

\noindent We call $P_+$ the \textit{\textbf{maximal matching}} of $S_{\alpha,T^{\circ}}$.

\end{prop}

\begin{thm}
\label{puncturearc}
Let $T^{\circ}$ be an ideal triangulation of a bordered surface $(S,M)$ with corresponding tagged triangulation $T = i(T^{\circ})$, and let $\mathbf{L_T}$ be a principal lamination. Then for any arc $\alpha$ we have:

\begin{equation}
x_{\mathbf{L_T}}(\alpha) = \frac{1}{bad(\alpha, T^{\circ})cross(\alpha, T^{\circ})} {\displaystyle \sum_{P} x(P)y_{\mathbf{L_T}}(P)}
\end{equation}

Where:

\begin{itemize}

\item The sum is over all perfect matchings of the snake graph $\mathcal{G}_{\alpha, T^{\circ}}$.

\item $x_{\mathbf{L_T}}(\alpha)$ is the cluster variable (corresponding to $\alpha$) in the quasi-cluster algebra $\mathcal{A} := \mathcal{A}_{\mathbf{L_T}}(S,M)$ with initial seed $\Sigma_T$.
 
 \end{itemize}

\end{thm}

\begin{proof}
We split the proof into two cases depending on orientability. \newline
 
\underline{Case 1}: $(S,M)$ is an orientable bordered surface. \newline First note that since $(S,M)$ is orientable we have that $bad(\alpha, T^{\circ}) = 1$. Consequently, when $\mathbf{L_T}$ is the principal lamination, $\mathbf{L_T^{\bullet}}$, corresponding to principal coefficients at $\Sigma_T$, the statement is precisely Theorem \ref{orientedpuncturearc}. Following the theme of Theorem \ref{arcexpansioncoefficients} we prove the result for a general principal lamination $\mathbf{L_T}$ by using Theorem \ref{orientablesep} and Proposition \ref{maximalpunctureorientable}. Namely, for any choice of $\epsilon_i \in \{1,-1\}$, we see 

$${x_{\mathbf{L_T^{\bullet}}}(\alpha)|_{Trop(y_1, \ldots, y_n)}}(1,\ldots, 1; y_1^{\epsilon_1}, \ldots, y_n^{\epsilon_n}) = \displaystyle \prod_{j=1}^n y_{j}^{-a_j}$$ where

\[  a_j = \left\{
\begin{array}{ll}
      0, & \epsilon_j = 1 \\
      \# \{\substack{\text{tiles in $S_{\alpha,T}$}\\ \text{corresponding to $\gamma_j$}}\} - \# \{\substack{\text{tiles in $S_{\alpha,T}$ corresponding}\\ \text{to some $\gamma$, if $\gamma_j = \gamma^{(p)}$}}\}, & \epsilon_j = -1 \\
\end{array} 
\right. \]

Therefore, for a given term $x(P)y_{\mathbf{L^{\bullet}_T}}(P)$ in the expansion of $x_{\mathbf{L_T^{\bullet}}}(\alpha)$, the corresponding term in $x_{\mathbf{L_T}}(\alpha)$ will be

$$\frac{x(P)y_{\mathbf{L^{\bullet}_T}}(P)|_{y_j \leftarrow y_j^{\epsilon_j}}}{\displaystyle \prod_{j=1}^n y_{j}^{-a_j}} = x(P)\displaystyle \prod_{j=1}^n y_{j}^{b_j}$$

where

\[  b_j = \left\{
\begin{array}{ll}
      \# \Big\{\substack{\text{tiles in $S_{\alpha,T}$ whose}\\ \text{diagonal $\gamma_j$ is $\mathbf{L^{\bullet}_T}$-oriented}}\Big\} - \# \Big\{\substack{\text{tiles in $S_{\alpha,T}$ corresponding}\\ \text{to some $\mathbf{L_T^{\bullet}} -directed$} \\ \text{radius $\gamma$, if $\gamma_j = \gamma^{(p)}$}}\Big\}, & \epsilon_j = 1 \\
      \# \Big\{\substack{\text{tiles in $S_{\alpha,T}$ whose}\\ \text{diagonal $\gamma_j$ is \underline{not} $\mathbf{L^{\bullet}_T}$-oriented}}\Big\} - \# \Big\{\substack{\text{tiles in $S_{\alpha,T}$ corresponding}\\ \text{to some radius $\gamma$ which is \underline{not}} \\ \text{$\mathbf{L_T^{\bullet}} -directed$, if $\gamma_j = \gamma^{(p)}$}}\Big\}, & \epsilon_j = -1 \\
\end{array} 
\right. \]

Hence $\displaystyle \prod_{j=1}^n y_{j}^{b_j} = y_{\mathbf{L_T}}(P)$ and this concludes the proof of Case 1.  \newline 

\underline{Case 2}: $(S,M)$ is a non-orientable bordered surface. In this setting the statement follows from Case 1 and Theorem \ref{closedcurveexpansioncoefficients}.

\end{proof}

\subsection{Expansion formulae for singly notched arcs}

\begin{defn}
\label{singlynotched}
Let $p$ be a puncture and $\gamma^{(p)}$ be a tagged arc which has one end notched at $p$, and its other end at $q \neq p$ tagged plain. We denote the underlying plain arc of $\gamma^{(p)}$ by $\gamma$. \newline \indent
Moreover, let $\ell_p$ be the unique arc enclosing $p$ in a monogon with radius $\gamma$. We say $\gamma^{(p)}$ is a \textit{\textbf{singly notched arc}} at $p$, and $\ell_p$ is the \textit{\textbf{loop}} corresponding to $\gamma^{(p)}$.

\end{defn}

Following \cite{musiker2011positivity}, to obtain expansion formulae for singly notched arcs $\gamma^{(p)}$ we introduce the notion of $\gamma$-symmetric perfect matchings. Firstly, note that if $\gamma$ has $d$ intersection points with $T^{\circ}$, then $\ell_p$ has $2d +k$ intersection points for $k \geq 1$ (here $k$ is actually the degree of the puncture $p$ in $T^{\circ}$).

\begin{prop}

Consider the snake graph $\mathcal{G}_{\ell_p,T^{\circ}} = (G_1, \ldots, G_{2d +k})$ and let $$\mathcal{H}_{\gamma, 1} := (G_1,\ldots,G_d) \setminus \{NE(G_d)\}$$ denote the subgraph of $G_{\gamma, 1} := (G_1,\ldots,G_d)$ where the north-east vertex of $G_d$ has been removed (and consequently so have the edges $N(G_d)$ and $E(G_d)$).

Similarly, let $$\mathcal{H}_{\gamma, 2} := (G_{d+k+1},\ldots,G_{2d+k}) \setminus \{SW(G_{d+k+1})\}$$ denote the subgraph of $G_{\gamma, 2} := (G_{d+k+1},\ldots,G_{2d+k})$ where the south-west vertex of $G_{d+k+1}$ has been removed (and consequently so have the edges $S(G_{d+k+1})$ and $W(G_{d+k+1})$). \newline \indent

Then, $$\mathcal{H}_{\gamma, 1} \cong \mathcal{H}_{\gamma, 2} \hspace{10mm} \text{and} \hspace{10mm} \mathcal{G}_{\gamma, 1} \cong \mathcal{G}_{\gamma, 2} \cong \mathcal{G}_{\gamma, T^{\circ}}$$

\end{prop}

\begin{defn}

A perfect matching $P$ of $\mathcal{G}_{\ell_p,T^{\circ}}$ is said to be \textit{\textbf{$\gamma$-symmetric}} if: $$P_{\vert_{\mathcal{H}_{\gamma, 1}}} \cong P_{\vert_{\mathcal{H}_{\gamma, 2}}}$$

\end{defn}

\begin{prop}[Lemma 12.5, \cite{musiker2011positivity}]
\label{maximalnotch}
The maximal matching $P_+$ of $\mathcal{G}_{\ell_p,T^{\circ}}$ is $\gamma$-symmetric.

\end{prop}

\begin{defn}

Let $P$ be a $\gamma$-symmetric perfect matching of $\mathcal{G}_{\ell_p,T^{\circ}}$. We define the associated \textbf{\textit{cluster monomial}} and \textbf{\textit{coefficient monomial}}, respectively, as follows: $$\overline{x}(P) := \frac{x(P)}{x(P_{\vert_{\mathcal{G}_{\gamma, i}}})} \hspace{10mm} \text{and} \hspace{10mm} \overline{y}(P) := \frac{y(P)}{y(P_{\vert_{\mathcal{G}_{\gamma, i}}})}.$$

Where the index $i$ is chosen such that $P_{\vert_{\mathcal{G}_{\gamma, i}}}$ is a perfect matching of $\mathcal{G}_{\gamma, i}$ for some $i \in \{1,2\}$ -- this is well defined by [Lemma 12.4, \cite{musiker2011positivity}].

\end{defn}

\begin{defn}

The \textit{\textbf{crossing monomial}} of the singly notched arc $\gamma^{(p)}$ with respect to $T^{\circ}$ is defined as: $$cross(\gamma^{(p)},T^{\circ}) := \frac{cross(\ell_p,T^{\circ})}{cross(\gamma,T^{\circ})}$$

\end{defn}

\begin{defn}

Let $T$ be a triangulation and $\mathbf{L_T^*}$ be a quasi-principal lamination. For any singly notched arc $\gamma^{(p})$ of $(S,M)$ we define $$ bad( \mathbf{L_T^*},\gamma^{(p)}) := \frac{bad( \mathbf{L_T^*},\ell_p)}{bad( \mathbf{L_T^*},\gamma)}$$.
\end{defn}

\begin{thm}
\label{singlynotchedarc}
Let $(S,M)$ be a bordered surface with puncture $p$. Let $T^{\circ}$ be an ideal triangulation with corresponding tagged triangulation $T = i(T^{\circ})$, and let $\mathbf{L_T}$ be a principal lamination. Let $\mathcal{A}_{\mathbf{L_T}}$ be the corresponding quasi-cluster algebra with respect to the seed $\Sigma_T$. Consider any singly notched arc $\gamma^{(p)}$. Then if the underlying arc $\gamma$ is not in $T$, and $T$ contains no notched arc at $p$ then the Laurent expansion of $x_{\gamma^{(p)}}$ with respect to $\Sigma_T$ may be written as follows:

\begin{equation}
x_{\mathbf{L_T}}(\gamma^{(p)}) = \frac{1}{bad(\gamma^{(p)},\mathbf{L_T^*})cross(\gamma^{(p)}, T^{\circ})} {\displaystyle \sum_{P} \overline{x}(P)\overline{y}_{\mathbf{\overline{L}_T^*}}(P)}
\end{equation}

\noindent where the sum is over all $\gamma$-symmetric perfect matchings of the snake graph $\mathcal{G}_{\ell_p, T^{\circ}}$.

\end{thm}

\begin{proof}

When $(S,M)$ is an oriented surface and $\mathcal{A}$ is the cluster algebra with principal coefficients (with respect to $\Sigma_T$), then the statement is precisely Theorem 4.16 of \cite{musiker2011positivity}. Consequently, for any bordered surface $(S,M)$ we may lift the picture to the orientable double cover $\overline{(S,M)}$ and get the following expansion formula: 
$$x_{\mathbf{L^{\bullet}_T}}(\overline{\gamma}^{(p)}) = \frac{1}{cross(\overline{\gamma}^{(p)}, \overline{T}^{\circ})} {\displaystyle \sum_{P} \overline{x}(P)\overline{y}_{\mathbf{L^{\bullet}_{\overline{T}}}}(P)}$$
where $\overline{T}$ and $\overline{\gamma}^{(p)}$ are the lifts of $T$ and ${\gamma}^{(p)}$ respectively; $\mathbf{L^{\bullet}_{\overline{T}}}$ is the principal lamination of $\overline{T}$ corresponding to principal coefficients at $\Sigma_{\overline{T}}$; and the sum is taken over all $\gamma$-symmetric perfect matchings $P$ of the snake graph $\mathcal{G}_{\ell_p, \overline{T}^{\circ}}.$ \newline \indent Using Fomin and Zelevinsky's separation of addition formula and recalling that $x_{\mathbf{L}}(\ell_p) = x_{\mathbf{L}}(\gamma)x_{\mathbf{L}}(\gamma^{(p)})$ for any multi-lamination $\mathbf{L}$ we see that: 
$$\frac{x_{\mathbf{\overline{L}_T^*}}(\overline{\ell}_p)}{x_{\mathbf{\overline{L}_T^*}}(\overline{\gamma})} = x_{\mathbf{\overline{L}_T^*}}(\overline{\gamma}^{(p)}) = \frac{1}{cross(\overline{\gamma}^{(p)}, \overline{T}^{\circ})} {\displaystyle \sum_{P} \overline{x}(P)\overline{y}_{\mathbf{\overline{L}_T^*}}(P)}$$
for any principal lamination $\mathbf{L_T}$.  \newline Moreover, by Theorem \ref{puncturearc} we know that $x_{\mathbf{L_T}}(\delta) = \frac{x_{\mathbf{L_T^*}}(\delta)}{bad(\delta, \mathbf{L_T^*})}$ for any plain arc or loop, $\delta$, of $(S,M)$. Therefore we obtain the following:
$$x_{\mathbf{L_T}}(\gamma^{(p)}) = \frac{bad(\gamma, \mathbf{L_T^*})x_{\mathbf{\overline{L}_T^*}}(\overline{\ell}_p)}{bad(\ell_p, \mathbf{L_T^*})x_{\mathbf{\overline{L}_T^*}}(\overline{\gamma})} = \frac{1}{bad(\gamma^{(p)}, \mathbf{L_T^*})cross(\gamma^{(p)}, T^{\circ})} {\displaystyle \sum_{P} \overline{x}(P)\overline{y}_{\mathbf{\overline{L}_T^*}}(P)}$$
where the sum is taken over all $\gamma$-symmetric perfect matchings $P$ of the snake graph $\mathcal{G}_{\ell_p, T^{\circ}}$.
\end{proof}

\subsection{Expansion formulae for doubly notched arcs}

\begin{defn}

Let $p$ and $q$ be punctures and $\gamma^{(pq)}$ be a tagged arc which has one end notched at $p$, and its other end notched at $q$ (we allow $p = q$). \indent Following Definition \ref{singlynotched}, we define $\ell_p$ and $\ell_q$ in the same fashion -- note, however, when $p=q$ then $\ell_p$ and $\ell_q$ are not strictly arcs as they are self-intersecting. \newline \indent We say $\gamma^{(pq)}$ is a \textit{\textbf{doubly notched arc}} at $p$ and $q$, and $\ell_p$ and $\ell_q$ are the two \textit{\textbf{loops}} corresponding to $\gamma^{(pq)}$.
\end{defn}

Following \cite{musiker2011positivity}, to obtain expansion formulae for doubly notched arcs $\gamma^{(pq)}$ with respect to an ideal triangulation $T^{\circ}$ we introduce the notion of a pair of $\gamma$-symmetric perfect matchings being \textit{$\gamma$-compatible}.

\begin{defn}
\label{compatiblepair}
Let $\gamma^{(pq)}$ be a doubly notched arc and let $P_p$ and $P_q$ be $\gamma$-symmetric perfect matchings of $\mathcal{G}_{\ell_p, T^{\circ}}$ and $\mathcal{G}_{\ell_p, T^{\circ}}$, respectively. \newline \indent The pair $(P_p, P_q)$ is called \textbf{\textit{$\gamma$-compatible}} if the following two conditions hold for some $i,j \in \{1,2\}$:

\begin{itemize}

\item ${P_p}_{\vert_{\mathcal{G}_{\gamma, i}}} \cong {P_q}_{\vert_{\mathcal{G}_{\gamma, j}}}$

\item ${P_p}_{\vert_{\mathcal{G}_{\gamma, i}}}$ and ${P_q}_{\vert_{\mathcal{G}_{\gamma, j}}}$ are perfect matchings of $\mathcal{G}_{\gamma, i}$ and $\mathcal{G}_{\gamma, j}$, respectively.
 
\end{itemize} 

\end{defn}

\begin{defn}

Let $(P_p,P_q)$ be a $\gamma$-compatible pair of $\gamma$-symmetric perfect matchings of $(\mathcal{G}_{\ell_p,T^{\circ}},\mathcal{G}_{\ell_q,T^{\circ}})$. We define the associated \textbf{\textit{cluster monomial}} and \textbf{\textit{coefficient monomial}}, respectively, as follows: $$\overline{\overline{x}}(P_p,P_q) := \frac{x(P_p)x(P_q)}{x({P_p}_{\vert_{\mathcal{G}_{\gamma, i}}})^3} \hspace{10mm} \text{and} \hspace{10mm} \overline{\overline{y}}(P_p,P_q) := \frac{y(P_p)y(P_q)}{y({P_p}_{\vert_{\mathcal{G}_{\gamma, i}}})^3}.$$

Where the index $i \in \{1,2\}$ is chosen such that ${P_p}_{\vert_{\mathcal{G}_{\gamma, i}}} \cong {P_q}_{\vert_{\mathcal{G}_{\gamma, j}}}$ for some $j \in \{1,2\}$, as in Definition \ref{compatiblepair}. This is well defined by [Lemma 12.4, \cite{musiker2011positivity}].

\end{defn}

\begin{defn}

The \textit{\textbf{crossing monomial}} of the doubly notched arc $\gamma^{(pq)}$ with respect to $T^{\circ}$ is defined as: $$cross(\gamma^{(pq)},T^{\circ}) := \frac{cross(\ell_p,T^{\circ})cross(\ell_q,T^{\circ})}{cross(\gamma,T^{\circ})^3}$$

\end{defn}

\begin{thm}[Theorem 4.2, \cite{musiker2011positivity}]
\label{doublynotchedarcorientable}
Let $(S,M)$ be an orientable bordered surface which is not a once or twice punctured closed surface. Let $T^{\circ}$ be an ideal triangulation with corresponding tagged triangulation $T = i(T^{\circ})$, and let $\mathbf{L_T^{\bullet}}$ be a principal lamination corresponding to principal coefficients at $\Sigma_T$. Then for any doubly notched arc $\gamma^{(pq)}$ of $(S,M)$, if $\gamma$ is not in $T$ and $T$ contains no notched arcs at $p$ or $q$, we have the following expansion formula:

\begin{equation}
 x_{\mathbf{L_T^{\bullet}}}(\gamma^{(pq)}) = \frac{1}{cross(\gamma^{(pq)}, T^{\circ})} {\displaystyle \sum_{(P_p,P_q)} \overline{\overline{x}}(P_p,P_q)\overline{\overline{y}}(P_p,P_q)}
\end{equation}

Where:

\begin{itemize}

\item The sum is over all $\gamma$-compatible pairs of $\gamma$-symmetric perfect matchings of the pair of snake graphs $(\mathcal{G}_{\ell_p, T^{\circ}},\mathcal{G}_{\ell_q, T^{\circ}})$.

\item $x_{\mathbf{L_T^{\bullet}}}(\gamma^{(pq)})$ is the cluster variable (corresponding to $\gamma^{(pq)}$) in the cluster algebra, $\mathcal{A}_{\mathbf{L_T^{\bullet}}}(S,M)$, with principal coefficients at $\Sigma_T$.
 
 \end{itemize}

\end{thm}

\begin{prop}
\label{maximaldoublenotch}
Let ${P_p}_+$ and ${P_q}_+$ be the maximal matchings of $\mathcal{G}_{\ell_p,T^{\circ}}$ and $\mathcal{G}_{\ell_q,T^{\circ}}$, respectively. Then $({P_p}_+,{P_q}_+)$ is $\gamma$-compatible.

\end{prop}

\begin{proof}

Recall that by Lemma \ref{maximalnotch} we know ${P_p}_+$ and ${P_q}_+$ are $\gamma$-symmetric perfect matchings. Hence, by the definition of $\gamma$-symmetric and the uniqueness part of Lemma \ref{maximalmatchingtriangulation}, ${P_p}_+$ and ${P_q}_+$ must restrict to the maximal matching $P_+$ of $\mathcal{G}_{\gamma,T^{\circ}}$ at one of the endpoints of $\mathcal{G}_{\ell_p,T^{\circ}}$ and $\mathcal{G}_{\ell_q,T^{\circ}}$, respectively. This concludes the proof.

\end{proof}

\begin{thm}
\label{doublynotchedarc}
Let $(S,M)$ be a bordered surface which is not a once punctured closed surface nor a twice punctured orientable closed surface. Let $T^{\circ}$ be an ideal triangulation with corresponding tagged triangulation $T = i(T^{\circ})$, and let $\mathbf{L_T}$ be a principal lamination. Then for any doubly notched arc $\gamma^{(pq)}$ of $(S,M)$, if $\gamma$ is not in $T$ and $T$ contains no notched arcs at $p$ or $q$, we have the following expansion formula:

\begin{equation}
 x_{\mathbf{L_T}}(\gamma^{(pq)}) = \frac{1}{bad(\gamma^{(pq)},\mathbf{L_T^*})cross(\gamma^{(pq)}, T^{\circ})} {\displaystyle \sum_{(P_p,P_q)} \overline{\overline{x}}(P_p,P_q)\overline{\overline{y}}(P_p,P_q)}
\end{equation}

Where:

\begin{itemize}

\item The sum is over all $\gamma$-compatible pairs of $\gamma$-symmetric perfect matchings of the pair of snake graphs $(\mathcal{G}_{\ell_p, T^{\circ}},\mathcal{G}_{\ell_q, T^{\circ}})$.

\item $x_{\mathbf{L_T}}(\gamma^{(pq)})$ is the cluster variable (corresponding to $\gamma^{(pq)}$) in the quasi-cluster algebra, $\mathcal{A}_{\mathbf{L_T}}(S,M)$, associated to the principal lamination $\mathbf{L_T}$.
 
 \end{itemize}

\end{thm}

\begin{proof}
As in the proof of Theorem \ref{singlynotchedarc} we shall split proof into two cases.

\textbf{Case 1}: $(S,M)$ is an orientable bordered surface. \newline The proof of this case follows analogous to the argument used for Theorem \ref{puncturearc}. Namely, with respect to an arbitrary principal lamination $\mathbf{L_T}$, one applies the separation of additions formulae to Theorem \ref{doublynotchedarcorientable} and employs Proposition \ref{maximaldoublenotch}.

\textbf{Case 2}: $(S,M)$ is a non-orientable bordered surface. \newline In this case we follow the usual procedure of lifting the picture to the orientable double cover $\overline{(S,M)}$. Note that since $(S,M)$ is not a once-punctured closed surface, $\overline{(S,M)}$ is not a once or twice punctured closed surface. Therefore we may apply Case 1 to obtain an expansion formula for the lifted tagged arc $\overline{\gamma}^{(pq)}$ with respect to the lifted principal quasi-lamination $\mathbf{\overline{L}_T^*}$. Consequently, we obtain 
\begin{equation}
x_{\mathbf{\overline{L}_T^*}}(\overline{\gamma}^{(pq)}) = \frac{1}{bad(\gamma^{(pq)},\mathbf{\overline{L}_T^*})cross(\gamma^{(pq)}, \overline{T}^{\circ})} {\displaystyle \sum_{(P_p,P_q)} \overline{\overline{x}}(P_p,P_q)\overline{\overline{y}}(P_p,P_q)}
\end{equation}

\noindent where the sum is over all $\gamma$-compatible pairs of $\gamma$-symmetric perfect matchings of the pair of snake graphs $(\mathcal{G}_{\ell_p, T^{\circ}},\mathcal{G}_{\ell_q, T^{\circ}})$. Applying Theorem \ref{separation} (which is proven in Section \ref{allpositivity}) completes the proof.

\end{proof}

\begin{rmk}

Note that in Theorem \ref{doublynotchedarc} we assume none of the tagged arcs $\gamma$, $\gamma^{(p)}$, $\gamma^{(q)}$ are in $T$. However, in view of Theorem \ref{singlynotchedarc}, the case where $\gamma^{(p)}$ or $\gamma^{(q)}$ is in $T$ has already been treated. Therefore our remaining task is to obtain an expansion for $x_{\mathbf{L_T}}(\gamma^{(pq)})$ when $\gamma$ is in $T$.
\end{rmk}

\begin{defn}

Let $\delta$ be a tagged arc of some bordered surface $(S,M)$. Suppose $p$ is a puncture of $(S,M)$. Then $e_p(\delta)$ is defined as the number of endpoints of $\delta$ incident to $p$.

\end{defn}

\begin{thm}
\label{doublynotchedarc2}
Let $(S,M)$ be an orientable bordered surface which is not a once or twice puncture closed surface. Let $T$ and $\mathbf{L_T}$ be as in Theorem \ref{doublynotchedarc}. Then for any doubly notched arc $\gamma^{(pq)}$ of $(S,M)$, if $\gamma$ is an arc in $T$ and $T$ contains no notched arcs at $p$ or $q$, we have the following expansion formula:

\begin{equation}
 x_{\mathbf{L_T}}(\gamma^{(pq)}) = \frac{x_{\mathbf{L_T}}(\gamma^{(p)})x_{\mathbf{L_T}}(\gamma^{(q)})y_{\gamma} + \displaystyle \Big(\prod_{\substack{\delta \in T \\ L_{\delta}=-1}}y_{\delta}^{e_p(\delta)}-\prod_{\substack{\delta \in T \\ L_{\delta}=1}}y_{\delta}^{e_p(\delta)}\Big)\Big(\prod_{\substack{\delta \in T \\ L_{\delta}=-1}}y_{\delta}^{e_q(\delta)}-\prod_{\substack{\delta \in T \\ L_{\delta}=1}}y_{\delta}^{e_q(\delta)}\Big)}{x_{\mathbf{L_T}}(\gamma)}
\end{equation}

\noindent where, by abuse of notation, $L_{\delta}=\pm 1$ is used as shorthand for $b_{T}(\delta,L_{\delta})=\pm 1$.

Moreover, the two negative terms in this expansion cancel with terms in $x_{\mathbf{L_T}}(\gamma^{(p)})x_{\mathbf{L_T}}(\gamma^{(q)})y_{\gamma}$. Consequently, the resulting expansion has positive coefficients.
\end{thm}

\begin{proof}

When $L_T$ is the principal lamination corresponding to principal laminations at $(S,M)$ the statement is found in (the proof of) [Proposition 4.21, \cite{musiker2011positivity}]. The statement then follows by an application of Theorem \ref{orientablesep}.

\end{proof}

\begin{thm}

Let $(S,M)$, $T$ and $\mathbf{L_T}$ be as in Theorem \ref{doublynotchedarc}. Then for any doubly notched arc $\gamma^{(pq)}$ of $(S,M)$, if $\gamma$ is an arc in $T$ and $T$ contains no notched arcs at $p$ or $q$, we have the following expansion formula:

\begin{equation}
 x_{\mathbf{L_T}}(\gamma^{(pq)}) = \frac{x_{\mathbf{L_T^*}}(\gamma^{(pq)})}{bad(\mathbf{L_T^*},\gamma^{(p)})bad(\mathbf{L_T^*},\gamma^{(q)})}
\end{equation}

\end{thm}

\begin{proof}

This follows from Theorem \ref{singlynotchedarc}, Theorem \ref{doublynotchedarc2} and an application of Theorem \ref{separation}.

\end{proof}

\subsection{Expansion formulae for one-sided closed curves}

\begin{thm}
\label{closedcurveexpansioncoefficients3}
Let $(S,M)$ be a bordered surface and let $T^{\circ}$ be an ideal triangulation with corresponding tagged triangulation $T = i(T^{\circ})$. Consider a principal lamination $\mathbf{L_T}$ and let $\mathcal{A}_{\mathbf{L_T}}$ be the corresponding quasi-cluster algebra with respect to the seed $\Sigma_T$. Then for any one-sided closed curve $\alpha$ we have:

\begin{equation}
\label{onesidedbadexpansion}
x_{\mathbf{L_T}}(\alpha) = \frac{1}{bad(\mathbf{L_T},\alpha)cross(\alpha, T)} {\displaystyle \sum_{P} x(P)y_{\mathbf{L_T^*}}(P)}
\end{equation}

Where:

\begin{itemize}

\item The sum is over all good matchings of the band graph $S_{\alpha, T}$.

\item $x_{\mathbf{L_T}}(\alpha)$ is the cluster variable (corresponding to $\alpha$) in the quasi-cluster algebra, $\mathcal{A}_{\mathbf{L_T}}(S,M)$, associated to the principal lamination $\mathbf{L_T}$.
 
 \end{itemize}

\end{thm}

\begin{proof}

Note that our proofs of Proposition \ref{splitproduct}, Proposition \ref{badencountermobius} and Proposition \ref{badsplitproductcoeffcient} also hold in the setting of ideal triangulations with self-folded triangles. Therefore we may employ the same argument used in the proof of Theorem \ref{closedcurveexpansioncoefficients2}.

\end{proof}

\section{Extending results to all quasi-triangulations}

So far we have only considered expansion formulae for tagged quasi-arcs with respect to an initial triangulation. We shall now explain how one can use this to obtain expansion formulae with respect to any quasi-triangulation. First we need to extend our notion of a principal lamination to quasi-triangulations. Recall that for a fixed quasi-triangulation $T$, each one-sided closed curve $\alpha \in T$ intersects precisely one arc $\beta_{\alpha} \in T$. Moreover, there exists a unique tagged arc $\gamma_{\alpha}$ enclosing $\alpha$ and $\beta_{\alpha}$ in a M\"{o}bius strip with one marked point on the boundary. We define the \textit{\textbf{associated triangulation}} of $T$ as follows: $$T^{\otimes} := T\setminus\{\alpha | \substack{\text{$\alpha \in T$ is a one-sided} \\ \text{closed curve}}\}\cup \{\gamma_{\alpha} | \substack{\text{$\alpha \in T$ is a one-sided} \\ \text{closed curve}}\}$$

\begin{defn}

Let $T$ be a quasi-triangulation with associated triangulation $T^{\otimes}$. We define a \textit{\textbf{principal lamination}} $\mathbf{L_T}$ of $T$ to be a principal lamination $\mathbf{L_{T^{\otimes}}}$ of $T^{\otimes}$.

\end{defn}

Let $T$ be a quasi-triangulation of a non-orientable bordered surface $(S,M)$. Then for any tagged quasi-arc $\delta$, using the results of Section \ref{alltriangulations}, we may obtain an expansion of $x_{\mathbf{L_T}}(\delta)$ with respect to the initial collection of variables $\{x_{\mathbf{L_T}}(\gamma_i) | \gamma_i \in T^{\otimes}\}$. Finally, recall that $$x_{\mathbf{L_T}}(\gamma_{\alpha}) = x_{\mathbf{L_T}}(\alpha)x_{\mathbf{L_T}}(\beta_{\alpha})$$ for each one-sided closed curve $\alpha$. Consequently, making the substitution $x_{\mathbf{L_T}}(\gamma_{\alpha}) \rightarrow x_{\mathbf{L_T}}(\alpha)x_{\mathbf{L_T}}(\beta_{\alpha})$ for each $\gamma_{\alpha} \in T^{\otimes}$ provides us with a positive Laurent polynomial expansion of $x_{\mathbf{L_T}}(\delta)$ with respect to the initial seed $\Sigma_T$.

\section{Positivity for quasi-cluster algebras with arbitrary coefficients}
\label{allpositivity}
In order to prove positivity for all coefficient systems, not just those coming from principal laminations, we shall prove that Fomin and Zelevinsky's `\textit{Separation of additions}' formula extends to quasi-cluster algebras. To achieve this we shall need a modified version of their Y-dynamics.

\subsection{Y-dynamics for quasi-cluster algebras}

Let $\mathbb{T}_n$ be the (labelled) $n$-regular tree where edges are labelled by the numbers $1,\ldots, n$ such that the $n$ edges incident to a vertex receive different labels.

\begin{defn}
\label{labelledseedpattern}
A \textit{\textbf{(labelled) seed pattern}} on $\mathbb{T}_n$ is an assignment of a (labelled) quasi-cluster algebra seed to each vertex in $\mathbb{T}_n$, such that if two seeds are connected by an edge labelled $k$, then they are related by a mutation in direction $k$.

\end{defn}

Let us fix an initial seed $t_0 \in \mathbb{T}_n$ which corresponds to a triangulation $T_0$ of $(S,M)$. We shall denote by $\mathbb{T}_n^\Delta$ the subgraph of  $\mathbb{T}_n$ consisting of all seeds corresponding to triangulations, which are connected to $t_0$ via sequences of mutations corresponding to flips between triangulations.

\begin{defn}
\label{positivenegative}
For each $t \in \mathbb{T}_n^\Delta$ note that the collection of exchange polynomials, $F_t := \{F_{1;t},\ldots, F_{n;t}\}$, corresponding to $t$ are binomials. For each binomial $F_{j;t}$ we make a choice regarding which monomial is \textit{\textbf{positive}} and which monomial is \textit{\textbf{negative}}. We indicate such a choice using the following notation: $$F_{j;t} = M_{j;t}^{+} + M_{j;t}^{-}$$

The \textit{\textbf{$\textbf{Y}$-seed}} at $t$ is then defined as $Y_{t} := \{Y_{1:t},\ldots, Y_{n;t}\}$ where: $$ Y_{j:t} := \frac{M_{j;t}^{+}}{M_{j;t}^{-}}.$$

\end{defn}

\begin{rmk}
\label{uniquematrix}
Recall that the shortened exchange matrix of a seed is only defined up to multiplication of its columns by $-1$. However, with respect to our choice of a `positive' and 'negative' monomial for each $F_{j;t} \in F_t$ we can uniquely assign a \textit{\textbf{shortened exchange}} matrix $\overline{B}_t = (\overline{b}_{ij})_{\substack{1 \leq i \leq m \\ 1 \leq j \leq n}}$ to the seed $t$ such that for each $j \in \{1,\ldots,n\}$ we have: $$ M_{j;t}^{+} = \displaystyle \prod_{i=1}^{m} x_{i;t}^{[\overline{b}_{ij}]_+} \hspace{10mm} \text{and} \hspace{10mm} M_{j;t}^{-} = \displaystyle \prod_{i=1}^{m} x_{i;t}^{[-\overline{b}_{ij}]_+}.$$ \newline \indent With respect to this notation, our \textit{Y-variables} in the seed $t$ may be written as: $$ Y_{j:t} = \displaystyle \prod_{i=1}^{m} x_{i;t}^{\overline{b}_{ij}}.$$

\end{rmk}

\begin{defn}
\label{prepostmultiplication}
By [Lemma 6.23,\cite{wilson2018laurent}], to perform matrix mutation on $\overline{B}_t$, in a direction $k$, we may need to multiply some of its columns indexed by $\{1,\ldots,n\}\setminus \{k\}$ by $-1$ to obtain $\overline{B}_t^* = (\epsilon_{j;t}\overline{b}_{ij})_{\substack{1 \leq i \leq m \\ 1 \leq j \leq n}}$, which is mutable at $k$ -- we call this \textbf{\textit{pre-multiplication}} with respect to $k$. \newline \indent Similarly, we need to multiply the columns of $\mu_k(\overline{B}_t^*)$ indexed by $j \in \{1,\ldots,n\}$ by $\mu_{j;t} \in \{1,-1\}$ to obtain $\overline{B}_t$ (see Remark \ref{uniquematrix}) -- we call this \textbf{\textit{post-multiplication}} with respect to $k$.

\end{defn}

\begin{prop}
\label{ymutation}
Let $t$ and $t'$ be seeds in $\mathbb{T}_n^{\Delta}$ which are connected by an edge $k$. With respect to our choices and notation discussed in Definition \ref{positivenegative} and Remark \ref{uniquematrix}, the $Y$-seeds $Y_t = \{Y_{1,t},\ldots, Y_{n;t}\}$ and $Y_{t'} = \{Y_{1;t'} \ldots Y_{n;t'}\}$ are related as follows:

\[  Y_{j;t'}^{\mu_{j;t}}  = \left\{
\begin{array}{ll}
         {Y_{k;t}^{-1}}, &j=k. \\
        \\
        {Y_{j;t}^{\epsilon_{j;t}}}{Y_{k;t}^{[\overline{b}_{kj}]_+}}{({Y_{k;t}^{\epsilon_{j;t}}}+1)^{-\epsilon_{j;t}\overline{b}_{kj}}}, &j\neq k.
\end{array} 
\right. \]

Where: \begin{itemize}

\item $\overline{B}_t = (\overline{b}_{ij})_{\substack{1 \leq i \leq m \\ 1 \leq j \leq n}}$

\item $\epsilon_{j;t}$ (resp. $\mu_{j;t})$ arises from the required pre-multiplication (resp. post-multiplication) occurring in matrix mutation -- see Definition \ref{prepostmultiplication}.

\end{itemize}

\end{prop}

\begin{proof}

Let us denote the entries of the matrix $(\overline{B}_{t'})$ by $\overline{b}'_{ij}$. \newline \indent 

If $j = k$ then by [Lemma 6.23,\cite{wilson2018laurent}] we see: $$ Y_{k;t'}^{\mu_{k;t}} = \displaystyle \prod_{i=1}^{m} x_{i;t'}^{\mu_{k;t}\overline{b}'_{ij}} = \displaystyle \prod_{i=1}^{m} x_{i;t}^{-\overline{b}_{ij}} = Y_{k;t}^{-1}$$

If $j \neq k$ then by employing [Lemma 6.23,\cite{wilson2018laurent}] again we see: $$ Y_{j;t'}^{\mu_{j;t}} = \displaystyle \prod_{i=1}^{m} x_{i;t'}^{\mu_{j;t}\overline{b}'_{ij}} = x_{k;t'}^{-\epsilon_{j;t}\overline{b}_{kj}}\displaystyle \prod_{\substack{i=1\\ (i\neq k)}}^{m} x_{i;t}^{\epsilon_{j;t}\overline{b}_{ij}+sgn(\epsilon_{j;t}\overline{b}_{kj})[\epsilon_{j;t}\overline{b}_{ik}\overline{b}_{kj}]_+} = $$ $$ = (\displaystyle \prod_{i=1}^{m} x_{i;t}^{[\overline{b}_{ik}]_+} + \displaystyle \prod_{i=1}^{m} x_{i;t}^{[-\overline{b}_{ik}]_+})^{-\epsilon_{j;t}\overline{b}_{kj}}\Big(\displaystyle \prod_{i=1}^{m} x_{i;t}^{\epsilon_{j;t}\overline{b}_{ij}}\Big)\Big(\displaystyle \prod_{i=1}^{m} x_{i;t}^{sgn(\epsilon_{j;t}\overline{b}_{kj})[\epsilon_{j;t}\overline{b}_{ik}\overline{b}_{kj}]_+}\Big)$$

Moreover, since we have the following identity: \[  sgn(\epsilon_{j;t}\overline{b}_{kj})[\epsilon_{j;t}\overline{b}_{ik}\overline{b}_{kj}]_+  = \left\{
\begin{array}{ll}
         \epsilon_{j;t}[\epsilon_{j;t}\overline{b}_{ik}]_+\overline{b}_{kj}, &\overline{b}_{kj} \geq 0. \\
        \\
        \epsilon_{j;t}[-\epsilon_{j;t}\overline{b}_{ik}]_+\overline{b}_{kj}, & \overline{b}_{kj} \leq 0.
\end{array} 
\right. \]

Then we obtain: \[ Y_{j;t'}^{\mu_{j;t}} =  \left\{
\begin{array}{ll}
         Y_{j;t}^{\epsilon_{j;t}}(\displaystyle \prod_{i=1}^{m} x_{i;t}^{[\overline{b}_{ik}]_+} + \displaystyle \prod_{i=1}^{m} x_{i;t}^{[-\overline{b}_{ik}]_+})^{-\epsilon_{j;t}\overline{b}_{kj}}\Big(\displaystyle \prod_{i=1}^{m} x_{i;t}^{-[\epsilon_{j;t}\overline{b}_{ik}]_+}\Big)^{-\epsilon_{j;t}\overline{b}_{kj}}, &\overline{b}_{kj} \geq 0. \\
        \\
        Y_{j;t}^{\epsilon_{j;t}}(\displaystyle \prod_{i=1}^{m} x_{i;t}^{[\overline{b}_{ik}]_+} + \displaystyle \prod_{i=1}^{m} x_{i;t}^{[-\overline{b}_{ik}]_+})^{-\epsilon_{j;t}\overline{b}_{kj}}\Big(\displaystyle \prod_{i=1}^{m} x_{i;t}^{-[-\epsilon_{j;t}\overline{b}_{ik}]_+}\Big)^{-\epsilon_{j;t}\overline{b}_{kj}}, & \overline{b}_{kj} \leq 0.
\end{array} 
\right. \]

\[ \hspace{-48mm} = \left\{
\begin{array}{ll}
         Y_{j;t}^{\epsilon_{j;t}}(1+ Y_{k;t}^{-\epsilon_{j;t}})^{-\epsilon_{j;t}\overline{b}_{kj}}, & \overline{b}_{kj} \geq 0. \\
        \\
         Y_{j;t}^{\epsilon_{j;t}}(1+ Y_{k;t}^{\epsilon_{j;t}})^{-\epsilon_{j;t}\overline{b}_{kj}}, & \overline{b}_{kj} \leq 0.
\end{array} 
\right. \] \newline

$ \hspace{3.5mm}= \hspace{2mm} {Y_{j;t}^{\epsilon_{j;t}}}{Y_{k;t}^{[\overline{b}_{kj}]_+}}{({Y_{k;t}^{\epsilon_{j;t}}}+1)^{-\epsilon_{j;t}\overline{b}_{kj}}}.$ \newline \indent 

This concludes the proof.

\end{proof}

Proposition \ref{ymutation} gives us the following immediate corollary. 

\begin{cor}
\label{ymuationcorollary}
The Y-variables, $$Y_{j;t} := \displaystyle \prod_{i=1}^{m} {x_{i;t}^{{b}^{t}_{ij}}}_{{\vert}_{\mathcal{F}}}(x_1,\dots,x_m),$$ are functions in the initial Y-variables, $Y_{j;t_0}$. So we may write: $$Y_{j;t} = Y_{j;t}(Y_{1;t_0}, \ldots, Y_{n;t_0}).$$ \indent Moreover, the exchange relations between the Y-seeds are independent of the extended part of the exchange matrix (note, however, the initial Y-variables do depend on the extended part of the exchange matrix).

\end{cor}

\subsection{Separation of additions for quasi-cluster algebras}
\label{Separation}

\begin{lem}

Let us define $y_j := \displaystyle \prod_{i=n+1}^m x_i^{b_{ij}}$ and $ \epsilon_{j} := b_{\gamma_{j}}(L_{\gamma_{j}}^{\bullet},\overline{T})$. Then:

\label{ydynamics}
\begin{equation}\label{normalydynamics}
\displaystyle \prod_{i=1}^{m} {x_{i;t}^{{b}^{t}_{ik}}}_{{\vert}_{\mathcal{F}}}(x_1,\dots,x_m) = \displaystyle \prod_{i=1}^{2n} {X_{i;t}^{\overset{\bullet}{b}^{t}_{ik}}}_{{\vert}_{\mathcal{F}}}(x_1,\ldots,x_n,y_1^{\epsilon_1}, \ldots, y_n^{\epsilon_n})
\end{equation}

and 

\begin{equation}\label{tropicalydynamics}
\displaystyle \prod_{i=n+1}^{m} {x_{i;t}^{{b}^{t}_{ik}}}_{{\vert}_{Trop(x_{n+1},\ldots,x_m)}}(x_1,\dots,x_m) = \displaystyle \prod_{i=1}^{2n} {X_{i;t}^{\overset{\bullet}{b}^{t}_{ik}}}_{{\vert}_{Trop(x_{n+1},\ldots,x_m)}}(1,\ldots,1,y_1^{\epsilon_1}, \ldots, y_n^{\epsilon_n}).
\end{equation}

\end{lem}

\begin{proof}

Let us begin by proving the first statement, (\ref{normalydynamics}). By the definition of $Y$-dynamics and Corollary \ref{ymuationcorollary}, we see that (\ref{normalydynamics}) may be rewritten as follows:

\begin{equation}
Y_{k;t}(Y_{1;t_0},\ldots,Y_{n;t_0}) =  Y_{k;t}(Y^{\bullet}_{1;t_0},\ldots,Y^{\bullet}_{n;t_0})_{{\vert}_{x_{n+j} \leftarrow y_j^{\epsilon_j}}}
\end{equation}

So it suffices to show that $$Y_{j;t_0} = {Y^{\bullet}_{j;t_0}}_{{\vert}_{x_{n+j} \leftarrow y_j^{\epsilon_j}}}$$ for each $j \in \{1,\ldots,n\}$. Indeed, $$Y^{\bullet}_{j;t_0} = \displaystyle \prod_{i=1}^{2n} X_{i;t}^{\overset{\bullet}{b}^{t}_{ik}} = (\prod_{i=1}^{n} X_{i;t}^{b^{t}_{ik}})x_{n+j}^{\epsilon_{j}}.$$

So, $${Y^{\bullet}_{j;t_0}}_{{\vert}_{x_{n+j} \leftarrow y_j^{\epsilon_j}}} = (\prod_{i=1}^{n} X_{i;t}^{b^{t}_{ik}})y_{j} = Y_{j;t_0}.$$ \newline

The second statement, (\ref{tropicalydynamics}), then follows from (\ref{normalydynamics}) after evaluating in $Trop(x_{n+1}, \ldots, x_m)$ and noting that for any cluster variable $X_{i;t}$, we have $${X_{i;t}}_{{\vert}_{Trop(x_{n+1}, \ldots, x_m)}} = 1$$.

\end{proof}

The following is a generalisation of Fomin and Zelevinsky's `separation of additions' formula [Theorem 3.7, \cite{fomin2007cluster}].

\begin{thm}
\label{separation}
Let $\mathcal{A}(S,M,\mathbf{L})$ be a quasi-cluster algebra with initial seed $t_0$ corresponding to some tagged quasi-triangulation $T$. Then the cluster variables $x_{k;t}$ in $\mathcal{A}(S,M,\mathbf{L})$ can be expressed as follows:

$$x_{k;t}(x_1,\ldots, x_n, x_{n+1},\ldots, x_m) = \frac{X_{k;t} \vert_{\mathcal{F}}(x_1,\ldots, x_n, y_1^{\epsilon_1},\ldots, y_n^{\epsilon_n})}{X_{k;t} \vert_{Trop(x_{n+1},\ldots,x_m)}(1,\ldots, 1, y_1^{\epsilon_1},\ldots, y_n^{\epsilon_n})}$$

where:

\begin{itemize}

\item $y_j := \displaystyle \prod_{i=n+1}^m x_i^{\overline{b}_{ij}}.$

\item $ \epsilon_{j} := b_{\gamma_{j}}(L_{\gamma_{j}}^{\bullet},\overline{T^{\otimes}})$

\end{itemize}

\end{thm}

\begin{proof} We split the proof in to two cases depending on whether or not $T$ contains one-sided closed curves. \newline

\noindent \textbf{Case 1:} $T$ is a tagged triangulation. \newline \indent Note that, given a triangulation $T$, every quasi-cluster variable can be obtained via a sequence of flips on arcs. We therefore prove the theorem by induction on the number of flips of arcs. \newline

Directly from the exchange relations we know:

\begin{equation} \label{exchangerelation}
{x_{k;t'}}_{{\vert}_{\mathcal{F}}}(\mathbf{x}) = \frac{\Big(\displaystyle \prod_{i=1}^m {x_{i;t}^{b^{t}_{ik}}}_{{\vert}_{\mathcal{F}}}(\mathbf{x})+1\Big)_{{\vert}_{\mathcal{F}}}}{\Big(\displaystyle \prod_{i=n+1}^m {x_{i;t}^{b^{t}_{ik}}}_{{\vert}_{\mathcal{F}}}(\mathbf{x})\oplus1\Big)_{{\vert}_{Trop(x_{n+1},\ldots,x_m)}}}\Big( \displaystyle \prod_{i=1}^n {x_{i;t}^{[-b^{t}_{ik}]_{+}}}_{{\vert}_{\mathcal{F}}}(\mathbf{x})\Big){\Big(x_{k:t}^{-1}}_{{\vert}_{\mathcal{F}}}(\mathbf{x})\Big)
\end{equation}
where $\mathbf{x} := (x_1,\ldots, x_m)$. \newline

Since equation (\ref{exchangerelation}) is true for every cluster algebra, we may consider the cluster algebra with principal lamination $L$ and evaluate $x_{n+j}$ at $y_j^{\epsilon_j}$ for each $j \in \{1,\ldots,n\}$ to get:

\begin{equation} \label{evaluatingprincipal}
{X_{k;t'}}_{{\vert}_{\mathcal{F}}}(\mathbf{y}) = \frac{\Big(\displaystyle \prod_{i=1}^{2n} {X_{i;t}^{\overset{\bullet}{b}^{t}_{ik}}}_{{\vert}_{\mathcal{F}}}(\mathbf{y})+1\Big)_{{\vert}_{\mathcal{F}}}}{\Big(\displaystyle \prod_{i=n+1}^{2n} {X_{i;t}^{\overset{\bullet}{b}_{ik}^t}}_{{\vert}_{\mathcal{F}}}(\mathbf{y})\oplus1\Big)_{{\vert}_{Trop(y_1^{\epsilon_1},\ldots,y_n^{\epsilon_n})}}}\Big( \displaystyle \prod_{i=1}^n {X_{i;t}^{[-\overset{\bullet}{b}^{t}_{ik}]_{+}}}_{{\vert}_{\mathcal{F}}}(\mathbf{y})\Big){\Big(X_{k:t}^{-1}}_{{\vert}_{\mathcal{F}}}(\mathbf{y})\Big)
\end{equation}

where $\mathbf{y} := (x_1,\ldots, x_n, {y_1}^{\epsilon_1},\ldots, {y_n}^{\epsilon_n})$. \newline

Specialising (\ref{evaluatingprincipal}) at $x_1 = \ldots, = x_n = 1$ and evaluating in $Trop(x_1, \ldots, x_n)$ we get:

\begin{equation} \label{specialisingprincipal}
\begin{aligned}
&{X_{k;t'}}_{{\vert}_{Trop(x_{n+1},\ldots,x_m)}}(\mathbf{z}) = \\ = & \frac{\Bigg(\Big(\displaystyle \prod_{i=1}^{2n} {X_{i;t}^{\overset{\bullet}{b}^{t}_{ik}}}_{{\vert}_{\mathcal{F}}}(\mathbf{z})+1\Big)\Big( \displaystyle \prod_{i=1}^n {X_{i;t}^{[-\overset{\bullet}{b}^{t}_{ik}]_{+}}}_{{\vert}_{\mathcal{F}}}(\mathbf{z})\Big){\Big(X_{k:t}^{-1}}_{{\vert}_{\mathcal{F}}}(\mathbf{z})\Big)\Bigg)_{{\vert}_{Trop(x_{n+1},\ldots,x_{m})}}}{\Big(\displaystyle \prod_{i=n+1}^{2n} {X_{i;t}^{\overset{\bullet}{b}_{ik}^t}}_{{\vert}_{\mathcal{F}}}(\mathbf{z})\oplus1\Big)_{{\vert}_{Trop(y_1^{\epsilon_1},\ldots,y_n^{\epsilon_n})}}}
\end{aligned}
\end{equation}

where $\mathbf{z} := (1,\ldots, 1, {y_1}^{\epsilon_1},\ldots, {y_n}^{\epsilon_n})$. \newline

Applying induction to the cluster variables in the terms: $$\Big(\displaystyle \prod_{i=1}^n {x_{i;t}^{[-b^{t}_{ik}]_{+}}}_{{\vert}_{\mathcal{F}}}(\mathbf{x})\Big) \hspace{4mm} \text{and} \hspace{4mm} {\Big(x_{k:t}^{-1}}_{{\vert}_{\mathcal{F}}}(\mathbf{x})\Big)$$

equation (\ref{exchangerelation}) becomes:

\begin{equation} \label{inductionexchangerelation}
{x_{k;t'}}_{{\vert}_{\mathcal{F}}}(\mathbf{x}) = \frac{\Big(\displaystyle \prod_{i=1}^{m} {x_{i;t}^{{b}^{t}_{ik}}}_{{\vert}_{\mathcal{F}}}(\mathbf{x})+1\Big)\Big( \displaystyle \prod_{i=1}^n {X_{i;t}^{[-{b}^{t}_{ik}]_{+}}}_{{\vert}_{\mathcal{F}}}(\mathbf{y})\Big){\Big(X_{k:t}^{-1}}_{{\vert}_{\mathcal{F}}}(\mathbf{y})\Big)}{\Bigg(\Big(\displaystyle \prod_{i=n+1}^{m} {x_{i;t}^{{b}_{ik}^t}}_{{\vert}_{\mathcal{F}}}(\mathbf{x})\oplus1\Big)\Big( \displaystyle \prod_{i=1}^n {X_{i;t}^{[-{b}^{t}_{ik}]_{+}}}_{{\vert}_{\mathcal{F}}}(\mathbf{z})\Big){\Big(X_{k:t}^{-1}}_{{\vert}_{\mathcal{F}}}(\mathbf{z})\Big)\Bigg)_{{\vert}_{Trop(x_{n+1},\ldots,x_m)}}}
\end{equation}

Case 1 then follows from Lemma \ref{ydynamics}, which states:

\begin{equation}
\displaystyle \prod_{i=1}^{2n} {x_{i;t}^{{b}^{t}_{ik}}}_{{\vert}_{\mathcal{F}}}(\mathbf{x}) = \displaystyle \prod_{i=1}^{2n} {X_{i;t}^{\overset{\bullet}{b}^{t}_{ik}}}_{{\vert}_{\mathcal{F}}}(\mathbf{y}).
\end{equation}

and 

\begin{equation}
\displaystyle \prod_{i=n+1}^{2n} {x_{i;t}^{{b}^{t}_{ik}}}_{{\vert}_{Trop(x_{n+1},\ldots,x_m)}}(\mathbf{x}) = \displaystyle \prod_{i=1}^{2n} {X_{i;t}^{\overset{\bullet}{b}^{t}_{ik}}}_{{\vert}_{Trop(x_{n+1},\ldots,x_m)}}(\mathbf{z}).
\end{equation}

As then, after multiplying the numerator and denominator of (\ref{inductionexchangerelation}) by $$\Big(\displaystyle \prod_{i=n+1}^{2n} {X_{i;t}^{\overset{\bullet}{b}_{ik}^t}}_{{\vert}_{\mathcal{F}}}(\mathbf{z})\oplus1\Big)_{{\vert}_{Trop(y_1^{\epsilon_1},\ldots,y_n^{\epsilon_n})}},$$ the new numerator and denominator equal (\ref{evaluatingprincipal}) and (\ref{specialisingprincipal}), respectively. \newline

\noindent \textbf{Case 2:} $T$ is a tagged quasi-triangulation. \newline \indent The proof in this case works analogous to Case 1. For each tagged quasi-triangulation $T$ one should consider the associated tagged triangulation $T^{\otimes}$. We can encode the exchange relations of $T$ (with respect to a multi-lamination $L$) via the extended exchange matrix $B_{\overline{T}^{\circ}}$. Moreover, by [Theorem 13.5,  \cite{fomin2012cluster}] we know that if $T$ and $T'$ are quasi-triangulations related by flipping a quasi-arc (labelled by $k$) which does not intersect a one-sided closed curve in $T$, then: $$\mu_{\tilde{k}} \circ \mu_{k}(B_{\overline{T^{\otimes}}}) = B_{\overline{T'^{\otimes}}}.$$ Similar to Definition \ref{labelledseedpattern} we may define a subgraph of $\mathbb{T}_n$ consisting of all labelled quasi-triangulations connected to some initial quasi-triangulation $t_0$ via flips of quasi-arcs not intersecting any one-sided closed curve. We denote this subgraph by $\mathbb{T}_n^{\Delta^{\otimes}}$. For any $t = (\{x_{1;t},\ldots, x_{n;t}\}, \tilde{B}_t = (b^t_ij))$ in $\mathbb{T}_n^{\Delta^{\otimes}}$ we may define the corresponding collection of $Y$-variables for $i \in \{1,\ldots, n\}$: $$ Y_{j:t} = \displaystyle \prod_{i=1}^{m} x_{i;t}^{\overline{b}^t_{ij}}.$$ where 

\[  \overline{b}^t_{ij} := \left\{
\begin{array}{ll}
         b^t_{ij} + b^t_{\tilde{i}j}, & i \in \{1,\ldots, n\} \\
        \\
        b^t_{ij}, & i \in \{n+1,\ldots, m\}.
\end{array} 
\right. \]

We may then obtain analogous formulae to Proposition \ref{ymutation} describing how $Y$-variables change under mutation. The crucial point being that Corollary \ref{ymuationcorollary} also holds in this setting. The proof then follows applying the same arguments used in Case 1.

\end{proof}

\bibliography{bases}
\bibliographystyle{plain}

\Addresses

\end{document}